\documentclass[a4paper,10pt,article]{amsart}
\usepackage{amsmath,amscd,amssymb,amsthm,verbatim,indentfirst,dsfont,mathrsfs,ipa,graphicx,textcomp,enumerate}
\usepackage[all]{xy}
\usepackage{ipa,booktabs,longtable,xypic}
\usepackage[numbers]{natbib}
\usepackage{commutative-diagrams}
\usepackage[colorlinks=true,
            linkcolor=blue,
            anchorcolor=blue,
            citecolor=blue]{hyperref}

\newtheorem{theorem}{Theorem}[section]
\newtheorem{lemma}[theorem]{Lemma}
\newtheorem{corollary}[theorem]{Corollary}
\newtheorem{proposition}[theorem]{Proposition}
\newtheorem{conjecture}[theorem]{Conjecture}

\theoremstyle{definition}
\newtheorem{definition}[theorem]{Definition}
\newtheorem{example}[theorem]{Example}

\newtheorem{question}[theorem]{Question}
\theoremstyle{remark}
\newtheorem{remark}[theorem]{Remark}

\numberwithin{equation}{section}

\newcommand{\Q}{\mathbb{Q}}
\newcommand{\Z}{\mathbb{Z}}
\newcommand{\N}{\mathbb{N}}
\newcommand{\R}{\mathbb{R}}
\newcommand{\C}{\mathbb{C}}
\newcommand{\K}{\mathcal{K}}
\newcommand{\ep}{\varepsilon}

\DeclareMathOperator{\supp}{supp}
\DeclareMathOperator{\prop}{prop}
\DeclareMathOperator{\Ind}{Ind}

\title{The Baum--Connes and the Mishchenko--Kasparov assembly maps for group extensions}
\begin{document}

\author{Jianguo Zhang}
\address{School of Mathematics and Statistics, Shaanxi Normal University, Xi’an 710119, China.}
\email{jgzhang@snnu.edu.cn}


\begin{abstract}
	In this paper, we investigate the injectivity, surjectivity and isomorphism of the Baum--Connes assembly map $e_{\ast}$ with coefficients, and the injectivity of the Mishchenko--Kasparov assembly map $\mu_{\ast}$ with coefficients for group extensions $1\rightarrow N \rightarrow \Gamma \xrightarrow{q} \Gamma/ N \rightarrow 1$. The main results are as follows.
	\begin{enumerate}
		\item Under the assumption that $e_{\ast}$ is isomorphic for $q^{-1}(F)$ for any finite subgroup $F$ of $\Gamma/N$, we prove that $e_{\ast}$ is injective, surjective and isomorphic for $\Gamma$ if they are also true for $\Gamma/N$, respectively. 
		\item Under the assumption that $e_{\ast}$ is rationally isomorphic for $N$, we verify that $\mu_{\ast}$ is rationally injective for $\Gamma$ if it is also rationally injective for $\Gamma/N$.
		\item When $\Gamma$ is an isometric semi-direct product $N\rtimes G$, we confirm that $e_{\ast}$ is injective, surjective and isomorphic for $\Gamma$ if they also hold for $G$ and $\Gamma$ satisfies three partial conjectures along $N$, respectively.
	\end{enumerate}
	As applications, we show that the strong Novikov conjecture, the surjective assembly conjecture and the Baum--Connes conjecture with coefficients are closed under direct products, central extensions of groups and extensions by finite groups. Meanwhile, we also show that the rational analytic Novikov conjecture with coefficients is preserved under extensions of finite groups. Besides, we employ these results to obtain some new examples for the rational analytic and the strong Novikov conjecture beyond the class of coarsely embeddable groups.
\end{abstract}

\pagestyle{plain}
\maketitle


\tableofcontents

\section{Introduction}

For a countable discrete group $\Gamma$ and a $C^{\ast}$-algebra $A$ equipped with a $\Gamma$-action, there exists an associated $C^{\ast}$-algebra $A\rtimes_{r} \Gamma$, called the \textit{reduced crossed product}. It is a central question to compute the $K$-theory of reduced crossed products in the areas of operator algebras and noncommutative geometry. The Baum--Connes and the Mishchenko--Kasparov assembly maps with coefficients provide two powerful routes to attack this question by the equivariant $K$-homology which is computable. More precisely, the \textit{Baum--Connes assembly map with coefficients} in $A$ for $\Gamma$ is a group homomorphism (see Definition \ref{Def-BC-assembly}):
  $$e_{\ast}: K^{\Gamma}_{\ast}(\underline{E}\Gamma; A) \rightarrow K_{\ast}(A\rtimes_{r} \Gamma),$$
and the \textit{Mishchenko--Kasparov assembly map with coefficients} in $A$ for $\Gamma$ is a group homomorphism (see Definition \ref{Def-MK-assembly}):
$$\mu_{\ast}: K^{\Gamma}_{\ast}(E\Gamma; A) \rightarrow K_{\ast}(A\rtimes_{r} \Gamma),$$
where the left-hand sides of $e_{\ast}$ and $\mu_{\ast}$ are the equivariant $K$-homology with coefficients of the classifying spaces $\underline{E}\Gamma$ for proper $\Gamma$-actions and $E\Gamma$ for proper and free $\Gamma$-actions, respectively. 

Inspired by the celebrated Atiyah--Singer index theorem, the \textit{strong Novikov conjecture} (SNC), the \textit{surjective assembly conjecture} (SAC) and the \textit{Baum--Connes conjecture} (BCC) with coefficients in $A$ for $\Gamma$ assert that the Baum--Connes assembly map $e_{\ast}$ is injective, surjective and isomorphic, respectively (see Conjecture \ref{Conj}). Moreover, the \textit{rational analytic Novikov conjecture} (RANC) with coefficients in $A$ for $\Gamma$ asserts that the Mishchenko--Kasparov assembly map $\mu_{\ast}$ is rationally injective (see Conjecture \ref{Conj-rational-analy-Novikov}). The benefits to investigate these conjectures are two-fold. Firstly, the Baum--Connes conjecture furnishes a topological algorithm for higher indices of equivariant elliptic differential operators on Riemannian manifolds which lie in the $K$-theory of reduced crossed products. Secondly, these conjectures have some significant applications to geometric topology and operator algebras. On the one hand, the strong and rational analytic Novikov conjectures imply the Novikov conjecture and the Gromov--Lawson--Rosenberg conjecture. On the other hand, the Kadison--Kaplansky conjecture follows from the surjective assembly conjecture (see \cite{BCH-1994}\cite{Survey-BCC}\cite{Rosenberg-IHES}\cite{WillettYu-Book}). 

In the last four decades, the above conjectures have been verified for a large class of groups. For examples, the strong Novikov conjecture with coefficients holds for groups acting properly and isometrically on bolic weakly geodesic metric spaces of bounded geometry (see \cite{KS-bolic}) and for groups coarsely embeddable into Hilbert spaces or Banach spaces with property (H) (see \cite{Yu2000}\cite{Kasparov-Yu-NovikovConj}). In addition, the rational analytic Novikov conjecture holds for groups acting properly and isometrically on simply connected, non-positively curved manifolds or on admissible Hilbert-Hadamard spaces (see \cite{Kasparov-equi-KK}\cite{GWY-Novikov}). Moreover, the Baum--Connes conjecture with coefficients holds for groups acting properly and isometrically on Hilbert spaces (see \cite{Higson-Kasparov}) and for hyperbolic groups (see \cite{Lafforgue-2012}\cite{Lafforgue-2002}\cite{MY-BC-hyperbolic}). Besides, there are several counterexamples for the above conjectures with certain non-trivial coefficients (see \cite{HLS-2002}). In addition, it is unknown whether $SL_{3}(\Z)$ satisfies the surjective assembly conjecture, and whether the infinite free Burnside group $B(d,n)$ satisfies the strong Novikov conjecture or the surjective assembly conjecture (see \cite{Survey-BCC}). 

A typical construction for producing new groups is via a group extension of a group $\Gamma/N$ by a group $N$:
$$1\rightarrow N \rightarrow \Gamma \xrightarrow{q} \Gamma/ N \rightarrow 1.$$
In order to enlarge the class of groups satisfying the above conjectures, we have the following natural question.
\begin{question}\label{Question}
Whether or not SNC, SAC, BCC and RANC with coefficients are closed under group extensions?
\end{question}
For the Baum--Connes conjecture with coefficients, this question has been explored by J. Chabert, S. Echterhoff and H. Oyono-Oyono who proved that BCC with coefficients holds for $\Gamma$ if both $\Gamma/N$ and $q^{-1}(F)$ satisfy BCC with coefficients for any finite subgroup $F$ of $\Gamma/N$ (see \cite{CEO-2004}\cite{Oyono-BC-extensions}\cite{MN-2006}\cite{Arano-Kubota-BCextensions}). T. Schick verified that the Baum-Connes conjecture with coefficients holds for extensions with finite quotients by certain geometric groups, containing full braid groups and fundamental groups of certain link complements in $S^3$ (see \cite{Schick-BCextensions}). More recently, R. Meyer constructed an example to show that BCC with certain coefficients is not true for $\Gamma=N\times \Gamma/N$, even when $\Gamma/N$ is finite and $N$ satisfies BCC with the same coefficients (see \cite{Meyer-BC-counterex}). However, there are just a few results of this question on the strong and the rational analytic Novikov conjectures. J. Deng proved that SNC with coefficients holds for $\Gamma$ if both $N$ and $\Gamma/N$ admit a coarse embedding into Hilbert spaces (see \cite{Deng-Novikov-extensions}) and J. Rosenberg proved that RANC holds for the direct product $\Gamma=N\times \Gamma/N$ if $N$ and $\Gamma/N$ satisfy RANC and $C^{\ast}_{r}(N)$ satisfies the K\"unneth formula (see \cite{Rosenberg-IHES}). In addition, it is unknown whether the strong Novikov conjecture and the surjective assembly conjecture with coefficients are closed even for direct products.

\subsection{Main results}

In this paper, we continued to investigate Question \ref{Question}, especially on the strong and the rational analytic Novikov conjectures as well as the surjective assembly conjecture with coefficients. Our first main result is the following theorem.

\begin{theorem}[see Theorem \ref{first-main-thm}]\label{first-main-thm-intro}
	Let $A$ be a $\Gamma$-$C^{\ast}$-algebra. If the following two conditions hold:
	\begin{enumerate}
		\item $q^{-1}(F)$ satisfies the Baum--Connes conjecture with coefficients in $A$ for any finite subgroup $F$ of $\Gamma/N$;
		\item $\Gamma/N$ satisfies SNC, SAC and BCC with coefficients in $C_0(\Gamma/N, A) \rtimes_r \Gamma$, respectively.
	\end{enumerate}
	Then $\Gamma$ satisfies SNC, SAC and BCC with coefficients in $A$, respectively.
\end{theorem}

Theorem \ref{first-main-thm-intro} not only recovers J. Chabert, S. Echterhoff and H. Oyono-Oyono's result on the Baum--Connes conjecture with coefficients for extensions in \cite{CEO-2004}, but also generalizes it to the strong Novikov conjecture and the surjective assembly conjecture.

The first condition in Theorem \ref{first-main-thm-intro} is strictly stronger than the condition stating that $N$ satisfies the Baum--Connes conjecture with coefficients in $A$ by R. Meyer's example in \cite{Meyer-BC-counterex}. However, for the rational analytic Novikov conjecture, this weaker condition is sufficient. Thus, we have the following second main result of this paper.

\begin{theorem}[see Theorem \ref{second-main-thm}]\label{second-main-thm-intro}
	Let $A$ be a $\Gamma$-$C^{\ast}$-algebra satisfying the K\"unneth formula. If the following two conditions hold:
	\begin{enumerate}
		\item $N$ satisfies the rational Baum--Connes conjecture with coefficients in $A$;
		\item $\Gamma/N$ satisfies the rational analytic Novikov conjecture with coefficients in $C_0(\Gamma/N, A) \rtimes_r \Gamma$.
	\end{enumerate}
	Then $\Gamma$ satisfies the rational analytic Novikov conjecture with coefficients in $A$.
\end{theorem}

In Theorem \ref{first-main-thm-intro} and Theorem \ref{second-main-thm-intro}, we always require the condition stating that the (rational) Baum--Connes conjecture with coefficients holds, even for the strong Novikov conjecture and the surjective assembly conjecture. Next, we weaken this condition for a class of group extensions defined below.
\begin{definition}[see Definition \ref{Def-semi-pro}]\label{Def-semi-pro-intro}
	Let $G$ and $N$ be two countable discrete groups. Equipped $N$ with a $G$-action. If there exists a left $N$-invariant proper metric on $N$ such that the $G$-action on $N$ is isometric, then we call $N\rtimes G$ an \textit{isometric semi-direct product}.
\end{definition}

For an isometric semi-direct product $N\rtimes G$ and an $(N\rtimes G)$-$C^{\ast}$-algebra $A$, we introduce the following partial Baum--Connes assembly map (see Proposition \ref{Prop-Condition2}):
\begin{equation*}
	e_{\ast}: \lim_{k\rightarrow \infty}K_{\ast}\left( C^{\ast}_{L, P_{k}(N)}(P_{k}(N)\times G, A)^{N\rtimes G} \right) \rightarrow \lim_{k\rightarrow \infty}K_{\ast}\left( C^{\ast}(P_{k}(N)\times G, A)^{N\rtimes G} \right).
\end{equation*}
The left-hand side of $e_{\ast}$ is the $K$-theory of the equivariant localization algebra along the Rips complex $P_{k}(N)$ with coefficients in $A$ of $P_{k}(N)\times G$ (see Definition \ref{Def-loc-along}) which can be regarded as a partial equivariant $K$-homology, and the right-hand side is the $K$-theory of the equivariant Roe algebra with coefficients in $A$ of $P_{k}(N)\times G$ (see Definition \ref{Def-eqRoecoe}) which is isomorphic to the $K$-theory of the reduced crossed product $A\rtimes_{r} (N\rtimes G)$. Similarly, we can define a partial Mishchenko--Kasparov assembly map $\mu_{\ast}$ by replacing the Rips complex $P_{k}(N)$ by the Milnor--Rips complex $\widetilde{P_{k, m}}(N)$ as above.

Then the \textit{strong Novikov conjecture along $N$}, the \textit{surjective assembly conjecture along $N$} and the \textit{Baum--Connes conjecture along $N$} with coefficients in $A$ for $N\rtimes G$ assert that the partial Baum--Connes assembly map $e_{\ast}$ is injective, surjective and isomorphic, respectively (see Conjecture \ref{partial-conj}). In addition, the \textit{rational analytic Novikov conjecture along $N$} with coefficients in $A$ for $N\rtimes G$ asserts that the partial Mishchenko--Kasparov assembly map $\mu_{\ast}$ is rationally injective.

Now, we can state the third main result of this paper as below.

\begin{theorem}[see Theorem \ref{main-THM} and Theorem \ref{main-THM-rational}]\label{third-main-thm-intro}
	Let $N\rtimes G$ be an isometric semi-direct product and $A$ be an $(N\rtimes G)$-$C^{\ast}$-algebra. If
	\begin{enumerate}
		\item there exists a non-negative sequence $\{k_{i}\}_{i\in \N}$ with $\lim_{i\rightarrow \infty}k_{i}=\infty$ such that $G$ satisfies SNC, SAC, BCC and RANC with coefficients in $C^{\ast}_{L}(P_{k_{i}}(N), A)^{N}$ for all $i\in \N$, respectively;
		\item $N\rtimes G$ satisfies SNC along $N$, SAC along $N$, BCC along $N$ and RANC along $N$ with coefficients in $A$, respectively.
	\end{enumerate}
	Then $N\rtimes G$ satisfies SNC, SAC, BCC and RANC with coefficients in $A$, respectively.
\end{theorem}

\subsection{Applications}

Firstly, since all a-T-menable groups and all hyperbolic groups are preserved under finite index extensions and satisfy the Baum--Connes conjecture with coefficients (see \cite{Higson-Kasparov}\cite{Lafforgue-2012}), we have the following result by Theorem \ref{first-main-thm-intro}.

\begin{theorem}[see Theorem \ref{Thm-N-special}]\label{Thm-N-special-intro}
	Let $A$ be a $\Gamma$-$C^{\ast}$-algebra. Assume that $N$ is an a-T-menable group or a hyperbolic group. If $\Gamma/N$ satisfies SNC, SAC and BCC with coefficients in $C_0(\Gamma/N, A)\rtimes_{r}\Gamma$, respectively. Then $\Gamma$ satisfies SNC, SAC and BCC with coefficients in $A$, respectively.
\end{theorem}

The following two corollaries of Theorem \ref{Thm-N-special-intro} indicate that SNC, SAC and BCC with coefficients is closed under group extensions by finite groups and under central extensions. However, it is unknown whether or not coarsely embeddable groups are preserved by central extensions (see \cite[Question 5.2]{AT-NotCoarseEmbed}, \cite{DG-2003}), although they satisfy the strong Novikov conjecture with coefficients (see \cite{Yu2000}).

\begin{corollary}[see Corollary \ref{Cor-finite-extension}]\label{Cor-finite-extension-intro}
	Let $1\rightarrow N \rightarrow \Gamma \rightarrow \Gamma/N \rightarrow 1$ be a group extension by a finite group $N$ and $A$ be a $\Gamma$-$C^{\ast}$-algebra. If $\Gamma/N$ satisfies SNC, SAC and BCC with coefficients in $C_0(\Gamma/N, A)\rtimes_{r}\Gamma$, respectively. Then $\Gamma$ satisfies SNC, SAC and BCC with coefficients in $A$, respectively.
\end{corollary}

\begin{corollary}[see Corollary \ref{Cor-central-extension}]\label{Cor-central-extension-intro}
	Let $1\rightarrow N \rightarrow \Gamma \rightarrow \Gamma/N \rightarrow 1$ be a central extension of groups and $A$ be a $\Gamma$-$C^{\ast}$-algebra. If $\Gamma/N$ satisfies SNC, SAC and BCC with coefficients in $C_0(\Gamma/N, A)\rtimes_{r}\Gamma$, respectively. Then $\Gamma$ satisfies SNC, SAC and BCC with coefficients in $A$, respectively.
\end{corollary}

Secondly, as an immediate application of Theorem \ref{second-main-thm-intro} to group extensions of finite groups, we obtain the following theorem.

\begin{theorem}[see Theorem \ref{Thm-rational-Novikov-finite extension}]\label{Thm-rational-Novikov-finite extension-intro}
	Let $1\rightarrow N \rightarrow \Gamma \rightarrow G \rightarrow 1$ be a group extension of a finite group $G$ and $A$ be a $\Gamma$-$C^{\ast}$-algebra satisfying the K\"unneth formula. If $N$ satisfies the rational Baum--Connes conjecture with coefficients in $A$, then $\Gamma$ satisfies the rational analytic Novikov conjecture with coefficients in $A$.
\end{theorem}

We employ Theorem \ref{Thm-rational-Novikov-finite extension-intro} to show that Meyer's example in \cite{Meyer-BC-counterex} satisfies the rational analytic Novikov conjecture with coefficients (see Example \ref{EX-Meyer}). When $A=\mathbb{C}$, we have the following corollary.
\begin{corollary}[see Theorem \ref{Cor-rational-Novikov-finite extension}]\label{Cor-rational-Novikov-finite extension-intro}
	Let $\Gamma$ be an extension of a finite group $G$ by a group $N$. If $N$ satisfies the rational Baum--Connes conjecture, then $\Gamma$ satisfies the rational analytic Novikov conjecture.
\end{corollary}

Lastly, we illustrate an application of Theorem \ref{third-main-thm-intro} to direct products of groups which is a starting point of this paper. For direct products, the partial Baum--Connes and the partial Mishchenko--Kasparov assembly maps with coefficients can be translated to the Baum--Connes and the Mishchenko--Kasparov assembly maps with different coefficients (see Lemma \ref{Lem-for-products}). Thus, we have the following consequences by Theorem \ref{third-main-thm-intro}.

\begin{theorem}[see Theorem \ref{Thm-direct-prod} and Theorem \ref{Thm-direct-prod-rational}]\label{Thm-direct-prod-intro}
	Let $N$ and $G$ be two countable discrete groups and $A$ be an $(N\times G)$-$C^{\ast}$-algebra. If
	\begin{enumerate}
		\item there exists a non-negative sequence $\{k_{i}\}_{i\in \N}$ with $\lim_{i\rightarrow \infty}k_{i}=\infty$ such that $G$ satisfies SNC, SAC, BCC and RANC with coefficients in $C^{\ast}_{L}(P_{k_{i}}(N), A)^{N}$ for all $i\in \N$, respectively;
		\item $N$ satisfies SNC, SAC, BCC and RANC with coefficients in $A\rtimes_{r} G$, respectively.
	\end{enumerate}
	Then $N\times G$ satisfies SNC, SAC, BCC and RANC with coefficients in $A$, respectively.
\end{theorem}

The first condition of Theorem \ref{Thm-direct-prod-intro} is weaker than the first condition of Theorem \ref{first-main-thm-intro} for the case of direct products (see Remark \ref{remark-new}). Thus, this theorem is new even for the Baum--Connes conjecture with coefficients. The following two corollaries of Theorem \ref{Thm-direct-prod-intro} are straightforward.

\begin{corollary}[see Corollary \ref{Cor-direct-prod} and Corollary \ref{Cor-direct-prod-rational}]\label{Cor-direct-prod-intro}
	If $G$ and $N$ satisfies SNC, SAC, BCC and RANC with coefficients, respectively. Then $N\times G$ satisfies SNC, SAC, BCC and RANC with coefficients, respectively.
\end{corollary}

\begin{corollary}[see Corollary \ref{Cor-direct-prod-C} and Corollary \ref{Cor-direct-prod-rational-C}]\label{Cor-direct-prod-C-intro}
	Let $A$ be a $C^{\ast}$-algebra with the trivial $(N\times G)$-action (in particular, $A=\C$). If
	$G$ and $N$ satisfy SNC, SAC, BCC and RANC with coefficients in any $C^{\ast}$-algebra with the trivial action, respectively. Then $N\times G$ satisfies SNC, SAC, BCC and RANC with coefficients in $A$, respectively.
\end{corollary}

G. Arzhantseva and R. Tessera constructed a group that can not be coarsely embedded into a Hilbert space and this group satisfies the Baum-Connes conjecture with coefficients in \cite{AT-NotCoarseEmbed}. Based on this group, we will produce more examples satisfying the (rational analytic) strong Novikov conjecture without coarse embeddability in Section \ref{Sec-6} by using the above results.

\subsection{Strategies of the proofs}

Let us briefly summarize the strategies to proving our main results. The crucial ingredients of the proofs are three new variations of equivariant localization algebras, Mayer--Vietoris six-term exact sequence, the imprimitivity theorem and the quantitative $K$-theory. 

The proofs of Theorem \ref{first-main-thm-intro} and Theorem \ref{second-main-thm-intro} are similar and both of them are inspired by the method of J. Chabert, S. Echterhoff and H. Oyono-Oyono developed in \cite{CEO-2004} and \cite{Oyono-BC-extensions}. Here we just summarize the proof of Theorem \ref{first-main-thm-intro} as the following three steps.
\begin{enumerate}
	\item [\textbf{step} $1_{a}$:] For a direct product $N\times G$, by introducing a notion of \textit{equivariant localization algebra along $P_{k}(N)$ with coefficients} (see Definition \ref{Def-loc-along}), the Baum--Connes assembly map with coefficients for $N\times G$ can be broken into a localized Baum--Connes assembly map with coefficients for $G$ and a Baum--Connes assembly map with coefficients for $N$. Namely, we have the following commutative diagram (see Subsection \ref{Subsection-prod}):
    $$\xymatrix{
    	K_{\ast}\left(C^{\ast}_{L}(P_{k}(N)\times P_{l}(G), A)^{N\times G}\right)\:\:\: \ar[d]_{\text{localized BC for $G$}} \ar[r]^{\text{\qquad\:\:\:\:\:\:BC for $N\times G$}} & \:\:\:K_{\ast}\left(A\rtimes_{r} (N\times G)\right). \\
    	K_{\ast}\left(C^{\ast}_{L, P_{k}(N)}(P_{k}(N)\times P_{l}(G), A)^{N\times G}\right)\ar[ur]_{\text{\:\:\:\:\:\:\:\:\:\: BC for $N$}} & 
   	}$$
	\item [\textbf{step} $2_{a}$:] Under the assumption that $F\times G$ satisfies the Baum--Connes conjecture with coefficients for any finite subgroup $F$ of $N$, we prove that the localized Baum--Connes assembly map with coefficients for $G$ is an isomorphism by a Mayer--Vietoris argument (see Proposition \ref{Lem-Key}). Consequently, Theorem \ref{first-main-thm-intro} is proved for the case of direct products (see Theorem \ref{main-thm1}).
	\item [\textbf{step} $3_{a}$:] By using the imprimitivity theorem (see Proposition \ref{Prop-Green-impri}), we can transform Theorem \ref{first-main-thm-intro} for group extensions to the case of direct products (see Subsection \ref{Subsection-Green}). Thus, we complete the proof of Theorem \ref{first-main-thm-intro} by the above step.  
\end{enumerate}

The strategy of the proof of Theorem \ref{third-main-thm-intro} is very different from the proofs of Theorem \ref{first-main-thm-intro} and Theorem \ref{second-main-thm-intro}, and we sum it up as follows.

\begin{enumerate}
	\item [\textbf{step} $1_{c}$:] First of all, we construct a model of $\underline{E}\Gamma$ for the isometric semi-direct product $\Gamma=N\rtimes G$ by the product $P_{k}(N)\times P_{k}(G)$ of Rips complexes (see Corollary \ref{Cor-red-conj}). 
	\item [\textbf{step} $2_{c}$:] Next, we introduce a notion of \textit{two-parametric equivariant localization algebras with coefficients} (see Definition \ref{Def-two-localg}) and we employ it to establish an isomorphism for the left-hand sides of the Baum--Connes assembly maps with coefficients of $N\rtimes G$ and $G$ (see Proposition \ref{Prop-loc-twoloc} and Proposition \ref{Prop-oneloc-twoloc}). Based on this, we can reduce the Baum--Connes assembly map with coefficients for $N\rtimes G$ to a composition of the Baum--Connes assembly map with coefficients for $G$ and a $G$-equivariant Baum--Connes assembly map with coefficients for $N$ (see Proposition \ref{Prop-reduction-Conj}). 
	$$\xymatrix{
	  K_{\ast}\left(C^{\ast}_{L}(P_{l}(G)\times P_{k}(N), A)^{N\rtimes G}\right) \ar[d]_{\cong} \ar[dr]^{\text{\quad BC for $N\rtimes G$}} & \\
	  K_{\ast}\left(C^{\ast}_{L}(P_{l}(G), C^{\ast}_{L}(P_{k}(N), A)^{N})^{G}\right) \ar[r] \ar[d]_{\text{BC for $G$}} & K_{\ast}\left(A\rtimes_{r} (N\rtimes G)\right). \\
	  K_{\ast}\left(C^{\ast}(P_{l}(G), C^{\ast}_{L}(P_{k}(N), A)^{N})^{G}\right) \ar[ur]_{\text{\qquad\quad\quad $G$-equivariant BC for $N$}} &
	}$$
	\item [\textbf{step} $3_{c}$:] In the end, we introduce a notion of \textit{uniformly controlled equivariant localization algebra along $P_{k}(N)$ with coefficients} (see Definition \ref{Def-uniloc-along}) and we prove that its $K$-theory is isomorphic to the left-hand side of the $G$-equivariant Baum--Connes assembly map with coefficients for $N$ (see Lemma \ref{Lem-trans-loc}). Besides, we verify that two equivariant localization algebras along $P_{k}(N)$ with coefficients are isomorphic on the $K$-theory by using the quantitative $K$-theory and Mayer--Vietoris six-term exact sequence (see Proposition \ref{Prop-localong-uniforalong}). Base on these facts, we complete the proof of Theorem \ref{third-main-thm-intro}.
\end{enumerate}

\subsection{Organization}

The paper is organized as follows. In Section \ref{Sec-1}, we recall the notions of equivariant Roe algebras with coefficients and equivariant localization algebras with coefficients. Then, we use them to formulate the Baum--Connes and the Mishchenko--Kasparov assembly maps. In Section \ref{Sec-2}, we introduce a notion of equivariant localization algebra along one direction with coefficients for direct products and prove Theorem \ref{first-main-thm-intro}. In Section \ref{Sec-3}, we verify Theorem \ref{second-main-thm-intro}. In Section \ref{Sec-4}, we introduce a notion of two-parametric equivariant localization algebras and analyze the left-hand side of the Baum--Connes conjecture with coefficients for isometric semi-direct products. In Section \ref{Sec-5}, we complete the proof of Theorem \ref{third-main-thm-intro}. Finally in Section \ref{Sec-6}, we demonstrate some applications of our main results and show some new examples for the (rational analytic) strong Novikov conjecture with coefficients.

\section{Preliminaries} \label{Sec-1}
  In this section, we will recall the concepts of equivariant Roe algebras, equivariant localization algebras and the formulations of the Baum--Connes and the Mishchenko-Kasparov assembly maps.

\subsection{Equivariant Roe algebras with coefficients}\label{Sec-1-Roe}
Let $\Gamma$ be a countable discrete group and $X$ be a proper metric space (properness means that any bounded closed subset is compact). An isometric action of $\Gamma$ on $X$ is called to be \textit{proper} if the set $\{\gamma\in \Gamma: \gamma K\cap K \neq \emptyset\}$ is finite for any compact subset $K$ in $X$, and the action is called to be \textit{co-compact} if the quotient space $X/\Gamma$ is compact.

\begin{definition}\label{Def-Gamma-C}
	Let $\Gamma$ be a countable discrete group. 
	\begin{enumerate}
		\item A proper metric space $X$ is called to be a \textit{$\Gamma$-space}, if there exists a proper, co-compact and isometric left $\Gamma$-action on $X$.
		\item A $C^{\ast}$-algebra $A$ is called to be a \textit{$\Gamma$-$C^{\ast}$-algebra}, if there exists a $\Gamma$-action on $A$ by $\ast$-automorphisms. 
	\end{enumerate}
\end{definition}

Let $X$ be a $\Gamma$-space and $Z_X$ be a $\Gamma$-invariant countable dense subset in $X$. Let $A$ be a $\Gamma$-$C^{\ast}$-algebra equipped with the action $\alpha$ and $H$ be a separable Hilbert space endowed with the trivial $\Gamma$-action.

Consider a right Hilbert $A$-module
$$_{X}E_A=\ell^2(Z_X) \otimes A \otimes H \otimes \ell^2(\Gamma),$$
equipped with the $A$-action by
$$(\xi \otimes a \otimes h \otimes \eta)\cdot a'=\xi \otimes aa' \otimes h \otimes \eta,$$
and the $A$-valued inner product by
$$\langle \xi_1 \otimes a_1 \otimes h_1 \otimes \eta_1, \xi_2 \otimes a_2 \otimes h_2 \otimes \eta_2 \rangle=\langle \xi_1, \xi_2 \rangle \langle h_1, h_2 \rangle \langle \eta_1, \eta_2 \rangle a_1^{\ast}a_2.$$
Define a $\Gamma$-action on $_{X}E_A$ by 
$$U_{\gamma}: {_{X}E_A} \rightarrow {_{X}E_A}, \:\: \delta_{x}\otimes a\otimes h\otimes \delta_{r'} \mapsto \delta_{\gamma x}\otimes \alpha_{\gamma}(a)\otimes h\otimes \delta_{\gamma'\gamma^{-1}},$$
for any $\gamma\in \Gamma$. In addition, any bounded Borel function $f$ on $X$ can be regarded as an adjointable operator on $_{X}E_A$ by
$$f\cdot (\xi\otimes a\otimes h\otimes \eta)=f\xi\otimes a\otimes h\otimes \eta.$$ 


We denote by $\mathcal{L}(_{X}E_A)$ the algebra of all adjointable operators on $_{X}E_A$ and by $\mathcal{K}(_{X}E_A)$ the algebra of all compact operators on $_{X}E_A$. Then we have an isomorphism $\mathcal{K}(_{X}E_A)\cong \mathcal{K}(\ell^2(Z_X))\otimes \mathcal{K}(H)\otimes \mathcal{K}(\ell^2(\Gamma))\otimes A$. 

\begin{definition}\label{Def-prop}
	Let $\chi_{S}$ be the characteristic function on a Borel subset $S$ of $X$. For an operator $T\in \mathcal{L}(_{X}E_A)$. 
	\begin{enumerate}
		\item The \textit{support} of $T$, denoted by $\supp(T)$, is defined to be
		$$\{(x,x')\in X\times X: \chi_{V} T \chi_{U}\neq 0 \:\: \text{for any open neighborhoods $U$ of $x$, $V$ of $x'$}\}.$$
		\item The \textit{propagation} of $T$, denoted by $\text{prop}(T)$, is defined to be
		$$\text{prop}(T)=\sup\{d(x,x'): (x,x')\in \supp(T)\}.$$  
		\item $T$ is called to be \textit{locally compact}, if $\chi_K T, T\chi_K\in \mathcal{K}(_{X}E_A)$ for any compact subset $K\subseteq X$.
		\item $T$ is called to be \textit{$\Gamma$-invariant}, if $U_{\gamma} T U_{\gamma^{-1}}=T$ for any $\gamma\in \Gamma$.
	\end{enumerate}
\end{definition}

\begin{remark}
	For any $T\in \mathcal{L}(_{X}E_A)$, we can represent $T$ as a matrix $(T_{x,x'})_{x,x'\in Z_X}$, where $T_{x,x'}=\chi_{x} T \chi_{x'}$. Then $T$ has finite propagation if and only if there exists $R>0$ such that $T_{x,x'}=0$ for all $x,x'\in Z_X$ with $d(x,x')>R$. In addition, if $T_{x,x'}=a_{x,x'}\otimes K_{x,x'} \otimes Q_{x,x'}$ for some $a_{x,x'}\in A$, $K_{x,x'}\in \mathcal{K}(H)$ and $Q_{x,x'}\in \mathcal{K}(\ell^2(\Gamma))$. Then $T$ is $\Gamma$-invariant if and only if $a_{\gamma x,\gamma x'}=\alpha_{\gamma}(a_{x, x'})$, $K_{\gamma x,\gamma x'}=K_{x, x'}$ and $Q_{\gamma x,\gamma x'}=W_{\gamma} Q_{x, x'} W_{\gamma^{-1}}$ for any $\gamma\in \Gamma$, where $W_{\gamma}: \ell^2(\Gamma) \rightarrow \ell^2(\Gamma),\: \delta_{\gamma'}\mapsto \delta_{\gamma'\gamma^{-1}}$.
\end{remark}

\begin{definition}\label{Def-eqRoecoe}
	Let $\Gamma$ be a countable discrete group, $X$ be a $\Gamma$-space and $A$ be a $\Gamma$-$C^{\ast}$-algebra. Define $\mathbb{C}[X, A]^{\Gamma}$ to be the $\ast$-algebra consisting of all $\Gamma$-invariant, locally compact operators with finite propagation.
	The \textit{equivariant Roe algebra with coefficients} in $A$ of $X$, denoted by $C^{\ast}(X, A)^{\Gamma}$, is defined to be the norm closure of $\mathbb{C}[X, A]^{\Gamma}$ in $\mathcal{L}(_{X}E_A)$. 
\end{definition}

The Roe algebra was first introduced by J. Roe in \cite{Roe1988}, whose $K$-theory provides a receptacle for higher indices of elliptic differential operators on open manifolds. In the $\ast$-isomorphic sense, the definition of equivariant Roe algebra with coefficients does not depend on the choice of $\Gamma$-invariant dense subsets $Z_X$ (please see \cite[Section 5.2]{WillettYu-Book}).

If $A$ is covariantly represented on a Hilbert space $H_A$. Then $\mathbb{C}[X, A]^{\Gamma}$ can be represented on a Hilbert space
  $$_{X}H_{A}=\ell^2(Z_X)\otimes H_A\otimes H\otimes \ell^2(\Gamma).$$ 
Thus, we have the following characterization for equivariant Roe algebras.  
\begin{lemma}\label{Lem-Roe-repspace}
	The equivariant Roe algebra $C^{\ast}(X, A)^{\Gamma}$ is isomorphic to the norm closure of $\mathbb{C}[X, A]^{\Gamma}$ in $\mathcal{L}(_{X}H_{A})$.
\end{lemma}
\begin{proof}
	Let $_{X}E_{A} \otimes_{A} H_A$ be the interior tensor product of the Hilbert $A$-module $_{X}E_{A}$ and the Hilbert space $H_{A}$. Define
	$$V: {_{X}}E_{A} \otimes_{A} H_A \rightarrow {_{X}}H_{A},\:(\xi\otimes a\otimes h\otimes \eta)\otimes \zeta \mapsto \xi\otimes a\zeta \otimes h \otimes \eta.$$
	Then $V$ is an isometric surjective linear map, and hence is a unitary. Thus, the map defined by $T \mapsto V(T\otimes I_{H_A})V^{\ast}$ is an isometric homomorphism from $\mathcal{L}(_{X}E_{A})$ to $\mathcal{L}(_{X}H_{A})$ and induces a $\ast$-isomorphism from $C^{\ast}(X, A)^{\Gamma}$ to the norm closure of $\mathbb{C}[X, A]^{\Gamma}$ in $\mathcal{L}(_{X}H_{A})$.
\end{proof}

Now, we consider the functoriality property of equivariant Roe algebras.

\begin{definition}\label{Def-equi-coarsemap}
	Let $X$ and $Y$ be two $\Gamma$-spaces. A map $f: X\rightarrow Y$ is called to be a \textit{coarse $\Gamma$-map}, if 
	\begin{itemize}
		\item $f(\gamma x)=\gamma f(x)$ for all $x\in X$ and $\gamma\in \Gamma$;
		\item $f^{-1}(K)$ is pre-compact for any compact subset $K\subseteq Y$;
		\item for any $R>0$, there exists $S>0$ such that $d(f(x), f(x'))\leq S$ for any $x,x'\in X$ with $d(x,x')\leq R$.
	\end{itemize}
	
	$X$ and $Y$ are called to be \textit{coarsely $\Gamma$-equivalent}, if there exist two coarse $\Gamma$-maps $f: X\rightarrow Y$, $g: Y\rightarrow X$ and a constant $C>0$ such that
	$$\max\{d(gf(x), x), d(fg(y), y)\}\leq C,$$
	for all $x\in X$ and $y\in Y$.  
\end{definition}

\begin{lemma}\label{Lem-coveringmap}(\cite[Theorem 5.2.6]{WillettYu-Book})
	Let $X$, $Y$ be two $\Gamma$-spaces and $A$ be a $\Gamma$-$C^{\ast}$-algebra. Assume that $f: X\rightarrow Y$ be a coarse $\Gamma$-map. Then $f$ leads to a $\ast$-homomorphism from $C^{\ast}(X, A)^{\Gamma}$ to $C^{\ast}(Y, A)^{\Gamma}$, and hence induces a homomorphism
	$$f_{\ast}: K_{\ast}\left(C^{\ast}(X, A)^{\Gamma}\right) \rightarrow K_{\ast}\left(C^{\ast}(Y, A)^{\Gamma}\right).$$
\end{lemma}

\begin{corollary}\label{Cor-coarseequi-K}
	If $X$ is coarsely $\Gamma$-equivalent to $Y$, then $C^{\ast}(X, A)^{\Gamma}$ is naturally isomorphic to $C^{\ast}(Y, A)^{\Gamma}$.
\end{corollary}

In the end of this subsection, we discuss the relationship between the equivariant Roe algebras with coefficients and the reduced crossed product. Let $X$ be a $\Gamma$-space and $A$ be a $\Gamma$-$C^{\ast}$-algebra. Since the action of $\Gamma$ on $X$ is co-compact, the map defined by $\gamma \mapsto \gamma x_0$ is a coarsely $\Gamma$-equivalent map from $\Gamma$ to $X$, where $x_0\in X$ is a fixed point. Hence, by Corollary \ref{Cor-coarseequi-K}, $C^{\ast}(X, A)^{\Gamma}$ is naturally $\ast$-isomorphic to $C^{\ast}(\Gamma, A)^{\Gamma}$ which is naturally $\ast$-isomorphic to $(A \rtimes_r \Gamma) \otimes \mathcal{K}(H)$. Therefore, we have the following lemma.

\begin{lemma}(\cite[Lemma B.12]{GWY-2024} or \cite[Theorem 5.3.2]{WillettYu-Book})\label{Lem-eqRoe-crossprod}
	Let $X$ be a $\Gamma$-space and $A$ be a $\Gamma$-$C^{\ast}$-algebra. Then $C^{\ast}(X, A)^{\Gamma}$ is naturally $\ast$-isomorphic to $(A \rtimes_r \Gamma) \otimes \mathcal{K}(H)$.
\end{lemma} 

\subsection{Equivariant localization algebras with coefficients}\label{Sec-1-loc}
\begin{definition}\label{Def-localg}
	Let $X$ be a $\Gamma$-space and $A$ be a $\Gamma$-$C^{\ast}$-algebra. Define $\C_{L}[X, A]^{\Gamma}$ to be the $\ast$-algebra consisting of all bounded and uniformly continuous functions $u: [0,\infty) \rightarrow C^{\ast}(X, A)^{\Gamma}$ such that
	$$\lim_{t\rightarrow \infty}\prop(u(t))=0.$$
	The \textit{equivariant localization algebra with coefficients} in $A$ of $X$, denoted by $C^{\ast}_{L}(X, A)^{\Gamma}$, is defined to be the norm closure of $\C_{L}[X, A]^{\Gamma}$ in $\mathcal{L}\left(L^2([0, \infty))\otimes {_{X}H_{A}}\right)$.
\end{definition}

The localization algebra was first introduced by G. Yu in \cite{Yu-Localizationalg}, whose $K$-theory provides an analytic model for $K$-homology (please see \cite{QiaoRoe}\cite[Chapter 6]{WillettYu-Book}).

We have the following functoriality and homological properties for $K$-theory of localization algebras, please refer to \cite{Yu-Localizationalg} or \cite[Chapter 6]{WillettYu-Book} for the proofs.
\begin{lemma}\label{Lem-conticovmap}
	Let $X$, $Y$ be two $\Gamma$-spaces and $A$ be a $\Gamma$-$C^{\ast}$-algebra. Assume that $f: X \rightarrow Y$ be a uniformly continuous and coarse $\Gamma$-map. Then $f$ induces a homomorphism
	$$f_{\ast}: K_{\ast}\left(C^{\ast}_{L}(X,A)^{\Gamma}\right) \rightarrow K_{\ast}\left(C^{\ast}_{L}(Y, A)^{\Gamma}\right).$$
\end{lemma}

\begin{definition}\label{Def-strLip}
	Let $X$, $Y$ be as above. Two equivariant Lipschitz maps $f,g: X\rightarrow Y$ are called to be \textit{strongly Lipschitz $\Gamma$-homotopic}, if there exists a continuous coarse $\Gamma$-map $F:[0,1]\times X \rightarrow Y$ satisfying that
  \begin{itemize}
	\item there exists $C>0$ such that $d(F(t, x), F(t, x'))\leq C d(x,x')$ for all $x,x'\in X$ and $t\in [0,1]$;
	\item for any $\varepsilon>0$, there exists $\delta>0$ such that $d(F(t,x), F(t',x))<\varepsilon$ for all $x\in X$ and $|t-t'|<\delta$;
	\item $F(0,x)=f(x)$ and $F(1,x)=g(x)$ for all $x\in X$.
  \end{itemize} 
		
	$X$ is call to be \textit{strongly Lipschitz $\Gamma$-homotopy equivalent} to $Y$, if there exist two equivariant Lipschitz maps $f: X\rightarrow Y$ and $h: Y\rightarrow X$ such that $hf$ and $fh$ are strongly Lipschitz $\Gamma$-homotopic to $id_{X}$ and $id_{Y}$, respectively.  
\end{definition}

\begin{lemma}\label{Lem-Liphtp}
	Let $X$, $Y$ and $A$ be as above. If $X$ is strongly Lipschitz $\Gamma$-homotopy equivalent to $Y$, then $K_{\ast}\left(C^{\ast}_{L}(X, A)^{\Gamma}\right)$ is naturally isomorphic to $K_{\ast}\left(C^{\ast}_{L}(Y, A)^{\Gamma}\right)$.
\end{lemma}
	
\begin{definition}\label{Def-unicover}
	Let $X_1$ and $X_2$ be two $\Gamma$-invariant subsets of $X$. Let $B_{\varepsilon}(Y)$ be the $\varepsilon$-neighborhood of $Y\subseteq X$. The pair $\{X_1, X_2\}$ is called to be a \textit{uniformly excisive $\Gamma$-cover} of $X$, if 
  \begin{itemize}
	\item $X=X_1\cup X_2$;
	\item for any $\varepsilon_t\rightarrow 0\:(t\rightarrow \infty)$, there exists $\delta_t\rightarrow 0\:(t\rightarrow \infty)$ such that 
	$$B_{\varepsilon_t}(X_1) \cap B_{\varepsilon_t}(X_2) \subseteq B_{\delta_t}(X_1\cap X_2),\:\:t>0;$$
	\item there exists $\epsilon>0$ such that $B_{\epsilon}(X_1)$, $B_{\epsilon}(X_2)$ and $B_{\epsilon}(X_1\cap X_2)$ are all $\Gamma$-invariant and strongly Lipschitz $\Gamma$-homotopy equivalent to $X_1$, $X_2$ and $X_1\cap X_2$, respectively.
  \end{itemize}
\end{definition}

\begin{lemma}\label{Lem-equiMV}
	If $\{X_1, X_2\}$ is a uniformly excisive $\Gamma$-cover of $X$, then the following Mayer--Vietoris six-term exact sequence holds
	$$\small\xymatrix{
		K_0(C^{\ast}_{L}(X_1\cap X_2, A)^{\Gamma}) \ar[r] & 
		K_0(C^{\ast}_{L}(X_1, A)^{\Gamma})\oplus K_0(C^{\ast}_{L}(X_2, A)^{\Gamma}) \ar[r] & 
		K_0(C^{\ast}_{L}(X, A)^{\Gamma})\ar[d] \\
		K_1(C^{\ast}_{L}(X, A)^{\Gamma})\ar[u] &
		K_1(C^{\ast}_{L}(X_1, A)^{\Gamma})\oplus K_1(C^{\ast}_{L}(X_2, A)^{\Gamma})\ar[l] &
		K_1(C^{\ast}_{L}(X_1\cap X_2, A)^{\Gamma}).\ar[l] 
	}$$
\end{lemma}

\subsection{The Baum--Connes assembly map}\label{Sec-1-conj}
\begin{definition}\label{Def-length}
	Let $\Gamma$ be a countable discrete group. A \textit{length function} on $\Gamma$ is a function $|\cdot|: \Gamma \rightarrow \mathbb{N}$ such that
	\begin{itemize}
		\item $|\gamma|=0$ if and only if $\gamma$ is the identity element;
		\item $|\gamma^{-1}|=|\gamma|$ for any $\gamma\in \Gamma$;
		\item $|\gamma_1 \gamma_2|\leq |\gamma_1|+|\gamma_2|$ for any $\gamma_1, \gamma_2\in \Gamma$.
	\end{itemize}
	A length function $|\cdot|$ on $\Gamma$ is called to be \textit{proper}, if $\{\gamma: |\gamma|\leq R\}$ is finite for any $R\geq 0$.
\end{definition}

There always exists a proper length function on any countable discrete group (see, for example, \cite[Proposition 1.2.2]{Nowak-Yu-Book}). Moreover, every length function $|\cdot|$ on $\Gamma$ induces a left $\Gamma$-invariant metric defined by $d(\gamma_1, \gamma_2)=|\gamma^{-1}_1 \gamma_2|$. 

Let $\Gamma$ be a countable discrete group with a length function $|\cdot|$. For each $k\geq 0$, the \textit{Rips complex} of $\Gamma$ at scale $k$ denoted by $P_{k}(\Gamma)$, is a simplicial complex whose set of vertices is $\Gamma$ and a subset $\{\gamma_0, \cdots, \gamma_n\}$ spans an $n$-simplex if and only if $|\gamma^{-1}_i \gamma_j|\leq k$ for any $i,j=0, \cdots, n$. 

Let $d_{S_k}$ be a path metric on $P_{k}(\Gamma)$ whose restriction to each simplex is the standard spherical metric on the unit sphere by mapping $\sum_{i=0}^n t_i\gamma_i$ to 
$$\left(t_0 \Big/ \sqrt{\sum t^2_i}, \cdots, t_n \Big/\sqrt{\sum t^2_i}\right).$$ 
Then we define a metric $d_{P_{k}}$ on $P_{k}(\Gamma)$ to be 
$$d_{P_{k}}(x,x')=\inf\left\{\sum_{i=0}^n d_{S_k}(\gamma_i, \gamma'_i)+\sum_{i=0}^{n-1} |\gamma^{-1}_{i+1}\gamma'_i|\right\},$$
for all $x,x'\in P_{k}(\Gamma)$, where the infimum is taken over all sequences of the form $x=\gamma_0, \gamma'_0, \gamma_1, \gamma'_1, \cdots, \gamma_n, \gamma'_n=y$ with $\gamma_1, \cdots, \gamma_n, \gamma'_0,\cdots, \gamma'_{n-1}\in \Gamma$. Then $P_{k}(\Gamma)$ is a $\Gamma$-space equipped with a left $\Gamma$-action given by 
$$\gamma: \sum_{i=0}^n t_i \gamma_i \mapsto \sum_{i=0}^n t_i \gamma\gamma_i,$$
for any $\gamma, \gamma_0, \cdots, \gamma_n\in \Gamma$. In addition, the embedding map from $\Gamma$ to $P_{k}(\Gamma)$ induces a coarse $\Gamma$-equivalence (please see \cite[Proposition 7.2.11]{WillettYu-Book}).

For $k_1\leq k_2$, the inclusion map $i_{k_1k_2}: P_{k_{1}}(\Gamma)\rightarrow P_{k_{2}}(\Gamma)$ induces the following two homomorphisms by Lemma \ref{Lem-coveringmap} and Lemma \ref{Lem-conticovmap}:
$$i_{k_1k_2, \ast}: K_{\ast}\left(C^{\ast}(P_{k_1}(\Gamma), A)^{\Gamma}\right)\rightarrow K_{\ast}\left(C^{\ast}(P_{k_2}(\Gamma), A)^{\Gamma}\right);$$
$$i_{k_1k_2, \ast}: K_{\ast}\left(C_L^{\ast}(P_{k_1}(\Gamma), A)^{\Gamma}\right)\rightarrow K_{\ast}\left(C_L^{\ast}(P_{k_2}(\Gamma), A)^{\Gamma}\right).$$

Moreover, the evaluation at zero map:
$$e: C^{\ast}_{L}(P_{k}(\Gamma), A)^{\Gamma}\rightarrow C^{\ast}(P_{k}(\Gamma), A)^{\Gamma}, \:u\mapsto u(0),$$
induces a homomorphism:
$$e_{\ast}: K_{\ast}\left(C^{\ast}_{L}(P_{k}(\Gamma), A)^{\Gamma}\right) \rightarrow K_{\ast}\left(C^{\ast}(P_{k}(\Gamma), A)^{\Gamma}\right),$$
that satisfies $i_{k_1k_2, \ast}\circ e_{\ast}=e_{\ast} \circ i_{k_1k_2, \ast}$ for all $k_1\leq k_2$. 

\begin{definition}\label{Def-BC-assembly}
The \textit{Baum--Connes assembly map with coefficients} in $A$ for $\Gamma$ is defined by 
$$e_{\ast}: \lim_{k\rightarrow \infty} K_{\ast}\left(C^{\ast}_{L}(P_{k}(\Gamma), A)^{\Gamma}\right) \rightarrow \lim_{k\rightarrow \infty} K_{\ast}\left(C^{\ast}(P_{k}(\Gamma), A)^{\Gamma}\right) \cong K_{\ast}\left( A \rtimes_{r} \Gamma \right),$$
where the last isomorphism as above is due to Lemma \ref{Lem-eqRoe-crossprod}.
\end{definition}

\begin{remark}
	The original Baum--Connes assembly map was stated by using of the equivariant $K$-homology and the classifying space $\underline{E}\Gamma$ for proper $\Gamma$-actions (please see \cite{BCH-1994}). Since Rips complexes provide a concrete model for $\underline{E}\Gamma$ and the $K$-theory of equivariant localization algebras gives a picture for the equivariant $K$-homology, Definition \ref{Def-BC-assembly} is identical with the original statement of the Baum--Connes assembly map (please see \cite[Appendix B]{GWY-2024} for the detailed proof).
\end{remark}

\begin{conjecture}\label{Conj}
	Let $\Gamma$ be a countable discrete group and $A$ be a $\Gamma$-$C^{\ast}$-algebra. Then we have the following conjectures.
	\begin{enumerate}
		\item\label{Conj-SN} The \textit{strong Novikov conjecture} (SNC, for short) with coefficients in $A$ for $\Gamma$ asserts that $e_{\ast}$ is injective.
		\item\label{Conj-SA} The \textit{surjective assembly conjecture} (SAC, for short) with coefficients in $A$ for $\Gamma$ asserts that $e_{\ast}$ is surjective.
		\item\label{Conj-BC} The \textit{Baum--Connes conjecture} (BCC, for short) with coefficients in $A$ for $\Gamma$ asserts that $e_{\ast}$ is isomorphic.
	\end{enumerate} 
\end{conjecture}

In particular, when $A=\mathbb{C}$, the above conjectures (\ref{Conj-SN}) and (\ref{Conj-BC}) are called the \textit{strong Novikov conjecture} and \textit{Baum--Connes conjecture}, respectively. Besides, the \textit{rational Baum--Connes conjecture} (QBCC, for short) with coefficients in $A$ for $\Gamma$ asserts that $e_{\ast}$ is rationally isomorphic. Similarly, the \textit{rational strong Novikov conjecture} and the \textit{rational surjective assembly conjecture} with coefficients in $A$ for $\Gamma$ assert that $e_{\ast}$ is rationally injective and rationally surjective, respectively.

\begin{remark}
	The (rational) strong Novikov conjecture implies the Novikov conjecture on the homotopy invariance of higher signatures for closed manifolds and the Gromov--Lawson--Rosenberg conjecture concerning the existence of Riemannian metrics with positive scalar curvature on a manifold. In addition, the surjective assembly conjecture implies the Kadison--Kaplansky conjecture which asserts that there is no non-trivial idempotent in the reduced group $C^{\ast}$-algebra of any torsion-free group (cf. \cite[Chapter 10]{WillettYu-Book}).
\end{remark}

\subsection{The Mishchenko--Kasparov assembly map}

Now, let us recall the statement of the Mishchenko--Kasparov assembly map using the Milnor--Rips complexes introduced by G. Yu in \cite{Yu-BC-coarse}. 

\begin{definition}\label{Def-MilnorRips}
	Let $\Gamma$ be a countable discrete group equipped with a proper left invariant metric. For any $k \geq 0$ and $m\in \N$, the \textit{Milnor--Rips complex} of $\Gamma$ at scale $k$ and $m$, denoted by $\widetilde{P_{k,m}}(\Gamma)$, is defined to be the set of equivalence classes consisting of all infinite sequences $\langle t,\gamma \rangle=(t_0\gamma_0, t_1\gamma_1, \cdots, t_i\gamma_i, \cdots)$ satisfying that
	\begin{enumerate}
		\item $t_i \geq 0$ for each $i$, $t_i=0$ for $i>m$ and $\sum_{i\in \N} t_i=1$;
		\item $\gamma_i\in \Gamma$ for each $i$ and $\max\{d(\gamma_i,\gamma_j): t_i,t_j \neq 0\}\leq k$;
		\item $\langle t,\gamma \rangle=(t_0\gamma_0, t_1\gamma_1, \cdots, t_i\gamma_i, \cdots)$ and $\langle t',\gamma' \rangle=(t'_0\gamma'_0, t'_1\gamma'_1, \cdots, t'_i\gamma'_i, \cdots)$ are equivalent if 
		\begin{enumerate}
			\item [$(3_a)$] $t_i=t'_i$ for each $i$ and $\gamma_i=\gamma'_i$ for all $i$ with $t_i \neq 0$, or; 
			\item [$(3_b)$] there exists $i_0$ with $t_{i_0}=0$ such that $t_i=t'_i$, $\gamma_i=\gamma'_i$ for all $i<i_0$ and $t_{i+1}=t'_{i}$, $\gamma_{i+1}=\gamma'_{i}$ for all $i\geq i_0$.
		\end{enumerate}
	\end{enumerate}
\end{definition}
$\widetilde{P_{k,m}}(\Gamma)$ is an $n$-dimensional CW-complex and a set $\{\langle t,\gamma \rangle\in \widetilde{P_{k,m}}(\Gamma): t=(t_0, t_1, \cdots, t_m, 0,\cdots),\:\sum_{i\in \N} t_i=1\}$ forms a simplex $\Delta_{\gamma}$ for a fixed sequence $\gamma=(\gamma_0, \gamma_1, \cdots, \gamma_m, \cdots)$ with $\text{diam}(\{\gamma_0, \gamma_1, \cdots, \gamma_m\})\leq k$. Equipped each simplex $\Delta_{\gamma}$ in $\widetilde{P_{k,m}}(\Gamma)$ with the spherical metric $d_{\Delta_{\gamma}}$. Then the metric on $\widetilde{P_{k,m}}(\Gamma)$ is defined to be the largest metric satisfying that
$$d(\langle 1, \gamma_1 \rangle, \langle 1, \gamma_2 \rangle)\leq d(\gamma_1, \gamma_2)\:\text{and}\: d(\langle t, \gamma \rangle, \langle t', \gamma \rangle)\leq d_{\Delta_{\gamma}}(\langle t, \gamma \rangle, \langle t', \gamma \rangle),$$
for $\langle 1, \gamma_j \rangle=(1\gamma_j, 0, 0, \cdots), j=1,2$ and $\langle t, \gamma \rangle=(t_0\gamma_0, t_1\gamma_1,\cdots, t_i\gamma_i, \cdots), \langle t', \gamma \rangle=(t'_0\gamma_0, t'_1\gamma_1,\cdots, t'_i\gamma_i, \cdots)\in \Delta_{\gamma}$ (please see \cite{GWWY-HHspaces-Novikov}). The action of $\Gamma$ on $\widetilde{P_{k,m}}(\Gamma)$ is defined to be
$$\gamma \cdot (t_0\gamma_0, t_1\gamma_1,\cdots, t_i\gamma_i, \cdots)=(t_0\gamma\gamma_0, t_1\gamma\gamma_1,\cdots, t_i\gamma\gamma_i, \cdots),$$
for any $(t_0\gamma_0, t_1\gamma_1,\cdots, t_i\gamma_i, \cdots)\in \widetilde{P_{k,m}}(\Gamma)$. Then this action is a free proper, co-compact isometric action and the canonical embedding map from $\Gamma$ to $\widetilde{P_{k,m}}(\Gamma)$ is a coarse $\Gamma$-equivalence. Moreover, for any $k_1\leq k_2$ and $m_1\leq m_2$, the canonical inclusion 
$$i_{k_1,k_2; m_1, m_2}: \widetilde{P_{k_1, m_1}}(\Gamma) \rightarrow \widetilde{P_{k_2, m_2}}(\Gamma),$$ 
is a uniformly continuous coarse $\Gamma$-map. 

Define a map $\lambda: \widetilde{P_{k,m}}(\Gamma) \rightarrow P_{k}(\Gamma)$ by
\begin{equation}\label{eq-twoRips}
	\lambda(\langle t, \gamma \rangle)= \sum_{i=1}^{m} t_i \gamma_i,
\end{equation}
for any $\langle t, \gamma \rangle=(t_0\gamma_0, t_1\gamma_1, \cdots, t_m\gamma_m,\cdots)\in \widetilde{P_{k,m}}(\Gamma)$. Then $\lambda$ is a uniformly continuous coarse $\Gamma$-map. Thus, $\lambda$ induces a homomorphism 
$$\lambda_{\ast}: K_{\ast}\left( C^{\ast}_{L}(\widetilde{P_{k,m}}(\Gamma), A)^{\Gamma} \right) \rightarrow K_{\ast}\left( C^{\ast}_{L}(P_{k}(\Gamma), A)^{\Gamma} \right),$$
for any $\Gamma$-$C^{\ast}$-algebra $A$ by Lemma \ref{Lem-conticovmap}.

Note that Milnor--Rips complexes provides a concrete model of the classifying space $E\Gamma$ for proper and free $\Gamma$-actions (please see \cite{Yu-BC-coarse}). Then the Mishchenko--Kasparov assembly map with coefficients and the rational analytic Novikov conjecture with coefficients are stated as follows.

\begin{definition}\label{Def-MK-assembly}
	The \textit{Mishchenko--Kasparov assembly map with coefficients} in $A$ for $\Gamma$ is defined to be the following composition map
	$$\mu_{\ast}: \lim_{k,m\rightarrow \infty} K_{\ast}\left( C^{\ast}_{L}(\widetilde{P_{k,m}}(\Gamma), A)^{\Gamma} \right) \xrightarrow{\lambda_{\ast}} 
	\lim_{k\rightarrow \infty} K_{\ast}\left( C^{\ast}_{L}(P_{k}(\Gamma), A)^{\Gamma} \right) \xrightarrow{e_{\ast}} 
	K_{\ast}\left( A \rtimes_{r} \Gamma \right).$$
\end{definition}

\begin{conjecture}[rational analytic Novikov conjecture with coefficients]\label{Conj-rational-analy-Novikov}
	Let $\Gamma$ be a countable discrete group and $A$ be a $\Gamma$-$C^{\ast}$-algebra. Then $\mu_{\ast}$ is rationally injective.
\end{conjecture}

\begin{remark}
	The rational analytic Novikov conjecture also implies the Novikov conjecture and the Gromov--Lawson--Rosenberg conjecture (cf. \cite{Kasparov-signatures}\cite{Rosenberg-IHES}).
\end{remark}

\section{The first main theorem}\label{Sec-2}

In this section, we investigate SNC, SAC and BCC for group extensions $1\rightarrow N \rightarrow \Gamma \rightarrow 
\Gamma/N \rightarrow 1$ under the assumption that any subgroup of $\Gamma$ containing $N$ as a subgroup with finite index satisfies BCC.

\subsection{In the case of direct products}\label{Subsection-prod}
Firstly, we assume that $\Gamma=N\times G$. For any $k\geq 0$, define a metric on $P_{k}(N)\times P_{k}(G)$ by 
$$d\left((x, y),(x',y')\right)=\max\{d_{P_{k}}(x,x'), d_{P_{k}}(y, y')\},$$
for all $x,x'\in P_{k}(N)$ and $y,y'\in P_{k}(G)$. Then $P_{k}(N)\times P_{k}(G)$ is an $(N \times G)$-space equipped with the $(N \times G)$-action defined by 
$$(n, g)\cdot(x,y)=(n x, g y),$$
for any $n\in N$, $g\in G$ and $x\in P_{k}(N)$, $y\in P_{k}(G)$.

Define $\rho: P_{k}(N \times G)\rightarrow P_{k}(N)\times P_{k}(G)$ and $\rho': P_{k}(N)\times P_{k}(G) \rightarrow P_{k}(N \times G)$ by
\begin{equation}\label{eq-Rip-prod}
  \begin{aligned}
  	& \rho\left( \sum_{i,j}t_{ij} (n_i, g_j)\right)=\left(\sum_i\left(\sum_j t_{ij}\right) n_i, \sum_{j}\left(\sum_i t_{ij}\right) g_j\right),\\
  	& \rho'\left(\sum_i t_i n_i, \sum_j s_j g_j\right)=\sum_{i,j} t_i s_j\left(n_i, g_j\right),
  \end{aligned}
\end{equation}
for all $n_i\in N$ and $g_i\in G$.  

Then $\rho$ and $\rho'$ are two strongly Lipschitz $(N\times G)$-homotopy equivalent and coarsely $(N\times G)$-equivalent maps between $P_{k}(N\times G)$ and $P_{k}(N)\times P_{k}(G)$ by \cite[Lemma 4.17]{Zhang-CBCFC}. Thus, we have the following lemma.
\begin{lemma}\label{Lem-reduction-product}
	For any $k\geq 0$, the Rips complex $P_{k}(N \times G)$ is strong Lipschitz $(N \times G)$-homotopy equivalent to and coarsely $(N\times G)$-equivalent to $P_{k}(N)\times P_{k}(G)$.
\end{lemma}

Thus, Combining Lemma \ref{Lem-reduction-product} with Corollary \ref{Cor-coarseequi-K} and Lemma \ref{Lem-Liphtp}, we have the following lemma.

\begin{lemma}\label{Prop-reduction-BC}
	Let $N$, $G$ be two countable discrete groups and $A$ be an $(N \times G)$-$C^{\ast}$-algebra. Then $N \times G$ satisfies SNC, SAC and BCC with coefficients in $A$ if and only if the following homomorphism induced by the evaluation at zero map
	$$e_{\ast}: \lim_{k\rightarrow \infty} K_{\ast}\left(C^{\ast}_{L}\left(P_{k}(N)\times P_{k}(G), A\right)^{N \times G}\right) \rightarrow \lim_{k\rightarrow \infty} K_{\ast}\left(C^{\ast}\left(P_{k}(N)\times P_{k}(G), A\right)^{N \times G}\right)$$
	is injective, surjective and isomorphic, respectively.
\end{lemma}

In order to divide the above conjectures for $N\times G$ into the corresponding conjectures for $N$ and $G$ independently, we introduce a partial version of the equivariant localization algebra.

For any $T\in C^{\ast}(P_{k}(N)\times P_{l}(G), A)^{N\times G}$, define
$$\prop_{P_{k}(N)}(T)=\sup\{d(x,x'): \text{there exist $y,y'\in P_{l}(G)$, s.t. $((x,y),(x',y'))\in \supp(T)$}\}.$$

\begin{definition}\label{Def-loc-along}
	The \textit{equivariant localization algebra along $P_{k}(N)$ with coefficients} in $A$ of $P_{k}(N)\times P_{l}(G)$, denoted by $C^{\ast}_{L, P_{k}(N)}(P_{k}(N)\times P_{l}(G), A)^{N\times G}$, is defined to be the norm closure of the $\ast$-algebra consisting of all bounded and uniformly continuous functions $u:[0,\infty)\rightarrow C^{\ast}(P_{k}(N)\times P_{l}(G), A)^{N\times G}$ such that
	$$\lim_{t\rightarrow \infty} \prop_{P_{k}(N)}(u(t))=0,$$
	in $\mathcal{L}\left(L^2([0,\infty)) \otimes \ell^2(Z_k(N)) \otimes \ell^2(Z_l(G)) \otimes H_A \otimes \ell^2(N\times G) \otimes H\right)$.
\end{definition}

Similar to Lemma \ref{Lem-Liphtp} and Lemma \ref{Lem-equiMV}, we have the following homological properties for the $K$-theory of the equivariant localization algebra along $P_{k}(N)$, please refer to \cite{Yu-Localizationalg} for the proofs.
\begin{lemma}\label{Lem-Liphomotopy-localong}
	Let $X_1$ and $X_2$ be two $N$-invariant subspaces of $P_{k}(N)$. If $X_1$ is strongly Lipschitz $N$-homotopy equivalent to $X_2$, then $K_{\ast}\left(C^{\ast}_{L, X_1}(X_1\times P_{l}(G), A)^{N\times G}\right)$ is naturally isomorphic to $K_{\ast}\left(C^{\ast}_{L, X_2}(X_2\times P_{l}(G), A)^{N\times G}\right)$.
\end{lemma}

\begin{lemma}\label{Lem-MV-localong}
	Let $X$ be an $N$-invariant subspace of $P_{k}(N)$. If $\{X_1, X_2\}$ is a uniformly excisive $N$-cover of $X$, then the following Mayer--Vietoris six-term exact sequence holds
	$$\tiny\xymatrix{
		K_0(C^{\ast}_{L, X_{1,2}}(X_{1,2}\times P_{l}(G), A)^{\Gamma}) \ar[r] & 
		\oplus_{i=1,2} K_0(C^{\ast}_{L, X_i}(X_i\times P_{l}(G), A)^{\Gamma}) \ar[r] & 
		K_0(C^{\ast}_{L, X}(X\times P_{l}(G), A)^{\Gamma})\ar[d] \\
		K_1(C^{\ast}_{L, X}(X\times P_{l}(G), A)^{\Gamma})\ar[u] &
		\oplus_{i=1,2} K_1(C^{\ast}_{L, X_i}(X_i\times P_{l}(G), A)^{\Gamma})\ar[l] &
		K_1(C^{\ast}_{L, X_{1,2}}(X_{1,2}\times P_{l}(G), A)^{\Gamma}),\ar[l] 
	}$$
	where $\Gamma=N\times G$ and $X_{1,2}=X_1\cap X_2$.
\end{lemma}

Let $A$ be an $(N\times G)$-$C^{\ast}$-algebra equipped with the action $\alpha$ of $N\times G$. Then there exists an action $\alpha^{G}$ of $G$ on $A$ defined by $\alpha^{G}_{g}(a)=\alpha_{(e, g)}(a)$. Thus, we obtain an equivariant Roe algebra $C^{\ast}(P_{l}(G), A)^{G}$. Besides, we have an action $\alpha^{N}_{G}$ of $N$ on $C^{\ast}(P_{l}(G), A)^{G}$ defined by
$$\alpha^{N}_{G, n}\left( (T_{y_1, y_2})_{y_1, y_2\in Z_{l}(G)} \right)=\left( (\alpha_{(n, e)}\otimes I)(T_{y_1, y_2}) \right)_{y_1, y_2\in Z_{l}(G)},$$
where $\alpha_{(n, e)}\otimes I$ is an operator on $A\otimes \K(H) \otimes \K(\ell^2(G))$ defined by $(\alpha_{(n, e)}\otimes I)(a\otimes K\otimes Q)=\alpha_{(n, e)}(a)\otimes K\otimes Q$. Therefore, we finally obtain an equivariant Roe algebra $C^{\ast}(P_{k}(N), C^{\ast}(P_{l}(G), A)^{G})^{N}$.

Let $U: H \rightarrow H\otimes H$ be a unitary. Define
$$\phi: C^{\ast}(P_{k}(N)\times P_{l}(G), A)^{N\times G} \rightarrow C^{\ast}(P_{k}(N), C^{\ast}(P_{l}(G), A)^{G})^{N},$$
by
$$\phi(T)_{x_1, x_2}=\left((U\otimes I) T_{(x_1, y_1),(x_2, y_2)} (U^{\ast}\otimes I)\right)_{y_1, y_2\in Z_{l}(G)},$$
for any $T\in C^{\ast}(P_{k}(N)\times P_{l}(G), A)^{N\times G}$ and $x_1, x_2\in Z_k(N)$. In addition, define
$$\tilde{\phi}: C^{\ast}_{L, P_{k}(N)}(P_{k}(N)\times P_{l}(G), A)^{N\times G} \rightarrow C^{\ast}_{L}(P_{k}(N), C^{\ast}(P_{l}(G), A)^{G})^{N},$$
by 
$$\tilde{\phi}(u)(t)=\phi(u(t)),$$
for any $u\in  C^{\ast}_{L, P_{k}(N)}(P_{k}(N)\times P_{l}(G), A)^{N\times G}$. 
Then $\phi$ and $\tilde{\phi}$ are two isomorphisms and we have the following naturally commutative diagram:
$$\xymatrix{
    C^{\ast}_{L, P_{k}(N)}(P_{k}(N)\times P_{l}(G), A)^{N\times G} \ar[r]^{e} \ar[d]_{\tilde{\phi}}^{\cong} & C^{\ast}(P_{k}(N)\times P_{l}(G), A)^{N\times G} \ar[d]^{\phi}_{\cong} \\
    C^{\ast}_{L}(P_{k}(N), C^{\ast}(P_{l}(G), A)^{G})^{N} \ar[r]_{e} &
    C^{\ast}(P_{k}(N), C^{\ast}(P_{l}(G), A)^{G})^{N}.
}$$
Moreover, by Lemma \ref{Lem-eqRoe-crossprod}, $C^{\ast}(P_{l}(G), A)^{G}$ is $N$-equivariantly isomorphic to $(A\rtimes_{r} G) \otimes \K(H)$ equipped with the action of $N$ by $n\cdot(\sum (a_{g}g)\otimes K)=\sum (\alpha_{(n, e)}(a_g)g)\otimes K$ for any $l\geq 0$. Thus, we obtain the following proposition.

\begin{proposition}\label{Prop-BCalong-BC}
	Let $N$, $G$ be two countable discrete groups and $A$ be an $(N \times G)$-$C^{\ast}$-algebra. Then $N$ satisfies SNC, SAC and BCC with coefficients in $A \rtimes_r G$ if and only if the following homomorphism
	$$e_{\ast}: \lim_{k\rightarrow \infty} K_{\ast}\left(C^{\ast}_{L, P_{k}(N)}\left(P_{k}(N)\times P_{l}(G), A\right)^{N \times G}\right)
	\rightarrow 
	\lim_{k\rightarrow \infty} K_{\ast}\left(C^{\ast}\left(P_{k}(N)\times P_{l}(G), A\right)^{N\times G}\right)$$
	is injective, surjective and isomorphic for each $l\geq 0$, respectively.
\end{proposition}

Next, we explore the connection between the localization algebras along one direction and the localization algebras of direct products. Firstly, let us recall an elementary lemma in the $K$-theory. A $C^{\ast}$-algebra $A$ is called to be \textit{quasi-stable}, if for all positive integer $n$, there exists an isometry $v$ in the multiplier algebra of $M_{n}(A)$ such that $vv^{\ast}$ is the matrix unit. Notice that any equivariant Roe algebras and equivariant localization algebras are quasi-stable.

\begin{lemma}(\cite[Lemma 12.4.3]{WillettYu-Book})\label{Lem-elementary}
	Let $C_{ub}([0,\infty), A)$ be the algebra of all bounded uniformly continuous functions from $[0,\infty)$ to $A$. If $A$ is quasi-stable, then the evaluation at zero map 
	$$e: C_{ub}([0,\infty), A)\rightarrow A, \: u\mapsto u(0)$$
	induces an isomorphism on the $K$-theory.
\end{lemma}


\begin{proposition}\label{Lem-Key}
	Let $A$ be an $(N \times G)$-$C^{\ast}$-algebra. If 
	$F\times G$ satisfies the Baum--Connes conjecture with coefficients in $A$ for any finite subgroup $F$ of $N$. Then for each $k\geq 0$, the inclusion map $\iota$ induces an isomorphism:
	$$\iota_{\ast}: \lim_{l\rightarrow \infty} K_{\ast}\left(C^{\ast}_{L}(P_{k}(N)\times P_{l}(G), A)^{N\times G}\right) \rightarrow \lim_{l\rightarrow \infty} K_{\ast}\left(C^{\ast}_{L,P_{k}(N)}(P_{k}(N)\times P_{l}(G), A)^{N\times G}\right).$$
\end{proposition}

\begin{proof}
	For each $k\geq 0$, since the metric on $\Gamma$ is proper, hence the dimension of $P_{k}(N)$ is finite. Let $P_{k}(N)^{(m)}$ be the $m$-dimensional skeleton of $P_{k}(N)$. We will prove the lemma by induction on $m$.
	
	For $m=0$, we have $P_{k}(N)^{(0)}=N$. Since $N$ is a proper metric space, there exists $c>0$ such that $d(n,n')>c$ for all $n, n'\in N$. Let $C^{\ast}_{L}(\prod_{N} P_{l}(G), A)^{N \times G}$ be a $C^{\ast}$-algebra consisting of all $u\in C^{\ast}_{L}(N \times P_{l}(G), A)^{N \times G}$ such that 
	$$u(t)_{(n, y), (n', y')}=0,$$
	for any $n\neq n' \in N$ and $y,y'\in P_l(G)$, $t\geq 0$.
	Let 
	$$J: C^{\ast}_{L}(\prod_{N} P_{l}(G), A)^{N \times G} \rightarrow C^{\ast}_{L}(N\times P_{l}(G), A)^{N \times G},$$
	be the inclusion map. Then it induces an isomorphism $J_{\ast}$ on the $K$-theory level. Indeed, for any element $u\in C^{\ast}_{L}(N\times P_{l}(G), A)^{N \times G}$, there exists $t_0$ such that $\prop(u(t))<c$ for any $t\geq t_0$. Thus, the element $t\mapsto u(t+t_0)\in C^{\ast}_{L}(\prod_{N} P_{l}(G), A)^{N \times G}$, which implies $J_{\ast}$ is surjective. For the injection of $J_{1}$, if $J_{1}([v])=0$ for an element $[v]\in K_1(C^{\ast}_{L}(\prod_{N} P_{l}(G), A)^{N \times G})$. Let $v^{(s)}$ be a path of unitaries connecting $v$ and $0$ in $M_{m}\left(C^{\ast}_{L}(N\times P_{l}(G), A)^{N \times G}\right)$. Choosing $v=v^{(0)}, v^{(1)}, \cdots, v^{(q)}=0$ such that $\|v^{(i)}-v^{(i-1)}\|<1$ for $i=1, \cdots, q$. Then there exists $t_1>0$ such that $\max_{0\leq i \leq q} \{\prop(v^{(i)}(t))\}<c$ for any $t\geq t_1$. Let $w^{(i)}(t)=v^{(i)}(t+t_1)$, then $w^{(i)}\in C^{\ast}_{L}(\prod_{N} P_{l}(G), A)^{N \times G}$ for any $i=0, \cdots, q$. Thus, $v$ is homotopic to $0$ by a sequence of linear paths between $w^{(i-1)}$ and $w^{(i)}$, which implies $[v]=0$ in $K_1(C^{\ast}_{L}(\prod_{N} P_{l}(G), A)^{N \times G})$. By a similar argument as above, we can also show that $J_0$ is injective. 
	
	Let $C^{\ast}_{L, N}(\prod_{N} P_{l}(G), A)^{N \times G}$ be a $C^{\ast}$-algebra consisting of all $u\in C^{\ast}_{L, N}(N \times P_{l}(G), A)^{N \times G}$ such that 
	$$u(t)_{(n, y), (n', y')}=0,$$
	for any $n\neq n' \in N$ and $y,y'\in P_l(G)$, $t\geq 0$. Then using a similar argument as the above paragraph, we can show that the inclusion map $J$ induces an isomorphism
	$$J_{\ast}: K_{\ast}\left( C^{\ast}_{L, N}(\prod_{N} P_{l}(G), A)^{N \times G} \right) \rightarrow K_{\ast}\left( C^{\ast}_{L, N}(N \times P_{l}(G), A)^{N \times G} \right).$$
	Thus, we have the following commutative diagram:
	$$\tiny\xymatrix{
		K_{\ast}\left(C^{\ast}_{L}(P_l(G), A)^{G}\right) \ar[r]^{\cong} \ar[d]^{\iota_{\ast}} & 
		K_{\ast}\left(C^{\ast}_{L}(\prod_{N} P_{l}(G), A)^{N \times G}\right)\ar[r]^{\cong}_{J_{\ast}} \ar[d]^{\iota_{\ast}} & 
		K_{\ast}\left(C^{\ast}_{L}(N\times P_{l}(G), A)^{N \times G}\right)\ar[d]^{\iota_{\ast}}  \\
		K_{\ast}\left( C_{ub}([0,\infty), C^{\ast}(P_l(G), A)^{G}) \right) \ar[r]^{\cong} &
		K_{\ast}\left( C^{\ast}_{L, N}(\prod_{N} P_{l}(G), A)^{N \times G} \right)\ar[r]^{\cong}_{J_{\ast}} &
		K_{\ast}\left( C^{\ast}_{L, N}(N \times P_{l}(G), A)^{N \times G} \right),
	}
	$$
	where the left two horizontal arrows are actually isomorphic on the level of $C^{\ast}$-algebras.
	Moreover, we also have the following commutative diagram:
	$$\xymatrix{
		\lim_{l\rightarrow \infty} K_{\ast}\left(C^{\ast}_{L}(P_l(G), A)^{G}\right) \ar[r]^{e_{\ast}} \ar[d]^{\iota_{\ast}} &
		\lim_{l\rightarrow \infty} K_{\ast}\left(C^{\ast}(P_l(G), A)^{G}\right)\\
		\lim_{l\rightarrow \infty} K_{\ast}\left( C_{ub}([0,\infty), C^{\ast}(P_l(G), A)^{G}) \right) \ar[ur]_{e_{\ast}},
	}
	$$
	and two evaluation at zero map $e$ induce two above isomorphisms $e_{\ast}$ by the assumption and Lemma \ref{Lem-elementary}. Thus, combing the above two commutative diagrams, we have that $\iota_{\ast}$ is an isomorphism for the case of $n=0$.
	
	Now, we assume the lemma holds for the case of $m=m'-1$. Next, we will prove it for $m=m'$. Let $c(\Delta)$ be the center of any $m'$-dimensional simplex $\Delta$ in $P_{k}(N)$. Define 
	$$\Delta_1=\{x\in \Delta: d(x, c(\Delta))\leq 1/10\};\:\: \Delta_2=\{x\in \Delta: d(x, c(\Delta))\geq 1/10\}.$$ 
	And let
	$$X_1=\bigcup\{\Delta_1:\text{dim}(\Delta)=m'\};\:\:X_2=\bigcup\{\Delta_2:\text{dim}(\Delta)=m'\}.$$
	Then $\{X_1, X_2\}$ is a uniformly excisive $N$-cover of $P_{k}(N)^{m'}$. Besides, $X_2$ and $X_1\cap X_2$ are strongly Lipschitz $N$-homotopy equivalent to $P_{k}(N)^{(m'-1)}$ and the disjoint union of the boundaries of all $m'$-dimensional simplices of $P_{k}(N)$, respectively. Thus, $\iota_{\ast}$ are isomorphic for $X_2$ and $X_1\cap X_2$ by the inductive assumption and Lemma \ref{Lem-Liphtp} as well as Lemma \ref{Lem-Liphomotopy-localong}. 
	Moreover, $X_1$ is strongly Lipschitz $N$-homotopy equivalent to $\{c(\Delta): \text{dim}(\Delta)=m'\}$. However, each $c(\Delta)$ is a $F$-fixed point for some finite subgroup $F$ of $N$. Thus by the assumption of the proposition and Lemma \ref{Lem-Liphtp}, Lemma \ref{Lem-Liphomotopy-localong} again, $\iota_{\ast}$ is an isomorphism for $X_1$. Therefore, by Lemma \ref{Lem-equiMV}, Lemma \ref{Lem-MV-localong} and the five lemma, $\iota_{\ast}$ is an isomorphism for the case of $m=m'$. 
\end{proof}

Combining Proposition \ref{Lem-Key} and Proposition \ref{Prop-BCalong-BC} with Lemma \ref{Prop-reduction-BC}, we obtain the following theorem.

\begin{theorem}\label{main-thm1}
	Let $N$ and $G$ be two discrete countable groups, $A$ be an $(N \times G)$-$C^{\ast}$-algebra. If 
	\begin{enumerate}
		\item \label{main-thm1-1} $F\times G$ satisfies the Baum--Connes conjecture with coefficients in $A$ for any finite subgroup $F$ of $N$.
		\item \label{main-thm1-2} $N$ satisfies SNC, SAC and BCC with coefficients in $A\rtimes_r G$, respectively.
	\end{enumerate} 
	Then $N\times G$ satisfies SNC, SAC and BCC with coefficients in $A$, respectively.
\end{theorem}

When $A=\mathbb{C}$, since $C^{\ast}_{r}(F)$ satisfies the K\"unneth formula for any finite group $F$ (cf. \cite[Corollary 0.2]{CEO-2004}), we obtain the following corollary.
\begin{corollary}\label{Cor-app-no-coeff}
	Let $N$ and $G$ be two discrete countable groups. If
	\begin{enumerate}
		\item $G$ satisfies the Baum--Connes conjecture;
		\item $N$ satisfies SNC, SAC and BCC with coefficients in $C^{\ast}_{r}(G)$ (with the trivial $N$-action), respectively.
	\end{enumerate}
	Then $N\times G$ satisfies SNC, SAC and BCC, respectively.
\end{corollary}

\subsection{The proof of the first main theorem}\label{Subsection-Green}
In this subsection, we generalize Theorem \ref{main-thm1} from the case of direct products to the case of group extensions. Let us first introduce a technical proposition associated to Green's imprimitivity theorem.

Let $\Gamma$ be a countable discrete group and $N$ be a subgroup of $\Gamma$. Let $B$ be an $N$-$C^{\ast}$-algebra equipped with the $N$-action $\alpha$.
Define $\Ind^{\Gamma}_{N}B$ to be the set consisting of all bounded functions $f: \Gamma \rightarrow B$ such that
\begin{itemize}
	\item $f(\gamma n)=\alpha_{n^{-1}}(f(\gamma))$ for any $\gamma\in \Gamma, n\in N$;
	\item $\gamma N \mapsto \|f(\gamma)\|$ is a function vanishing at infinity.
\end{itemize}
Then $\Ind^{\Gamma}_{N}B$ is a $C^{\ast}$-algebra equipped with the pointwise multiplication and the maximal norm. Endowed $\Ind^{\Gamma}_{N}B$ with a $\Gamma$-action $Ind\:{\alpha}$ by
$${Ind\:{\alpha}}_{\gamma}(f)(\gamma')=f(\gamma^{-1}\gamma'),$$
for any $f\in \Ind^{\Gamma}_{N}B$.

If $B$ is a $\Gamma$-$C^{\ast}$-algebra equipped with a $\Gamma$-action, also denoted by $\alpha$. And the action of $N$ on $B$ is the restriction of $\alpha$. Let $C_{0}(\Gamma / N, B)$ be the $C^{\ast}$-algebra of all functions from $\Gamma / N$ to $B$ vanishing at infinity. And let $\beta$ be a $\Gamma$-action on $C_{0}(\Gamma / N, B)$ defined by 
$$\beta_{\gamma}(f)(\gamma'N)=\alpha_{\gamma}(f(\gamma^{-1}\gamma'N)),$$
for any $f\in C_{0}(\Gamma / N, B)$. Define $\varrho: \Ind^{\Gamma}_{N}B \rightarrow C_0(\Gamma / N, B)$ by
$$\varrho(f)(\gamma'N)=\alpha_{\gamma'}(f(\gamma')),$$
for any $f\in \Ind^{\Gamma}_{N}B$. Then $\varrho$ is a $\Gamma$-equivariant isomorphism.

By \cite[Proposition 2.3]{CE-Permanence-BC}, we have the following proposition.

\begin{proposition}\label{Prop-Green-impri}
	Let $\Gamma$ be a countable discrete group, $N$ be a subgroup of $\Gamma$ and $B$ be an $N$-$C^{\ast}$-algebra. Then the following diagram commutes:
	$$\xymatrix{
		\lim_{k\rightarrow \infty}K_{\ast} \left( C^{\ast}_{L}(P_{k}(N), B)^{N}\right) \ar[r]^{e_{\ast}} \ar[d]_{\cong}  &
		\lim_{k\rightarrow \infty}K_{\ast}\left( C^{\ast}(P_{k}(N), B)^{N}\right) \ar[d]^{\cong}  \\
		\lim_{k\rightarrow \infty}K_{\ast}\left( C^{\ast}_{L}(P_{k}(\Gamma), \Ind^{\Gamma}_{N}B)^{\Gamma}\right) \ar[r]_{e_{\ast}}  &
		\lim_{k\rightarrow \infty}K_{\ast}\left( C^{\ast}(P_{k}(\Gamma), \Ind^{\Gamma}_{N}B)^{\Gamma}\right),
	}$$
	where two vertical maps are isomorphic because of Green's imprimitivity theorem (please see \cite[Theorem 2.2]{CE-Permanence-BC} and \cite[Theorem 17]{Green-78}).
\end{proposition}

Now, we begin to consider a group extension:
$$1\rightarrow N \rightarrow \Gamma \xrightarrow{q} \Gamma/ N \rightarrow 1.$$
Let $A$ be a $\Gamma$-$C^{\ast}$-algebra equipped with the $\Gamma$-action $\alpha$.

The group $\Gamma$ can be seen as a subgroup of $\Gamma \times \Gamma/ N$ by the following embedding:
$$\Gamma \rightarrow \Gamma \times \Gamma/ N, \: \gamma \mapsto (\gamma, [\gamma]).$$
Let $\beta'$ be an action of $\Gamma \times \Gamma/ N$ on $C_0(\Gamma/ N, A)$ defined by 
\begin{equation}\label{eq-product-action}
	\beta'_{(\gamma_1, [\gamma'_1])}(f)([\gamma_2])=\alpha_{\gamma_1}(f([\gamma^{-1}_1 \gamma_2 \gamma'_1])),
\end{equation}
for any $f\in C_0(\Gamma/N, A)$.
Thus, by Proposition \ref{Prop-Green-impri}, we have the following lemma.

\begin{lemma}\label{Lem-BC-sub-group}
	The following two statements are equivalent:
	\begin{enumerate}
		\item $\Gamma$ satisfies SNC, SAC and BCC with coefficients in $A$, respectively;
		\item $\Gamma \times \Gamma/N$ satisfies SNC, SAC and BCC with coefficients in $C_0(\Gamma/N, A)$, respectively.
	\end{enumerate}
\end{lemma}

Moreover, by applying Theorem \ref{main-thm1} to the $(\Gamma \times \Gamma/N)$-$C^{\ast}$-algebra $C_0(\Gamma/N, A)$ equipped with the action $\beta'$ defined by (\ref{eq-product-action}), we have the following lemma.

\begin{lemma}\label{Lem-BC-quotient}
	If the following two conditions hold:
	\begin{enumerate}
		\item $\Gamma \times F$ satisfies the Baum--Connes conjecture with coefficients in $C_0(\Gamma/N, A)$ for any finite subgroup $F$ of $\Gamma/N$;
		\item $\Gamma/N$ satisfies SNC, SAC and BCC with coefficients in $C_0(\Gamma/N, A) \rtimes_r \Gamma$, respectively.
	\end{enumerate}
	Then $\Gamma \times \Gamma/N$ satisfies SNC, SAC and BCC with coefficients in $C_0(\Gamma/N, A)$, respectively.
\end{lemma}

For any finite subgroup $F$ of $\Gamma/N$, $q^{-1}(F)$ can be seen as a subgroup of $\Gamma \times F$ by identifying $\gamma$ with $(\gamma, [\gamma])$ for any $\gamma \in q^{-1}(\Gamma)$. In addition, $A$ is a $q^{-1}(F)$-$C^{\ast}$-algebra by the action $\alpha$. Then $\Ind^{\Gamma \times F}_{q^{-1}(F)} A$ is $(\Gamma \times F)$-equivariantly isomorphic to $C_{0}(\Gamma/N, A)$ by an isomorphism $\pi': \Ind^{\Gamma \times F}_{q^{-1}(F)} A \rightarrow C_{0}(\Gamma/N, A)$ defined by 
$$\pi'(f)([\gamma_2])=\alpha_{\gamma_2}(f(\gamma_2, [e])),$$
where the action of $\Gamma \times F$ on $C_{0}(\Gamma/N, A)$ is defined by (\ref{eq-product-action}). Thus by applying Proposition \ref{Prop-Green-impri} again to $\Gamma \times F$ with a subgroup $q^{-1}(F)$, we get the following lemma.

\begin{lemma}\label{Lem-BC-N-Gamma}
	Let $F$ be a finite subgroup of $\Gamma/N$. Then the following two statements are equivalent:
	\begin{enumerate}
		\item $q^{-1}(F)$ satisfies the Baum--Connes conjecture with coefficients in $A$;
		\item $\Gamma \times F$ satisfies the Baum--Connes conjecture with coefficients in $C_0(\Gamma/N, A)$.
	\end{enumerate}
\end{lemma}

Combing Lemma \ref{Lem-BC-sub-group}-\ref{Lem-BC-N-Gamma}, we obtain the first main theorem of this paper.

\begin{theorem}\label{first-main-thm}
	For a group extension $1\rightarrow N \rightarrow \Gamma \xrightarrow{q} \Gamma/ N \rightarrow 1$. If the following two conditions hold:
	\begin{enumerate}
		\item the group $q^{-1}(F)$ satisfies the Baum--Connes conjecture with coefficients in $A$ for any finite subgroup $F$ of $\Gamma/N$;
		\item the group $\Gamma/N$ satisfies SNC, SAC and BCC with coefficients in $C_0(\Gamma/N, A) \rtimes_r \Gamma$, respectively.
	\end{enumerate}
	Then the group $\Gamma$ satisfies SNC, SAC and BCC with coefficients in $A$, respectively.
\end{theorem}

When $A=\mathbb{C}$, we have the following corollary.
\begin{corollary}\label{main-cor1}
	If the following two conditions hold:
	\begin{enumerate}
		\item the group $q^{-1}(F)$ satisfies the Baum--Connes conjecture for any finite subgroup $F$ of $\Gamma/N$;
		\item the group $\Gamma/N$ satisfies SNC, SAC and BCC with coefficients in $C_0(\Gamma/N) \rtimes_r \Gamma$, respectively.
	\end{enumerate}
	Then the group $\Gamma$ satisfies SNC, SAC and BCC, respectively.
\end{corollary}

\begin{remark}
	For the Baum--Connes conjecture with coefficients, Theorem \ref{first-main-thm} even holds for any locally compact group by J. Chabert, S. Echterhoff and H. Oyono-Oyono's result in \cite[Theorem 2.1]{CEO-2004}. However, Theorem \ref{first-main-thm} is new for the strong Novikov conjecture and the surjective assembly conjecture to the best of our knowledge. Besides, Theorem \ref{first-main-thm} is also true for the rational version of the above conjectures.
\end{remark}

\section{The second main theorem} \label{Sec-3}

In this section, we will show that the rational analytic Novikov conjecture holds for group extensions $1\rightarrow N \rightarrow \Gamma \rightarrow \Gamma/N \rightarrow 1$ under the assumption that $N$ satisfies the rational Baum--Connes conjecture and $\Gamma/N$ satisfies the rational analytic Novikov conjecture with coefficients.

\subsection{$K$-homology for Rips complexes and Milnor--Rips complexes}
Recall that a $C^{\ast}$-algebra $A$ satisfies the K\"{u}nneth formula in $K$-theory (please see \cite{Schochet-Kunneth}) if the following short exact sequence holds for any $C^{\ast}$-algebra $B$:
$$0\rightarrow K_{\ast}(A)\otimes K_{\ast}(B) \rightarrow K_{\ast}(A\otimes B) \rightarrow \text{Tor}(K_{\ast}(A), K_{\ast}(B)) \rightarrow 0,$$
where $A\otimes B$ is the minimal tensor product of $A$ and $B$. In \cite{Schochet-Kunneth}, C. Schochet proved that the K\"{u}nneth formula holds for any $C^{\ast}$-algebra in the Bootstrap class which contains all type I $C^{\ast}$-algebras. Besides for a commutative $\Gamma$-$C^{\ast}$-algebra $A$, if $\Gamma$ satisfies the Baum--Connes conjecture with coefficients in $A\otimes C$ for any $C^{\ast}$-algebra $C$ with the trivial $\Gamma$-action, then $A\rtimes_r \Gamma$ satisfies the K\"unneth formula by \cite[Corollary 0.2]{CEO-2004}.

The following lemma is a slight generalization of \cite[Lemma 11.1]{GWWY-HHspaces-Novikov} and its proof is similar.

\begin{lemma}\label{Lem-torsion-to-free}
	Let $N$, $G$ be two countable discrete groups and $\Gamma=N\times G$. If a $\Gamma$-$C^{\ast}$-algebra $A$ satisfies the K\"unneth formula. Then the following homomorphism  induced by $\lambda$ (see (\ref{eq-twoRips}))
	$$\lambda_{\ast}: \lim_{k,m\rightarrow \infty}K_{\ast}\left(C^{\ast}_{L}(\widetilde{P_{k,m}}(N)\times \widetilde{P_{k,m}}(G), A)^{\Gamma}\right) 
	\rightarrow 
	\lim_{k,m\rightarrow \infty}K_{\ast}\left(C^{\ast}_{L}(\widetilde{P_{k,m}}(N)\times P_{k}(G), A)^{\Gamma}\right),$$
	is rationally injective.
\end{lemma}

\begin{proof}
	Let $\Omega_{G}$ be the set of all linear orders on $G$ equipped with a $G$-action defined by $g_1 <_{g \mathcal{R}} g_2$ if and only if $g^{-1}g_1 <_{\mathcal{R}} g^{-1}g_2$ for any $g, g_1, g_2\in G$ and $\mathcal{R}\in \Omega_{G}$. Then every torsion element in $G$ acts freely on $\Omega_{G}$. There exists a left $G$-invariant metric on $G\otimes \Omega_{G}$ such that the canonical projection from $G\otimes \Omega_{G}$ to $G$ is a coarse $G$-equivalence and $\Omega_{G}$ is a compact Hausdorff space under this metric (please see \cite[Section 8]{GWWY-HHspaces-Novikov}).
	
	For any $k\geq 0$ and $m\geq \sharp B(e, k)$, define an ordering map $\varphi_{od}: P_{k}(G)\times \Omega_{G} \rightarrow \widetilde{P_{k,m}}(G)$ by
	$$\varphi_{od}\left(\sum_{g\in F}t_{g}g, \mathcal{R}\right) = (t_{\mathcal{R}, 1}g_{\mathcal{R}, 1}, \cdots, t_{\mathcal{R}, i}g_{\mathcal{R}, i}, 0, \cdots),$$
	where $F=\{g_{\mathcal{R}, 1}, \cdots, g_{\mathcal{R}, i}\}$ is a finite subset of $G$ with $g_{\mathcal{R}, j} <_{\mathcal{R}} g_{\mathcal{R}, j'}$ for any $j<j'$. In addition, define a collapsing map $\varphi_{cp}:\widetilde{P_{k,m}}(G)\times \Omega_{G} \rightarrow \widetilde{P_{k,m}}(G)$ by
	$$\varphi_{cp}\left(\langle s, g\rangle, \mathcal{R} \right) = (t_{1}g_{\mathcal{R}, 1}, \cdots, t_{i}g_{\mathcal{R}, i}, 0, \cdots),$$
	for any $\langle s,g \rangle \in \widetilde{P_{k,m}}(G)$ and $\mathcal{R}\in \Omega_{G}$, where $\{g_{\mathcal{R}, 1}, \cdots, g_{\mathcal{R}, i}: s_{g_{\mathcal{R}, i}\neq 0}, \: g_{\mathcal{R}, i'} <_{\mathcal{R}} g_{\mathcal{R}, i''} \: \text{for}\: i'<i''\}$ and $t_{i}=\sum_{g'=g_{\mathcal{R}, i}}s_{g'}$. Then both ordering map and collapsing map are continuous coarse $G$-equivalences. 
	
	In order to simplify the notations, we take $P_{k,m}=\widetilde{P_{k,m}}(N)\times P_{k,m}(G)$ and $\widetilde{P_{k,m}}=\widetilde{P_{k,m}}(N)\times \widetilde{P_{k,m}}(G)$ as well as $B=C(\Omega_{G})\rtimes_{r} G$.
	Define 
	$$\Phi_{\ast}: K_{\ast}\left(C^{\ast}_{L}(P_{k,m}, A)^{N\times G}\right) \rightarrow K_{\ast}\left(C^{\ast}_{L}(P_{k,m}\times \Omega_{G}, A\otimes B)^{N\times G}\right),$$ 
	by
	$$\Phi_{\ast}([u])=[V\left((u\otimes I_{B}) \oplus 0\right)V^{\ast}],$$
	where $V$ is an isometry defined in \cite[Section 10.1]{GWWY-HHspaces-Novikov}. Similarly, we can define $\Phi_{\ast}$ when replacing $P_{k,m}$ by $\widetilde{P_{k,m}}$.
	
	Since a continuous equivariant map between two metric space also induces a homomorphism between $K$-theory of equivariant localization algebras (cf. \cite[Chapter 6]{WillettYu-Book}), we have the following commutative diagram:
	$$\xymatrix{
		K_{\ast}\left(C^{\ast}_{L}(\widetilde{P_{k,m}}, A)^{N\times G}\right) \ar[r]^{\lambda_{\ast}} \ar[d]_{\Phi_{\ast}} &
		K_{\ast}\left(C^{\ast}_{L}(P_{k,m}, A)^{N\times G}\right) \ar[d]^{\Phi_{\ast}} \\
		K_{\ast}\left(C^{\ast}_{L}(\widetilde{P_{k,m}} \times \Omega_{G}, A\otimes B)^{N\times G}\right) \ar[r]^{\lambda_{\ast}} \ar[d]_{\varphi_{cp,\ast}} &
		K_{\ast}\left(C^{\ast}_{L}(P_{k,m}\times \Omega_{G}, A\otimes B)^{N\times G}\right) \ar[dl]_{\varphi_{od,\ast}}  \ar@/^/[ddl] \\
		K_{\ast}\left(C^{\ast}_{L}(\widetilde{P_{k,m}}, A\otimes B)^{N\times G}\right) \ar[d]_{i_{k,k';m,m', \ast}} &
		\\
		K_{\ast}\left(C^{\ast}_{L}(\widetilde{P_{k',m'}}, A\otimes B)^{N\times G}\right). &
	}$$
	such that $i_{k,k';m,m',\ast}\circ \varphi_{cp,\ast} \circ \Phi_{\ast}([u])=[i_{k,k';m,m'}(u)\otimes I_{B}]$ for any element $[u]\in K_{\ast}\left(C^{\ast}_{L}(\widetilde{P_{k,m}}, A)^{N\times G}\right)$.
	
	Define 
	$$\psi_{\ast}: \bigoplus_{j=0,1}K_{\ast-j}\left(C^{\ast}_{L}(\widetilde{P_{k',m'}}, A)^{N\times G}\right) \otimes K_{j}(B) \rightarrow K_{\ast}\left(C^{\ast}_{L}(\widetilde{P_{k',m'}}, A\otimes B)^{N\times G}\right),$$
	by
	$$\psi_{\ast}([p]\otimes [q])=[p\otimes q];\: \psi_{\ast}([u]\otimes [q])=[u\otimes q+I\otimes (I-q)],$$
	for any $[p]\in K_{0}\left(C^{\ast}_{L}(\widetilde{P_{k',m'}}, A)^{N\times G}\right)$, $[u]\in K_{1}\left(C^{\ast}_{L}(\widetilde{P_{k',m'}}, A)^{N\times G}\right)$ and $[q]\in K_{0}(B)$. By the assumption of the lemma, we have that $A$ satisfies the K\"unneth formula. Thus, by using Mayer--Vietoris sequence for $K$-theory of equivariant localization algebras and the five lemma, we can prove that $\psi_{\ast}$ is an isomorphism after tensoring $\Q$. 
	
	On the other hand, there exits an invariant trace $\tau_{\Omega_{G}}$ on $C(\Omega_{G})$ (cf. \cite[Section 8]{GWWY-HHspaces-Novikov}) which induces a canonical trace $\tau$ on $B=C(\Omega_{G})\rtimes_{r} G$ defined by
	$$\tau(\sum_{g\in G}c_{g}g)=\tau_{\Omega_{G}}(c_{e}).$$
	Then $\tau$ induces an homomorphism:
	$$\tau_{\ast}: K_{\ast}\left(C^{\ast}_{L}(\widetilde{P_{k',m'}}, A)^{N\times G}\right) \otimes K_{0}(B) \rightarrow K_{\ast}\left(C^{\ast}_{L}(\widetilde{P_{k',m'}}, A)^{N\times G}\right) \otimes_{\Z} \R.$$
	In addition, for any $[u]\in K_{\ast}\left(C^{\ast}_{L}(\widetilde{P_{k,m}}, A)^{N\times G}\right)\otimes_{\Z} \Q$, we have that
	$$\tau_{\ast}\circ \psi^{-1}_{\ast}\circ i_{k,k';m,m',\ast}\circ \varphi_{cp,\ast} \circ \Phi_{\ast}([u])=[i_{k,k';m,m'}(u)]\otimes 1.$$
	Thus, $\varphi_{cp,\ast} \circ \Phi_{\ast}=\varphi_{od,\ast}\circ \Phi_{\ast} \circ \lambda_{\ast}$ is rationally injective, which implies that $\lambda_{\ast}$ is rationally injective.
\end{proof}

\begin{remark}
	In \cite{AAS2016}, P. Antonini, S. Azzali and G. Skandalis introduced an equivariant $KK$-theory with real coefficients $KK^{\Gamma}_{\R, \ast}(A, B)$ for two $\Gamma$-$C^{\ast}$-algebra $A$ and $B$. Moreover, they proved that $\lambda$ defined in (\ref{eq-twoRips}) induces an injective homomorphism from $\lim_{k,m\rightarrow \infty}KK^{\Gamma}_{\R, \ast}(\widetilde{P_{k,m}}(\Gamma), A)$ to $\lim_{k\rightarrow \infty}KK^{\Gamma}_{\R, \ast}(P_{k}(\Gamma), A)$ for any $\Gamma$-$C^{\ast}$-algebra $A$ (please see \cite{AAS2020}). Based on these facts, we can give an another proof for Lemma \ref{Lem-torsion-to-free} by using the K\"unneth formula.
\end{remark}

\subsection{The proof of the second main theorem}
Let us first consider the rational analytic Novikov conjecture with coefficients for direct products of groups.

Define a map $\widetilde{\rho}: \widetilde{P_{k,m}}(N\times G) \rightarrow \widetilde{P_{k,m}}(N) \times \widetilde{P_{k,m}}(G)$ by
$$\widetilde{\rho}\left( t_{0}(n_0, g_0),\cdots, t_{m}(n_m, g_m), 0,\cdots \right)=\left( (t_{0}n_0,\cdots, t_{m}n_m, 0,\cdots), (t_{0}g_0,\cdots, t_{m}g_m, 0,\cdots) \right).$$
and 
a map $\widetilde{\rho'}: \widetilde{P_{k,m}}(N) \times \widetilde{P_{k,m}}(G) \rightarrow \widetilde{P_{k,m^2}}(N\times G)$ by
\begin{multline*}
	\widetilde{\rho'}\left( (t_{0}n_0,\cdots, t_{m}n_m, 0,\cdots), (s_{0}g_0,\cdots, s_{m}g_m, 0,\cdots) \right) =\\
	\left( t_{0}s_{0}(n_0, g_0), \cdots, t_{0}s_{m}(n_0, g_m), \cdots, t_{m}s_{m}(n_m, g_{m}), 0, \cdots \right).
\end{multline*}
Then $\widetilde{\rho'}\circ\widetilde{\rho}$ and $\widetilde{\rho}\circ\widetilde{\rho'}$ are strongly Lipschitz $(N\times G)$-homotopic to $i_{k,k;m,m^2}$ and $i_{k,k;m,m^2} \times i_{k,k;m,m^2}$, respectively (see \cite[Lemma 4.17]{Zhang-CBCFC}). Thus, $\widetilde{\rho}$ induces an isomorphism 
$$\widetilde{\rho}_{\ast}: \lim_{k,m\rightarrow \infty}K_{\ast}\left( C^{\ast}_{L}(\widetilde{P_{k,m}}(N\times G), A)^{N\times G} \right) \rightarrow \lim_{k,m\rightarrow \infty}K_{\ast}\left( C^{\ast}_{L}(\widetilde{P_{k,m}}(N)\times \widetilde{P_{k,m}}(G), A)^{N\times G} \right),$$
by Lemma \ref{Lem-Liphtp}. Moreover, we have the following commutative diagram
$$
\xymatrix{
	\widetilde{P_{k,m}}(N\times G) \ar[r]^{\widetilde{\rho}\:\:\:\:\:\:\:\:} \ar[d]_{\lambda} 
	& \widetilde{P_{k,m}}(N)\times \widetilde{P_{k,m}}(G) \ar[d]^{\lambda} \\
	P_{k}(N\times G) \ar[r]_{\rho\:\:\:\:\:\:\:} 
	&  P_{k}(N)\times P_{k}(G),    
}$$
where $\rho: P_{k}(N\times G) \rightarrow P_{k}(N)\times P_{k}(G)$ is defined as equation (\ref{eq-Rip-prod}). Thus, we have the following lemma.

\begin{lemma}\label{Lem-raNov-product}
	Let $N$, $G$ be two countable discrete groups and $A$ be an $(N\times G)$-$C^{\ast}$-algebra. Then $N\times G$ satisfies the rational analytic Novikov conjecture with coefficients in $A$ if and only if the following composition map
	$$\mu_{\ast}=e_{\ast}\circ \lambda_{\ast}: \lim_{k,m\rightarrow \infty} K_{\ast}\left(C^{\ast}_{L}(\widetilde{P_{k,m}}(N)\times \widetilde{P_{k,m}}(G), A)^{N\times G}\right)  \rightarrow K_{\ast}\left( A\rtimes_{r} (N\times G) \right)$$
	is rationally injective.	
\end{lemma}

By a similar argument of Proposition \ref{Prop-BCalong-BC}, we have the following proposition.
\begin{proposition}\label{Prop-localong-rational}
	Let $N$, $G$ be two countable discrete groups and $\Gamma=N\times G$. Let $A$ be a $\Gamma$-$C^{\ast}$-algebra. Then $N$ satisfies the rational analytic Novikov conjecture with coefficients in $A\rtimes_{r} G$ if and only if the following homomorphism
	$$e_{\ast}: \lim_{k, m\rightarrow \infty} K_{\ast}\left(C^{\ast}_{L, \widetilde{P_{k,m}}(N)}\left(\widetilde{P_{k,m}}(N)\times P_{l}(G), A\right)^{\Gamma}\right)
	\rightarrow 
	K_{\ast}\left( A \rtimes_{r} \Gamma \right)$$
	is rationally injective for any $l\geq 0$.
\end{proposition}

Moreover, since the action of $N$ on $\widetilde{P_{k,m}}(N)$ is free, we also have the following proposition by a similar proof of Proposition \ref{Lem-Key}.

\begin{proposition}\label{Prop-prod-key}
	 Let $N$, $G$ and $A$ be as above. If $G$ satisfies the rational Baum--Connes conjecture with coefficients in $A$. Then for any $k\geq 0$ and $m\in \N$, the inclusion map $\iota$ induces an isomorphism:
	$$\iota_{\ast}: \lim_{l\rightarrow \infty} K_{\ast}\left(C^{\ast}_{L}(\widetilde{P_{k,m}}(N)\times P_{l}(G), A)^{\Gamma}\right) \rightarrow \lim_{l\rightarrow \infty} K_{\ast}\left(C^{\ast}_{L,\widetilde{P_{k,m}}(N)}(\widetilde{P_{k,m}}(N)\times P_{l}(G), A)^{\Gamma}\right).$$
\end{proposition}

Combing Lemma \ref{Lem-torsion-to-free}, \ref{Lem-raNov-product} with Proposition \ref{Prop-localong-rational}, \ref{Prop-prod-key}, we obtain the following theorem.

\begin{theorem}\label{Thm-rational-Novikov-prod}
	Let $N$ and $G$ be two countable discrete groups. Assume that $A$ is an $(N\times G)$-$C^{\ast}$-algebra satisfying the K\"unneth formula. If
	\begin{enumerate}
		\item $N$ satisfies the rational analytic Novikov conjecture with coefficients in $A\rtimes_{r} G$, and;
		\item  $G$ satisfies the rational Baum--Connes conjecture with coefficients in $A$. 
	\end{enumerate}
	Then $N\times G$ satisfies the rational analytic Novikov conjecture with coefficients in $A$.
\end{theorem}

Finally, along the strategy of Subsection \ref{Subsection-Green}, we consider the rational analytic Novikov conjecture with coefficients for group extensions. Let us first relate the rational analytic Novikov conjecture with coefficients for groups and for their subgroups.

\begin{lemma}\label{Lem-Green-impri-free}
	Let $\Gamma$ be a countable discrete group, $N$ be a subgroup of $\Gamma$ and $B$ be an $N$-$C^{\ast}$-algebra satisfying the K\"unneth formula. If $\Gamma$ satisfies the rational analytic Novikov conjecture with coefficients in $\Ind^{\Gamma}_{N}B$, then $N$ satisfies the rational analytic Novikov conjecture with coefficients in $B$.
\end{lemma}
\begin{proof}
	By Proposition \ref{Prop-Green-impri}, we have the following commutative diagram as $k, m\rightarrow \infty$:
	$$\xymatrix{
		K_{\ast} \left( C^{\ast}_{L}(\widetilde{P_{k, m}}(N), B)^{N}\right) \ar[r]^{\:\:\:\:\lambda_{\ast}} \ar[d]_{\Ind^{\Gamma}_{N}}  &
		K_{\ast}\left( C^{\ast}_{L}(P_{k}(N), B)^{N}\right) \ar[r]^{\:\:\:\:\:\:\:\:\:\:\:\:e_{\ast}} \ar[d]_{\Ind^{\Gamma}_{N}}^{\cong} & 
		K_{\ast}\left(B\rtimes_{r} N\right) \ar[d]^{\cong}_{\Ind^{\Gamma}_{N}}
		\\
		K_{\ast}\left( C^{\ast}_{L}(\widetilde{P_{k, m}}(\Gamma), \Ind^{\Gamma}_{N}B)^{\Gamma}\right) \ar[r]^{\:\:\:\:\lambda_{\ast}}  &
		K_{\ast}\left( C^{\ast}_{L}(P_{k}(\Gamma), \Ind^{\Gamma}_{N}B)^{\Gamma}\right) \ar[r]^{\:\:\:\:\:\:\:\:\:\:\:\:e_{\ast}} &
		K_{\ast}\left(\Ind^{\Gamma}_{N}B\rtimes_{r} \Gamma\right),
	}$$
	where $\Ind^{\Gamma}_{N}$ is the induction homomorphism (cf. \cite[Section 2]{CE-Permanence-BC}). Moreover, the homomorphism $\lambda_{\ast}$ in the top row is rationally injective by Lemma \ref{Lem-torsion-to-free}. Thus, by tracing along the above commutative diagram, we complete the proof.
\end{proof}

For a group extension:
$$1\rightarrow N \rightarrow \Gamma \rightarrow \Gamma/N \rightarrow 1.$$
The group $\Gamma$ can be seen as a subgroup of $\Gamma \times \Gamma/N$ by the following embedding:
$$\Gamma \rightarrow \Gamma \times \Gamma/N, \: \gamma \mapsto (\gamma, [\gamma]).$$
In addition, we have $\Ind^{\Gamma\times \Gamma/N}_{\Gamma} A=C_0(\Gamma/N, A)$ for any $\Gamma$-$C^{\ast}$-algebra $A$. Thus, by Lemma \ref{Lem-Green-impri-free}, we have the following lemma.

\begin{lemma}\label{Lem-BC-sub-group-free}
	Let $A$ be a $\Gamma$-$C^{\ast}$-algebra satisfying the K\"unneth formula. If $\Gamma \times \Gamma/N$ satisfies the rational analytic Novikov conjecture with coefficients in $C_0(\Gamma/N, A)$, then $\Gamma$ satisfies the rational analytic Novikov conjecture with coefficients in $A$;
\end{lemma}

Moreover, by applying Theorem \ref{Thm-rational-Novikov-prod} to the $(\Gamma \times G)$-$C^{\ast}$-algebra $C_0(G, A)$, we have the following lemma.

\begin{lemma}\label{Lem-BC-quotient-free}
	Let $A$ be a $\Gamma$-$C^{\ast}$-algebra satisfying the K\"unneth formula. If the following two conditions hold:
	\begin{enumerate}
		\item $\Gamma$ satisfies the rational Baum--Connes conjecture with coefficients in $C_0(\Gamma/N, A)$;
		\item $\Gamma/N$ satisfies the rational analytic Novikov conjecture with coefficients in $C_0(\Gamma/N, A) \rtimes_r \Gamma$.
	\end{enumerate}
	Then $\Gamma \times \Gamma/N$ satisfies the rational analytic Novikov conjecture with coefficients in $C_0(\Gamma/N, A)$.
\end{lemma}

Then by applying Proposition \ref{Prop-Green-impri} to $\Gamma$ with a subgroup $N$, we get the following lemma.

\begin{lemma}\label{Lem-BC-N-Gamma-free}
	 Let $A$ be a $\Gamma$-$C^{\ast}$-algebra. Then the following two statements are equivalent:
	\begin{enumerate}
		\item $N$ satisfies the rational Baum--Connes conjecture with coefficients in $A$;
		\item $\Gamma$ satisfies the rational Baum--Connes conjecture with coefficients in $C_0(\Gamma/N, A)$.
	\end{enumerate}
\end{lemma}

Combing Lemma \ref{Lem-BC-sub-group-free}-\ref{Lem-BC-N-Gamma-free}, we obtain the second main theorem of this paper.

\begin{theorem}\label{second-main-thm}
	Let $1\rightarrow N \rightarrow \Gamma \rightarrow \Gamma/N \rightarrow 1$ be a group extension and $A$ be a $\Gamma$-$C^{\ast}$-algebra satisfying the K\"unneth formula. If the following two conditions hold:
	\begin{enumerate}
		\item the group $N$ satisfies the rational Baum--Connes conjecture with coefficients in $A$;
		\item the group $\Gamma/N$ satisfies the rational analytic Novikov conjecture with coefficients in $C_0(\Gamma/N, A) \rtimes_r \Gamma$.
	\end{enumerate}
	Then the group $\Gamma$ satisfies the rational analytic Novikov conjecture with coefficients in $A$.
\end{theorem}

\begin{remark}
	The first condition of Theorem \ref{second-main-thm} is strictly weaker than the first condition of Theorem \ref{first-main-thm} for certain coefficients. (Please see R. Meyer's example in \cite{Meyer-BC-counterex} and Example \ref{EX-Meyer}). 
\end{remark}

When $A=\mathbb{C}$, we have the following corollary.
\begin{corollary}\label{main-cor2}
	If the following two conditions hold:
	\begin{enumerate}
		\item the group $N$ satisfies the rational Baum--Connes conjecture;
		\item the group $\Gamma/N$ satisfies the rational analytic Novikov conjecture with coefficients in $C_0(\Gamma/N) \rtimes_r \Gamma$.
	\end{enumerate}
	Then the group $\Gamma$ satisfies the rational analytic Novikov conjecture.
\end{corollary}

\section{Isometric semi-direct products and two-parametric equivariant localization algebras} \label{Sec-4}

In the above two sections, we always assume that the (rational) Baum--Connes conjecture with coefficients holds for subgroups of extensions. In the next two sections, we will weaken this assumption for a class of extensions, called isometric semi-direct products.

\subsection{Isometric semi-direct products}
\begin{definition}\label{Def-semi-pro}
	Let $G$ and $N$ be two countable discrete groups. Equipped $N$ with a $G$-action $\beta$. The \textit{semi-direct product} $N\rtimes G$ is defined to be a group with the set $N\times G$ and the multiplication given by 
	$$(n, g)(n',g')=(n\beta_{g}(n'), gg'),$$
	for any $n,n'\in N$ and $g,g'\in G$.
	
	If there exists a left $N$-invariant proper metric $d_{N}$ on $N$ such that the action $\beta$ is isometric, then we call $N\rtimes G$ an \textit{isometric semi-direct product}.
\end{definition}

\begin{remark}
    The direct products, restricted wreath products and semi-direct products by finite groups are examples of isometric semi-direct products.
\end{remark}

Let $d_{G}$ be a left $G$-invariant proper metric on $G$. Then we can define a proper metric $d$ on the isometric semi-direct product $N\rtimes G$ by  
  $$d\left( (n_1,g_1), (n_2,g_2)\right)=\max\{d_{N}(n_1,n_2), d_{G}(g_1,g_2)\}.$$
This metric is left $(N\rtimes G)$-invariant since the action $\beta$ of $G$ on $N$ is isometric. 

Define a left $(N\rtimes G)$-action on $P_{k}(N)\times P_{k}(G)$ by 
$$(n,g)(\sum_{i=1}^{l} t_{i} n_{i},  \sum_{j=1}^{m} s_{j} g_{j})=(\sum_{i=1}^{l} t_{i} (n\beta_{g}(n_{i})), \sum_{j=1}^{m} s_{j} (gg_{j})),$$
for any $n,n_1, \cdots, n_l\in N$ and $g, g_1, \cdots, g_m\in G$. Then this action is proper, isometric and co-compact.
Thus, by \cite[Lemma 4.17]{Zhang-CBCFC}, we have the following lemma which is similar to Lemma \ref{Lem-reduction-product}.

\begin{lemma}\label{Lem-Rip-pro}
	Let $N\rtimes G$ be an isometric semi-direct product. Then $P_{k}(N\rtimes G)$ is strongly Lipschitz $(N\rtimes G)$-homotopy equivalent to and coarsely $(N\rtimes G)$-equivalent to $P_{k}(N)\times P_{k}(G)$ for any $k\geq 0$.
\end{lemma}

Thus, by Corollary \ref{Cor-coarseequi-K} and Lemma \ref{Lem-Liphtp}, we have the following corollary.

\begin{corollary}\label{Cor-red-conj}
	Let $N\rtimes G$ be an isometric semi-direct product and $A$ be an $(N\rtimes G)$-$C^{\ast}$-algebra. Then $N\rtimes G$ satisfies SNC, SAC and BCC with coefficients in $A$, respectively, if and only if the following evaluation at zero map 
	$$e_{\ast}: \lim_{k\rightarrow \infty} K_{\ast}\left(C^{\ast}_{L}(P_{k}(N)\times P_{k}(G), A)^{N\rtimes G}\right)\rightarrow \lim_{k\rightarrow \infty} K_{\ast}\left(C^{\ast}(P_{k}(N)\times P_{k}(G), A)^{N\rtimes G}\right)$$
	is injective, surjective and isomorphic, respectively.
\end{corollary}

\subsection{Two-parametric equivariant localization algebras}

In this subsection, we introduce a new variation of the equivariant localization algebra and we will use it to reduce the Baum-Connes assembly map with coefficients for isometric semi-direct products in the next subsection.

Let $N\rtimes G$ be an isometric semi-direct product, $A$ be an $(N\rtimes G)$-$C^{\ast}$-algebra and $H_A$ be a covariantly represented Hilbert space of $A$. For any $k, l\geq 0$, let $Z_{k}(N)\times Z_{l}(G)$ be an $(N\rtimes G)$-invariant countable dense subset in $P_{k}(N)\times P_{l}(G)$. Let $H$ be a separable Hilbert space equipped with the trivial $(N\rtimes G)$-action.

For an operator $T\in C^{\ast}(P_{k}(N)\times P_{l}(G), A)^{N\rtimes G}$, recall that
$$\prop_{P_{k}(N)}(T)=\sup\{d(x,x'): \text{there exist $y,y'\in P_{l}(G)$, s.t. $((x,y),(x',y'))\in \supp(T)$}\},$$
and 
$$\prop_{P_{l}(G)}(T)=\sup\{d(y,y'): \text{there exist $x,x'\in P_{k}(N)$, s.t. $((x,y),(x',y'))\in \supp(T)$}\}.$$

\begin{definition}\label{Def-two-localg}
	The \textit{two-parametric equivariant localization algebra} with coefficients in $A$ of $P_{k}(N)\times P_{l}(G)$, denoted by $C^{\ast}_{LL}(P_{k}(N)\times P_{l}(G), A)^{N\rtimes G}$, is defined to be the norm closure of the $\ast$-algebra consisting of all bounded functions $u:[0,\infty)\times [0,\infty) \rightarrow C^{\ast}(P_{k}(N)\times P_{l}(G), A)^{N\rtimes G}$ such that
	 \begin{itemize}
	 	\item the family of functions $\{u(s, \cdot)\}_{s\geq 0}$ is uniformly equicontinuous;
	 	\item for each $t\geq 0$, the function $u(\cdot,t)$ is uniformly continuous;
	 	\item $\lim_{t\rightarrow \infty}\sup_{s\in [0,\infty)} \prop_{P_{l}(G)}(u(s,t))=0$;
	 	\item $\lim_{s\rightarrow \infty}\prop_{P_{k}(N)}(u(s,t))=0$ for each $t\geq 0$,
	 \end{itemize}
	in $\mathcal{L}(L^2([0,\infty)^2) \otimes \ell^2(Z_{k}(N)\times Z_{l}(G)) \otimes H_A \otimes \ell^2(N\times G) \otimes H)$. 
\end{definition}

Let $Y_1$, $Y_2$ be two $G$-invariant subspaces of $P_{l}(G)$ and $Z_{Y_1}$, $Z_{Y_2}$ be two $G$-invariant countable dense subsets of $Y_1$, $Y_2$, respectively. Let $f:Y_1\rightarrow Y_2$ be a uniformly continuous coarse $G$-map and 
$$w_f(R)=\sup\{d(f(y_1),f(y_2)): d(y_1,y_2)\leq R\}.$$ Let $\{\varepsilon_{m}\}_{m\in \N}$ be a decreasing sequence of positive numbers with $\lim_{m\rightarrow \infty}\varepsilon_m=0$. Then there exists a $2\pi$-Lipschitz $G$-equivariant isometry 
$$V_f: L^2([0,\infty))\otimes \ell^2(Z_{Y_1})\otimes \ell^2(G)\otimes H \rightarrow L^2([0,\infty))\otimes \ell^2(Z_{Y_2})\otimes \ell^2(G)\otimes H,$$ 
such that
$$\supp(V_f(t))\subseteq \{(y_1,y_2)\in Y_1\times Y_2: d(f(y_1), y_2)<\varepsilon_{m}\},$$
for any $t\in [m,m+1]$ (cf. \cite[Proposition 4.5.12]{WillettYu-Book} or \cite{Zhang-CBCFC}). Thus, $V_f$ induces the following homomorphism:
$$Ad_{f}: C^{\ast}_{LL}(P_{k}(N)\times Y_1, A)^{N\rtimes G} \rightarrow C^{\ast}_{LL}(P_{k}(N)\times Y_2, A)^{N\rtimes G},$$
defined by $$Ad_{f}(u)(s,t)=\left(I_{\ell^2(Z_k(N))\otimes \ell^2(N)\otimes H_A}\otimes V_f(t)\right)(u(s,t)\oplus 0)\left(I_{\ell^2(Z_k(N))\otimes \ell^2(N)\otimes H_A}\otimes V^{\ast}_f(t)\right).$$
Moreover, we have that 
$$\prop_{Y_2}(Ad_{f}(u)(s,t))\leq w_f(\prop_{Y_1}(u(s,t)))+2\varepsilon_{m},$$ $$\prop_{P_{k}(N)}(Ad_{f}(u)(s,t))=\prop_{P_{k}(N)}(u(s,t)),$$
for any $t\in [m,m+1]$ and $m\in \N$.

Before discussing the homological properties for $K$-theory of two-parametric equivariant localization algebras, we first recall an elementary lemma in the $K$-theory (please see \cite[Proposition 2.7.5]{WillettYu-Book}).

\begin{lemma}\label{isometryequivalent}
  Let $\psi: B\rightarrow C$ be a $\ast$-homomorphism between two $C^{\ast}$-algebras and $v$ be an isometry in the multiplier algebra of $C$. Then the map
    $$B\rightarrow C,\: b \mapsto v\psi(b)v^{\ast}$$
  induces the same map on $K$-theory as $\psi$.
\end{lemma}

\begin{lemma}\label{Lem-Liphomotopy-twoloc}
	Let $Y_1$ and $Y_2$ be two $G$-invariant subspaces of $P_{l}(G)$. If $Y_1$ is strongly Lipschitz $G$-homotopy equivalent to $Y_2$, then $K_{\ast}\left(C^{\ast}_{LL}(P_{k}(N)\times Y_1, A)^{N\rtimes G}\right)$ is naturally isomorphic to $K_{\ast}\left(C^{\ast}_{LL}(P_{k}(N)\times Y_2, A)^{N\rtimes G}\right)$.
\end{lemma}
\begin{proof}
		By the assumption, there exist two Lipschitz maps $g: Y_1\rightarrow Y_2$ and $h:Y_2\rightarrow Y_1$ such that $hg$ and $gh$ are strongly Lipschitz $G$-homotopic to $id_{Y_1}$ and $id_{Y_2}$, respectively. It is sufficient to prove that $Ad_{hg}$ induces the identity map on $K_{\ast}\left(C^{\ast}_{LL}(P_{k}(N)\times Y_1, A)^{N\rtimes G}\right)$ and similar to $Ad_{gh}$. Let $F(t,y)$ be a strongly Lipschitz $G$-homotopy with the Lipschitz constant $C$ such that $F(0,\cdot)=hg$ and $F(1,\cdot)=id_{Y_1}$. Then we can choose a sequence of non-negative numbers $\{t_{i,j}\}_{i,j\geq 0}$ in $[0,1]$ such that
	\begin{itemize}
		\item $t_{0, j}=0$ and $t_{i, j}\leq t_{i+1,j}$;
		\item for each $j$, there exists $N_j\geq 0$ such that $t_{i,j}=1$ for all $i\geq N_j$;
		\item $d(F(t_{i,j}, y), F(t_{i+1,j}, y))\leq 1/(j+1)$ and $d(F(t_{i,j}, y), F(t_{i,j+1}, y))\leq 1/(j+1)$ for all $y\in Y_1$.
	\end{itemize}
	Then there exists a sequence of $2\pi$-Lipschitz $G$-equivariant isometry $V_{F(t_{i,j},\cdot)}$ on $L^2([0,\infty))\otimes \ell^2(Z_{Y_1})\otimes \ell^2(G)\otimes H$ such that
	$$\supp(V_{F(t_{i,j},\cdot)}(t))\subseteq \{(y,y')\in Y_1\times Y_1: d(F(t_{i,j},y), y')<1/(1+i+j)\}$$
	for any $i,j\geq 0$ by \cite[Proposition 4.5.12]{WillettYu-Book}. 
	Define
	$$V_i(t)=R(t-j) (V_{F(t_{i,j},\cdot)}(t)\oplus V_{F(t_{i,j+1},\cdot)}(t)) R^{\ast}(t-j),\: t\in [j,j+1],$$
	where 
	$$R(t)=
	\begin{pmatrix}
		\cos(\pi t/2) & \sin(\pi t/2)\\
		-\sin(\pi t/2) & \cos(\pi t/2)
	\end{pmatrix}.$$
	Now we prove the lemma for $\ast=1$ and by a suspension argument, we can similarly prove it for $\ast=0$. For an element $[u]\in K_{1}(C^{\ast}_{LL}(P_{k}(N)\times Y_1, A)^{N\rtimes G})$, we assume $u\in (C^{\ast}_{LL}(P_{k}(N)\times Y_1, A)^{N\rtimes G})^{+}$. Let $W=u\oplus I$. Define the following three unitaries:
	\begin{equation*}
		\begin{split}
			&a_1=\bigoplus_{i\geq 0} ((I\otimes V_i(t))((u(s,t)-I)\oplus 0)(I\otimes V^{\ast}_i(t))+I)\cdot W^{\ast};\\
			&a_2=\bigoplus_{i\geq 0} ((I\otimes V_{i+1}(t))((u(s,t)-I)\oplus 0)(I\otimes V^{\ast}_{i+1}(t))+I)\cdot W^{\ast};\\
			&a_3=I\bigoplus_{i\geq 1} ((I\otimes V_i(t))((u(s,t)-I)\oplus 0)(I\otimes V^{\ast}_i(t))+I)\cdot W^{\ast}.
		\end{split}
	\end{equation*}
	Since $\prop_{Y_1}(a_i)\leq (C+1) \prop_{Y_1}(u(s,t))+1/(j+1)$ for $t\in [j,j+1]$ and $\prop_{P_{k}(N)}(a_i)=2\prop_{P_{k}(N)}(u(s,t))$, we have that $a_i\in M_2(C^{\ast}_{LL}(P_{k}(N)\times Y_1, A)^{N\rtimes G})^{+}$. Moreover, since $d(F(t_{i,j}, y), F(t_{i+1,j}, y))\leq 1/(j+1)$, we have that $[a_1]=[a_2]$ in $K_{1}(C^{\ast}_{LL}(P_{k}(N)\times Y_1, A)^{N\rtimes G})$. Let 
	$$V: H^{\infty}\rightarrow H^{\infty}, \: (h_1, h_2, \cdots)\mapsto (0, h_1, \cdots),$$ 
	then $I\otimes V$ is an isometric multiplier of $C^{\ast}_{LL}(P_{k}(N)\times Y_1, A)^{N\rtimes G}$ and we have that
	$$a_3=(I\otimes V)(a_2-I)(I\otimes V^{\ast})+I.$$ 
	thus $[a_2]=[a_3]$ in $K_{1}(C^{\ast}_{LL}(P_{k}(N)\times Y_1, A)^{N\rtimes G})$ by Lemma \ref{isometryequivalent}. Besides, we have that $$[a_1a^{-1}_3]=[(V_0(t)((u(t)-I)\oplus 0)V^{\ast}_0(t)+I)\cdot W^{\ast}\bigoplus_{i\geq 1}I]=[Ad_{hg}(u)W^{\ast}\bigoplus_{i\geq 1}I],$$
	which is equal to $[I]$. Thus $Ad_{hg, \ast}=id$ on $K_{1}(C^{\ast}_{LL}(P_{k}(N)\times Y_1, A)^{N\rtimes G})$.
\end{proof}

\begin{lemma}\label{Lem-MV-twoloc}
	Let $Y$ be a $G$-invariant subspace of $P_{l}(G)$. If $\{Y_1, Y_2\}$ is a uniformly excisive $G$-cover of $Y$, then the following Mayer--Vietoris six-term exact sequence holds for any $k\geq 0$
	$$\scriptsize\xymatrix{
		K_0(C^{\ast}_{LL}(P_{k}(N)\times Y_{1,2}, A)^{\Gamma}) \ar[r] & 
		\bigoplus_{i=1,2} K_0(C^{\ast}_{LL}(P_{k}(N)\times Y_{i}, A)^{\Gamma}) \ar[r] & 
		K_0(C^{\ast}_{LL}(P_{k}(N)\times Y, A)^{\Gamma})\ar[d] \\
		K_1(C^{\ast}_{LL}(P_{k}(N)\times Y, A)^{\Gamma})\ar[u] &
		\bigoplus_{i=1,2} K_1(C^{\ast}_{LL}(P_{k}(N)\times Y_{i}, A)^{\Gamma})\ar[l] &
		K_1(C^{\ast}_{LL}(P_{k}(N)\times Y_{1,2}, A)^{\Gamma}),\ar[l] 
	}$$
	where $\Gamma=N\rtimes G$ and $Y_{1,2}=Y_1\cap Y_2$.
\end{lemma}
\begin{proof}
	Let $C^{\ast}_{LL}(P_{k}(N)\times Y, A; Y_i)^{\Gamma}$ be a closed subalgebra of $C^{\ast}_{LL}(P_{k}(N)\times Y, A)^{\Gamma}$ generated by all elements $u$ satisfying that there exists $c_t\geq 0$ with $\lim_{t\rightarrow \infty}c_t=0$ such that 
	\begin{equation}\label{equ-1}
	\supp(u(s,t))\subseteq \{((x,y),(x',y'))\in (P_{k}(N)\times Y)^2: d((y,y'), Y_i\times Y_i)\leq c_t\},
	\end{equation}
	for all $s,t\in [0,\infty)$ and $i=1,2$. Then there exists a natural inclusion map 
	$$\tau: C^{\ast}_{LL}(P_{k}(N)\times Y_i, A)^{\Gamma}\rightarrow C^{\ast}_{LL}(P_{k}(N)\times Y, A; Y_i)^{\Gamma}.$$
	Now we prove that $\tau$ induces an isomorphism on the $K$-theory. By the assumption, there exists $\varepsilon_0>0$ such that $\varepsilon_0$-neighborhood $B_{\varepsilon_0}(Y_i)$ of $Y_i$ is strongly Lipschitz $G$-homotopy equivalent to $Y_i$. For any element $u\in C^{\ast}_{LL}(P_{k}(N)\times Y, A; Y_i)^{\Gamma}$ satisfying the condition (\ref{equ-1}), there exists $t_0>0$ such that 
	$$\supp(u(s,t+t_0))\subseteq \{((x,y),(x',y'))\in (P_{k}(N)\times Y)^2: (y,y')\in B_{\varepsilon_0}(Y_i)\times B_{\varepsilon_0}(Y_i)\},$$
	for all $s,t\geq 0$. Let $v(s,t)=u(s,t+t_0)$, then $u$ and $v$ are homotopic equivalent in $C^{\ast}_{LL}(P_{k}(N)\times Y, A; Y_i)^{\Gamma}$ because $\{u(s,\cdot)\}_{s\geq 0}$ is uniformly equicontinuous. Thus, the inclusion map from $C^{\ast}_{LL}(P_{k}(N)\times B_{\varepsilon_0}(Y_i), A)^{\Gamma}$ to $C^{\ast}_{LL}(P_{k}(N)\times Y, A; Y_i)^{\Gamma}$ induces an isomorphism on the $K$-theory. Therefore, $\tau_{\ast}: K_{\ast}\left(C^{\ast}_{LL}(P_{k}(N)\times Y_i, A)^{\Gamma}\right)\rightarrow K_{\ast}\left(C^{\ast}_{LL}(P_{k}(N)\times Y, A; Y_i)^{\Gamma}\right)$ is an isomorphism by Lemma \ref{Lem-Liphomotopy-twoloc}.
	
	Moreover, $C^{\ast}_{LL}(P_{k}(N)\times Y, A; Y_i)^{\Gamma}$ is an ideal of $C^{\ast}_{LL}(P_{k}(N)\times Y, A)^{\Gamma}$ for $i=1,2$ and $C^{\ast}_{LL}(P_{k}(N)\times Y, A; Y_1)^{\Gamma}+C^{\ast}_{LL}(P_{k}(N)\times Y, A; Y_2)^{\Gamma}=C^{\ast}_{LL}(P_{k}(N)\times Y, A)^{\Gamma}$ due to the third condition in Definition \ref{Def-two-localg}. In addition, we have $C^{\ast}_{LL}(P_{k}(N)\times Y, A; Y_1)^{\Gamma} \cap C^{\ast}_{LL}(P_{k}(N)\times Y, A; Y_2)^{\Gamma}=C^{\ast}_{LL}(P_{k}(N)\times Y, A; Y_1\cap Y_2)^{\Gamma}$. Thus we complete the proof by using of Mayer--Vietoris six-term sequence in the $K$-theory for $C^{\ast}$-algebras.
\end{proof}

\subsection{Equivariant $K$-homology for isometric semi-direct products}
In this subsection, we first build a connection between $C^{\ast}_{L}(P_{k}(N)\times P_{l}(G), A)^{N\rtimes G}$ and $C^{\ast}_{L}(P_{l}(G), C^{\ast}_{L}(P_{k}(N), A)^{N})^{G}$ on the $K$-theory level for the isometric semi-direct product $N\rtimes G$ through $C^{\ast}_{LL}(P_{k}(N)\times P_{l}(G), A)^{N\rtimes G}$. Then we reduce the Baum-Connes assembly map with coefficients for isometric semi-direct products to two partial Baum-Connes assembly maps.

For any $k,l\geq 0$, define a $\ast$-homomorphism
  $$\varphi: C^{\ast}_{L}(P_{k}(N)\times P_{l}(G), A)^{N\rtimes G} \rightarrow C^{\ast}_{LL}(P_{k}(N)\times P_{l}(G), A)^{N\rtimes G},$$
by
$$\varphi(u)(s,t)=u(s+t),$$
for any $u\in C^{\ast}_{L}(P_{k}(N)\times P_{l}(G), A)^{N\rtimes G}$.

\begin{proposition}\label{Prop-loc-twoloc}
	For any $k,l\geq 0$, the $\ast$-homomorphism $\varphi$ induces the following isomorphism
	$$\varphi_{\ast}: K_{\ast}\left(C^{\ast}_{L}(P_{k}(N)\times P_{l}(G), A)^{N\rtimes G}\right) \rightarrow K_{\ast}\left(C^{\ast}_{LL}(P_{k}(N)\times P_{l}(G), A)^{N\rtimes G}\right).$$
\end{proposition}
\begin{proof}
	Since the metric on $G$ is proper, the Rips complex $P_{l}(G)$ is finite-dimensional for any $l\geq 0$. Let $P_{l}(G)^{n}$ be the $n$-dimensional skeleton of $P_{l}(G)$. We will prove the proposition by induction on $n$.
	
	Firstly, for $n=0$, we have that $P_{l}(G)^{0}=G$. Define a homomorphism $\tau: C^{\ast}_{L}(P_{k}(N), A)^{N\rtimes G} \rightarrow C^{\ast}_{L}(P_{k}(N)\times G, A)^{N\rtimes G}$ by
	$$\tau(u)(t)_{(x,e), (x',e)}=u(t)_{x,x'} \: \text{and} \: \tau(u)(t)_{(x,g), (x',g')}=0 \:\text{if $g\neq g'$},$$
	for any $x,x'\in P_{k}(N)$ and $g,g'\in G$. Then $\tau$ induces an isomorphism
	\begin{equation}\label{equ-proploc-1}
	K_{\ast}\left(C^{\ast}_{L}(P_{k}(N), A)^{N\rtimes G}\right) \cong K_{\ast}\left(C^{\ast}_{L}(P_{k}(N)\times G, A)^{N\rtimes G}\right).
	\end{equation}
	Moreover, by a similar argument as above, we have that
	\begin{equation}\label{equ-proploc-2}
	K_{\ast}\left(C_{ub}([0,\infty), C^{\ast}_{L}(P_{k}(N), A)^{N\rtimes G})\right) \cong K_{\ast} \left( C^{\ast}_{LL}(P_{k}\times G, A)^{N\rtimes G} \right).
	\end{equation}
	In addition, since $C^{\ast}_{L}(P_{k}(N), A)^{N\rtimes G}$ is quasi-stable. By Lemma \ref{Lem-elementary}, we have that 
	\begin{equation}\label{equ-proploc-3}
	K_{\ast}\left(C_{ub}([0,\infty), C^{\ast}_{L}(P_{k}(N), A)^{N\rtimes G})\right) \cong K_{\ast}\left(C^{\ast}_{L}(P_{k}(N), A)^{N\rtimes G}\right).
	\end{equation}
	Three isomorphic maps in (\ref{equ-proploc-1}-\ref{equ-proploc-3}) are commutative with $\varphi_{\ast}$, thus the proposition holds for the case of $n=0$. 
	
	Secondly, we assume that the proposition holds for the case of $n=m-1$. Let $c(\Delta)$ be the center of a $m$-dimensional simplex $\Delta$ in $P_{l}(G)$. Define
	$$\Delta_1=\{y\in \Delta: d(y, c(\Delta))\leq 1/10\};\:\: \Delta_2=\{y\in \Delta: d(y, c(\Delta))\geq 1/10\}.$$
	Let 
	$$Y_1=\bigcup \left\{ \Delta_1: \text{dim}(\Delta)=m\right\};\:\:Y_2=\bigcup \left\{ \Delta_2: \text{dim}(\Delta)=m\right\}.$$
	Then $Y_1$, $Y_2$ and $Y_1 \cap Y_2$ are strongly Lipschitz $G$-homotopy equivalent to $\{c(\Delta): \text{dim}(\Delta)=m\}$, $P_{l}(G)^{m-1}$ and the disjoint union of the boundaries of all $m$-dimensional simplices of $P_{l}(G)$, respectively. Thus, $\varphi_{\ast}$ are isomorphic for $Y_1$, $Y_2$ and $Y_1 \cap Y_2$ by Lemma \ref{Lem-Liphtp}, Lemma \ref{Lem-Liphomotopy-twoloc} and the assumption. Moreover, $\{Y_1, Y_2\}$ is a uniformly excisive $G$-cover of $P_{l}(G)^{m}$. Therefore, by Lemma \ref{Lem-equiMV}, Lemma \ref{Lem-MV-twoloc} and the five lemma, the proposition holds for the case of $n=m$.
\end{proof}

Let $A$ be an $(N\rtimes G)$-$C^{\ast}$-algebra equipped with the $(N\rtimes G)$-action $\alpha$. Then $A$ is also an $N$-$C^{\ast}$-algebra because $N$ is a subgroup of $N\rtimes G$. Thus, we have an equivariant localization algebra $C^{\ast}_{L}(P_{k}(N), A)^{N}$. For any $g\in G$, define a unitary $U_{g}$ on $\ell^2(Z_{k}(N))\otimes H_{A}\otimes \ell^2(N)\otimes H$ by
$$U_{g}(\delta_{x}\otimes h_{A}\otimes \delta_{n}\otimes h)=\delta_{\beta_{g}(x)}\otimes U^{\alpha}_{(e,g)}(h_{A})\otimes \delta_{\beta_{g}(n)}\otimes h,$$
where $\beta_{g}(x)=\sum_{i=1}^{l}t_i\beta_{g}(n_i)$ for any $x=\sum_{i=1}^{l}t_in_i\in Z_{k}(N)$ and $(H_A, U^{\alpha})$ is a covariant representation of $(A, N\rtimes G, \alpha)$. Then these unitaries induce an action of $G$ on $C^{\ast}_{L}(P_{k}(N), A)^{N}$ defined by 
$$v \mapsto (I\otimes U_g) v (I\times U^{\ast}_{g}),$$ 
for any $g\in G$ and $v\in C^{\ast}_{L}(P_{k}(N), A)^{N}$. Thus we obtain an equivariant localization algebra $C^{\ast}_{L}(P_{l}(G), C^{\ast}_{L}(P_{k}(N), A)^{N})^{G}$.

Let $\C[P_{l}(G), \C_{L}[P_{k}(N), A]^{N}]^{G}$ be the $\ast$-subalgebra of $\C[P_{l}(G), C_{L}(P_{k}(N), A)^{N}]^{G}$ consisting of all operators $T$ such that $\chi_{K}T, T\chi_{K}\in \mathcal{K}(\ell^2(Z_{l}(G))\otimes H \otimes \ell^2(G))\otimes \C_{L}[P_{k}(N), A]^{N}$ for any compact subset $K\subseteq P_{l}(G)$. Similarly, we can define $\C_{L}[P_{l}(G), \C_{L}[P_{k}(N), A]^{N}]^{G}$ to be the $\ast$-subalgebra of $\C_{L}[P_{l}(G), C_{L}(P_{k}(N), A)^{N}]^{G}$ consisting of all elements $u$ satisfying that $u(t)\in \C[P_{l}(G), \C_{L}[P_{k}(N), A]^{N}]^{G}$ for all $t\geq 0$.

\begin{lemma}\label{Lem-dense-subalg}
	Any element $u\in C^{\ast}_{L}(P_{l}(G), C^{\ast}_{L}(P_{k}(N), A)^{N})^{G}$ can be approximated by a sequence of elements in $\C_{L}[P_{l}(G), \C_{L}[P_{k}(N), A]^{N}]^{G}$.
\end{lemma}
\begin{proof}
	we can assume $u\in \C_{L}[P_{l}(G), C^{\ast}_{L}(P_{k}(N), A)^{N}]^{G}$. Then for any $\varepsilon>0$, there exists $\delta>0$ such that $\|u(t')-u(t)\|<\varepsilon$ for any $|t'-t|<\delta$. Choose $0=t_0<t_1<\cdots<t_n<\cdots$ with $|t_{i+1}-t_i|<\delta$ for any $i=0,1,\cdots$. For each $u_{t_i}$, since $P_{l}(G)/G$ is compact, there exists an operator $v_{t_i}\in \C[P_{l}(G), \C_{L}[P_{k}(N), A]^{N}]^{G}$ such that $\prop_{P_{l}(G)}(v_{t_i})=\prop_{P_{l}(G)}(u_{t_i})$ and $\|v_{t_i}-u_{t_i}\|<\varepsilon$. Let
	$$v(t)=\frac{t-t_{i}}{t_{i+1}-t_i}v_{t_{i+1}}+\frac{t_{i+1}-t}{t_{i+1}-t_i}v_{t_{i}},\:\text{for}\:t\in[t_i, t_{i+1}].$$
	Then $v$ belongs to $\C_{L}[P_{l}(G), \C_{L}[P_{k}(N), A]^{N}]^{G}$ and $\|v-u\|<5\varepsilon$. Thus, we complete the proof.
\end{proof}

Recall that $C^{\ast}_{L}(P_{l}(G), C^{\ast}_{L}(P_{k}(N), A)^{N})^{G}$ acts on 
$$L^{2}([0,\infty))\otimes \ell^2(Z_{l}(G))\otimes \left(L^2([0,\infty))\otimes \ell^2(Z_{k}(N))\otimes H_A\otimes \ell^2(N)\otimes H\right) \otimes \ell^2(G)\otimes H,$$ 
and $C^{\ast}_{LL}(P_{k}(N)\times P_{l}(G), A)^{N\rtimes G}$ acts on
$$L^{2}([0,\infty))\otimes L^2([0,\infty))\otimes \ell^2(Z_{l}(G))\otimes \ell^2(Z_{k}(N))\otimes H_A\otimes \ell^2(N) \otimes \ell^2(G)\otimes H.$$

Let $V: H\otimes H \rightarrow H$ be a unitary. Then $V$ induces a $\ast$-homomorphism
$$\psi: C^{\ast}_{L}(P_{l}(G), C^{\ast}_{L}(P_{k}(N), A)^{N})^{G} \rightarrow C^{\ast}_{LL}(P_{k}(N)\times P_{l}(G), A)^{N\rtimes G},$$
defined by 
$$\psi(u)(s,t)=(I\otimes V)\left(u(t)(s)\right)(I\otimes V^{\ast}),$$ 
for any $t\mapsto u(t)\in C^{\ast}_{L}(P_{l}(G), C^{\ast}_{L}(P_{k}(N), A)^{N})^{G}$. Notice that $\psi$ is well-defined by Lemma \ref{Lem-dense-subalg}.

\begin{proposition}\label{Prop-oneloc-twoloc}
	For any $k,l\geq 0$, the homomorphism $\psi$ induces an isomorphism:
	$$\psi_{\ast}: K_{\ast}\left(C^{\ast}_{L}\left(P_{l}(G), C^{\ast}_{L}(P_{k}(N), A)^{N} \right)^{G}\right) \rightarrow K_{\ast}\left(C^{\ast}_{LL}\left(P_{k}(N)\times P_{l}(G), A\right)^{N\rtimes G}\right).$$
\end{proposition}
\begin{proof}
	Let $P_{l}(G)^{n}$ be the $n$-dimensional skeleton of $P_{l}(G)$. We will prove the proposition by induction on $n$.
	
	When $n=0$, we have $P_{l}(G)^{0}=G$. Then the left-hand side of $\psi_{\ast}$ is isomorphic to $K_{\ast}\left(C_{ub}\left( [0,\infty), C^{\ast}_{L}(P_{k}(N), A)^{N}\otimes \K(\ell^2(G))\otimes \K(H) \right)\right)$ and the right-hand side of $\psi_{\ast}$ is isomorphic to $K_{\ast}\left(C_{ub}\left( [0,\infty),C^{\ast}_{L}(P_{k}(N), A\otimes \K(\ell^2(G)))^{N} \right)\right)$. Thus, $\psi_{\ast}$ is an isomorphism by Lemma \ref{Lem-elementary}.
	
	Let $c(\Delta)$ be the center of any $m$-dimensional simplex $\Delta$ in $P_{l}(G)$ and $c(P_{l}(G))=\{c(\Delta): \text{dim}(\Delta)=m\}$. 
	Let $F$ be the stabilizer subgroup of $c(\Delta_{e})$, where $\Delta_{e}$ is a fixed $m$-dimensional simplex containing $e$. Then by Lemma \ref{Lem-elementary}, we have that
	$$K_{\ast}\left(C^{\ast}_{L}\left(c(P_{l}(G)), C^{\ast}_{L}(P_{k}(N), A)^{N} \right)^{G}\right) \cong K_{\ast}\left(C^{\ast}_{L}(P_{k}(N), A)^{N}\rtimes_{r} F \right),$$
	and 
	$$K_{\ast}\left(C^{\ast}_{LL}\left(P_{k}(N)\times c(P_{l}(G)), A\right)^{N\rtimes G}\right) \cong K_{\ast}\left(C^{\ast}_{L}\left(P_{k}(N), A\otimes \K(\ell^2(G/F))\right)^{N\rtimes F}\right),$$
	where the action of $N\rtimes F$ on $\K(\ell^2(G/F))$ is trivial.
	Moreover, since $F$ is finite, we also have that
	$$K_{\ast}\left(C^{\ast}_{L}(P_{k}(N), A)^{N}\rtimes_{r} F \right) \cong K_{\ast}\left(C^{\ast}_{L}(P_{k}(N), A)^{N\rtimes F} \right).$$
	Thus, we have the following isomorphism:
	$$\psi_{\ast}: K_{\ast}\left(C^{\ast}_{L}\left(c(P_{l}(G)), C^{\ast}_{L}(P_{k}(N), A)^{N} \right)^{G}\right) \rightarrow K_{\ast}\left(C^{\ast}_{LL}\left(P_{k}(N)\times c(P_{l}(G)), A\right)^{N\rtimes G}\right).$$
	
	Next, we assume that $\psi_{\ast}$ is an isomorphism for the case of $n=m-1$. Define
	$$\Delta_1=\{y\in \Delta: d(y, c(\Delta))\leq 1/10\};\:\: \Delta_2=\{y\in \Delta: d(y, c(\Delta))\geq 1/10\}.$$
	Let 
	$$Y_1=\bigcup \left\{ \Delta_1: \text{dim}(\Delta)=m\right\};\:\:Y_2=\bigcup \left\{ \Delta_2: \text{dim}(\Delta)=m\right\}.$$
	Then $Y_1$, $Y_2$ and $Y_1 \cap Y_2$ are strongly Lipschitz $G$-homotopy equivalent to $c(P_{l}(G))$, $P_{l}(G)^{m-1}$ and the disjoint union of the boundaries of all $m$-dimensional simplices of $P_{l}(G)$, respectively. Thus, $\psi_{\ast}$ are isomorphic for $Y_1$, $Y_2$ and $Y_1 \cap Y_2$ by Lemma \ref{Lem-Liphtp}, Lemma \ref{Lem-Liphomotopy-twoloc}. Moreover, $\{Y_1, Y_2\}$ is a uniformly excisive $G$-cover of $P_{l}(G)^{m}$. Therefore, by Lemma \ref{Lem-equiMV}, Lemma \ref{Lem-MV-twoloc} and the five lemma, the proposition holds for the case of $n=m$.
\end{proof}

Combining Proposition \ref{Prop-loc-twoloc} with Proposition \ref{Prop-oneloc-twoloc}, we obtain the following commutative diagram for any $k,l\geq 0$.

\begin{displaymath}
	\xymatrix{
	 K_{\ast}\left(C^{\ast}_{L}\left(P_{k}(N)\times P_{l}(G), A\right)^{N\rtimes G}\right) \ar[r]^{e_{\ast}} \ar[d]^{\cong}_{\varphi_{\ast}}  
	 & K_{\ast}\left(C^{\ast}\left(P_{k}(N)\times P_{l}(G), A\right)^{N\rtimes G}\right),
	\ar[d]^{=} \\
	K_{\ast}\left(C^{\ast}_{LL}\left(P_{k}(N)\times P_{l}(G), A\right)^{N\rtimes G}\right) \ar[r]^{e_{\ast}}
	& K_{\ast}\left(C^{\ast}\left(P_{k}(N)\times P_{l}(G), A\right)^{N\rtimes G}\right),
	 \\
	K_{\ast}\left(C^{\ast}_{L}\left(P_{l}(G), C^{\ast}_{L}(P_{k}(N), A)^{N}\right)^{G}\right) \ar[r]^{e_{\ast}} \ar[u]_{\cong}^{\psi_{\ast}}
	& K_{\ast}\left(C^{\ast}\left(P_{l}(G), C^{\ast}(P_{k}(N), A)^{N}\right)^{G}\right).  \ar[u]_{\cong}
    }
\end{displaymath}

Thus, we have the following proposition by Lemma \ref{Cor-red-conj}.

\begin{proposition}\label{Prop-reduction-Conj}
	Let $N\rtimes G$ be an isometric semi-direct product and $A$ be an $(N\rtimes G)$-$C^{\ast}$-algebra. Then $N\rtimes G$ satisfies SNC, SAC and BCC with coefficients in $A$, respectively, if and only if the following evaluation at zero map 
	\begin{equation}\label{equ-red}
	e_{\ast}: \lim_{l,k\rightarrow \infty} K_{\ast}\left(C^{\ast}_{L}\left(P_{l}(G), C^{\ast}_{L}(P_{k}(N), A)^{N}\right)^{G}\right) \rightarrow \lim_{l,k\rightarrow \infty} K_{\ast}\left(C^{\ast}\left(P_{l}(G), C^{\ast}(P_{k}(N), A)^{N}\right)^{G}\right)
	\end{equation}
	is injective, surjective and isomorphic, respectively.
\end{proposition}

We can divide (\ref{equ-red}) into the following two homomorphisms:
\begin{equation}\label{equ-red-1}
e_{\ast}: \lim_{l,k\rightarrow \infty} K_{\ast}\left(C^{\ast}_{L}\left(P_{l}(G), C^{\ast}_{L}(P_{k}(N), A)^{N}\right)^{G} \right) \rightarrow \lim_{l,k\rightarrow \infty} K_{\ast}\left(C^{\ast}\left(P_{l}(G), C^{\ast}_{L}(P_{k}(N), A)^{N}\right)^{G} \right),
\end{equation}
and
\begin{equation}\label{equ-red-2}
e_{\ast}: \lim_{l,k\rightarrow \infty} K_{\ast}\left( C^{\ast}\left(P_{l}(G), C^{\ast}_{L}(P_{k}(N), A)^{N}\right)^{G} \right) \rightarrow \lim_{l,k\rightarrow \infty} K_{\ast}\left(C^{\ast}\left(P_{l}(G), C^{\ast}(P_{k}(N), A)^{N}\right)^{G} \right).
\end{equation}

In addition, since the natural embedding map from $G$ to $P_{l}(G)$ is a coarse $G$-equivalence (cf. \cite[Proposition 7.2.11]{WillettYu-Book}) for any $l\geq 0$, we can simplify (\ref{equ-red-2}) to the following homomorphism by Corollary \ref{Cor-coarseequi-K}:
\begin{equation}\label{equ-red-2'}
e_{\ast}: \lim_{k\rightarrow \infty} K_{\ast}\left( C^{\ast}\left(G, C^{\ast}_{L}(P_{k}(N), A)^{N}\right)^{G} \right) \rightarrow \lim_{k\rightarrow \infty} K_{\ast}\left(C^{\ast}\left(G, C^{\ast}(P_{k}(N), A)^{N}\right)^{G} \right). 
\end{equation}

\section{The third main theorem} \label{Sec-5}

In this section, we investigate when $e_{\ast}$ in (\ref{equ-red-2'}) is injective, surjective and isomorphic, and complete the proof of the third main theorem. we assume that $N\rtimes G$ is an isometric semi-direct product and $A$ be an $(N\rtimes G)$-$C^{\ast}$-algebra in this whole section. 

\subsection{Two equivariant localization algebras along one direction}

Recall that $C^{\ast}_{L, P_{k}(N)}(P_{k}(N)\times G, A)^{N\rtimes G}$ denotes the equivariant localization algebra along $P_{k}(N)$ with coefficients in $A$ of $P_{k}(N)\times G$ defined in Definition \ref{Def-loc-along}.
\begin{definition}\label{Def-uniloc-along}
  Define $C^{\ast}_{L, P_{k}(N),u}(P_{k}(N)\times G, A)^{N\rtimes G}$ to be a $C^{\ast}$-subalgebra of $C^{\ast}_{L, P_{k}(N)}(P_{k}(N)\times G, A)^{N\rtimes G}$ generated by all elements $u$ with
  $$\sup_{t\in[0,\infty)}\prop_{G}(u(t))<\infty.$$
\end{definition}

We also have the following lemmas similar to Lemma \ref{Lem-Liphomotopy-localong} and Lemma \ref{Lem-MV-localong}.

\begin{lemma}\label{Lem-Liphomotopy-unilocalong}
  Let $X_1$ and $X_2$ be two $(N\rtimes G)$-invariant subspaces of $P_{k}(N)$. If $X_1$ is strongly Lipschitz $(N\rtimes G)$-homotopy equivalent to $X_2$, then $K_{\ast}\left(C^{\ast}_{L, X_1, u}(X_1\times G, A)^{N\rtimes G}\right)$ is naturally isomorphic to $K_{\ast}\left(C^{\ast}_{L, X_2, u}(X_2\times G, A)^{N\rtimes G}\right)$.
\end{lemma}

\begin{lemma} \label{Lem-MV-unilocalong}
	Let $X$ be an $(N\rtimes G)$-invariant subspace of $P_{k}(N)$. If $\{X_1, X_2\}$ is a uniformly excisive $(N\rtimes G)$-cover of $X$, then the following Mayer--Vietoris six-term exact sequence holds
	$$\scriptsize\xymatrix{
		K_0(C^{\ast}_{L, X_{1,2}, u}(X_{1,2}\times G, A)^{\Gamma}) \ar[r] & 
		\oplus_{i=1,2} K_0(C^{\ast}_{L, X_i, u}(X_i\times G, A)^{\Gamma}) \ar[r] & 
		K_0(C^{\ast}_{L, X, u}(X\times G, A)^{\Gamma})\ar[d] \\
		K_1(C^{\ast}_{L, X, u}(X\times G, A)^{\Gamma})\ar[u] &
		\oplus_{i=1,2} K_1(C^{\ast}_{L, X_i, u}(X_i\times G, A)^{\Gamma})\ar[l] &
		K_1(C^{\ast}_{L, X_{1,2}, u}(X_{1,2}\times G, A)^{\Gamma}),\ar[l] 
	}$$
	where $\Gamma=N\rtimes G$ and $X_{1,2}=X_1\cap X_2$.
\end{lemma}

Next, we explore the relationship between $C^{\ast}_{L, P_{k}(N)}(P_{k}(N)\times G, A)^{N\rtimes G}$ and $C^{\ast}_{L, P_{k}(N), u}(P_{k}(N)\times G, A)^{N\rtimes G}$ on the $K$-theory level by using G. Yu's quantitative $K$-theory. Firstly, we need recall some basic lemmas for the quantitative $K$-theory (please see \cite{OyonoYu2015}\cite{Yu1998}\cite{Zhang-quan-CBC}).

\begin{definition}(\cite[Definition 1.1]{OyonoYu2015})\label{Def-filtration}
	Let $A$ be a $C^{\ast}$-algebra. A \textit{filtration} of $A$ is a family of closed linear subspaces $(A_r)_{r>0}$ such that
	\begin{itemize}
		\item $A_r\subseteq A_{r'}$, if $r\leq r'$;
		\item $A_r$ is closed by involution;
		\item $A_rA_{r'}\subseteq A_{r+r'}$;
		\item $\cup_{r>0}A_r$ is dense in $A$. 
	\end{itemize}
	If $A$ is unital, we also require the unit of $A$ belongs to $A_r$ for each $r>0$. A $C^{\ast}$-algebra equipped with a filtration is called a \textit{filtered $C^{\ast}$-algebra}.
\end{definition} 

\begin{example}\label{Ex-Roe-fil}
	Let $\Gamma$ be a countable discrete group, $X$ be a $\Gamma$-space and $A$ be a $\Gamma$-$C^{\ast}$-algebra. Then the equivariant Roe algebra $C^{\ast}(X, A)^{\Gamma}$ is a filtered $C^{\ast}$-algebra equipped with the following filtration
	$$\C[X, A]^{\Gamma}_{r}=\{T\in \C[X, A]^{\Gamma}: \prop(T)\leq r\},$$
	for any $r>0$.
\end{example}

\begin{definition}(\cite{OyonoYu2015})\label{Def-quasielement}
	Let $A$ be a unital filtered $C^{\ast}$-algebra with the unit $I$ and $0<\ep<1/4$, $r>0$.
	\begin{enumerate}
		\item An element $p$ in $A$ is called an \textit{$(\ep, r)$-projection}, if $p\in A_r$, $p=p^{\ast}$ and $\|p^2-p\|<\ep$. Denote $P^{\ep, r}(A)$ to be the set of all $(\ep, r)$-projections in $A$.
		\item An element $u$ in $A$ is called an \textit{$(\ep, r)$-unitary}, if $u\in A_r$ and $\max\{\|uu^{\ast}-I\|, \|u^{\ast}u-I\|\}<\ep$. Denote $U^{\ep, r}(A)$ to be the set of all $(\ep, r)$-unitaries in $A$.
	\end{enumerate}
\end{definition}

\begin{remark}\label{Rmek-almost-proj-proj}
	We have the following two comments for the above two notions.
	\begin{enumerate}
		\item For an $(\ep, r)$-projection $p$ in a filtered $C^{\ast}$-algebra $A$, the spectrum of $p$ is contained in $(\frac{1-\sqrt{1+4\ep}}{2}, \frac{1-\sqrt{1-4\ep}}{2})\cup (\frac{1+\sqrt{1-4\ep}}{2}, \frac{1+\sqrt{1+4\ep}}{2})$. Let $\chi_0$ be a real-valued continuous function on $\R$ which satisfies $\chi_0(x)=0$ if $x\leq \frac{1-\sqrt{1-4\ep}}{2}$ and $\chi_0(x)=1$ if $x\geq \frac{1+\sqrt{1-4\ep}}{2}$. Then by the continuous functional calculus, $\chi_0(p)$ is a projection in $A$ and $\|\chi_0(p)-p\|<2\ep$. 
		\item Any $(\ep, r)$-unitary $u$ in $A$ is invertible. Let $\chi_1(u)=u(u^{\ast}u)^{-1/2}$, then $\chi_1(u)$ is a unitary that satisfies $\|\chi_1(u)-u\|<\ep$.
	\end{enumerate}
\end{remark}

Let $A$ be a filtered $C^{\ast}$-algebra. For a positive integer $n$, the matrix algebra $M_n(A)$ is a filtered $C^{\ast}$-algebra equipped with a filtration $(M_n(A_r))_{r>0}$. Let $C([0,1],A)$ be the $C^{\ast}$-algebra of all continuous functions from $[0,1]$ to $A$. Then $C([0,1],A)$ is a filtered $C^{\ast}$-algebra equipped with a filtration $(C([0,1],A_r))_{r>0}$.

For $0<\ep<1/4$, $r>0$, let $P^{\ep, r}_{\infty}(A)=\lim_{n\rightarrow\infty}P^{\ep, r}(M_{n}(A))$ and $U^{\ep, r}_{\infty}(A)=\lim_{n\rightarrow\infty}U^{\ep, r}(M_{n}(A))$. Two elements $p$ and $p'$ in $P^{\ep, r}_{\infty}(A)$ are called to be \textit{$(\ep, r)$-homotopic}, if there exists an element $p$ in $P^{\ep, r}_{\infty}(C([0,1],A))$ such that $p_0=p$ and $p_1=p'$. In addition, two elements $u$ and $u'$ in $U^{\ep, r}_{\infty}(A)$ are called to be \textit{$(\ep, r)$-homotopic}, if there exists an element $u$ in $U^{\ep, r}_{\infty}(C([0,1],A))$ such that $u_0=u$ and $u_1=u'$. 

\begin{definition}\cite[Definition 2.7]{Zhang-quan-CBC}\label{Def-quan-K}
	Let $A$ be a filtered $C^{\ast}$-algebra and $0<\ep<1/4$, $r>0$. Let $A^{+}=A\oplus \C$ be the one-point compactification of $A$ with a map $\pi: A^{+}\rightarrow \C$ defined by $(a, c)\mapsto c$.
	\begin{enumerate}
		\item If $A$ is unital, define
		$$K^{\ep, r}_0(A)=\{(p)_{\ep, r}-(q)_{\ep, r}: p, q\in P^{\ep, r}_{\infty}(A)\}/\sim.$$
		In addition, if $A$ is non-unital, define
		$$K^{\ep, r}_0(A)=\{(p)_{\ep, r}-(q)_{\ep, r}: p, q\in P^{\ep, r}_{\infty}(A^{+}),\:\:\textrm{dim}\:\chi_0(\pi(p))=\textrm{dim}\:\chi_0(\pi(q))\}/\sim.$$
		Where $(p)_{\ep,r}-(q)_{\ep, r}\sim (p')_{\ep,r}-(q')_{\ep, r}$ if and only if there exists a positive integer $k$ such that $p\oplus q'\oplus I_k$ is $(\ep, r)$-homotopic to $p'\oplus q\oplus I_k$. Denote by $[p]_{\ep, r}-[q]_{\ep, r}$ the equivalence class of $(p)_{\ep, r}-(q)_{\ep, r}$ modulo $\sim$. $[p]_{\ep, r}$ and $-[q]_{\ep, r}$ are the abbreviations of $[p]_{\ep, r}-[0]_{\ep, r}$ and $[0]_{\ep, r}-[q]_{\ep, r}$, respectively.
		\item Let $\Tilde{A}= A$ if $A$ is unital and $\Tilde{A}=A^{+}$ if $A$ is non-unital. Define
		$$K^{\ep, r}_1(A)=U^{\ep, r}_{\infty}(\Tilde{A})/\sim,$$
		where two $(\ep, r)$-unitaries $u\sim u'$ if and only if $u$ is $(3\ep, 2r)$-homotopic to $u'$. Denote by $[u]_{\ep, r}$ the equivalence class of $u$ modulo $\sim$.
	\end{enumerate}
\end{definition}

The quantitative $K$-theory is a refinement of the $K$-theory. Besides, the $K$-theory can be approached by quantitative $K$-theory as $r\rightarrow \infty$. More precisely, for a filtered $C^{\ast}$-algebra $A$ and $0<\ep<1/4$, $r>0$, define (see Remark \ref{Rmek-almost-proj-proj})
$$\iota^{\ep, r}_0: K^{\ep, r}_0(A)\rightarrow K_0(A)\:\: [p]_{\ep, r}-[q]_{\ep, r}\mapsto [\chi_0(p)]-[\chi_0(q)],$$
and 
$$\iota^{\ep, r}_1: K^{\ep, r}_1(A)\rightarrow K_1(A)\:\: [u]_{\ep, r}\mapsto [\chi_1(u)].$$                                                                               
Then we have the following lemma.

\begin{lemma}\cite[Lemma 2.14]{Zhang-quan-CBC}\label{quan-K-and-K}
	Let $A$ be a filtered $C^{\ast}$-algebra equipped with a filtration $(A_r)_{r>0}$.
	\begin{enumerate}
		\item For any $0<\ep<1/4$ and any $w\in K_{\ast}(A)$, there exists a positive number $r$ and $z\in K^{\ep, r}_{\ast}(A)$ such that $\iota^{\ep, r}_{\ast}(z)=w$.
		\item For any $0<\ep<1/64$ and $r>0$. If an element $z\in K^{\ep, r}_{\ast}(A)$ satisfies $\iota^{\ep, r}_{\ast}(z)=0$ in $K_{\ast}(A)$, then there exists $r'\geq r$ such that $\iota^{\ep, 16\ep, r, r'}_{\ast}(z)=0$ in $K^{16\ep, r'}_{\ast}(A)$.
	\end{enumerate}
\end{lemma}

\begin{definition}\label{uniformproduct}\cite[Definition 5.14]{OyonoYu2019}
	Let $(A_i)_{i\in \mathcal{I}}$ be a family of filtered $C^{\ast}$-algebras equipped with filtrations $(A_{i,r})_{i\in \mathcal{I}, r>0}$. The \textit{uniform product} of $(A_i)_{i\in \mathcal{I}}$, denoted by $\prod^{uf}_{i\in \mathcal{I}} A_i$, is defined to be the norm closure of $\cup_{r>0}(\prod_{i\in \mathcal{I}} A_{i, r})$ in $\prod_{i\in \mathcal{I}} A_i$.
\end{definition}

\begin{remark}
	The uniform product $\prod^{uf}_{i\in \mathcal{I}} A_i$ is a filtered $C^{\ast}$-algebra equipped with a filtration $(\prod_{i\in \mathcal{I}} A_{i, r})_{r>0}$.
\end{remark}

\begin{definition}\cite[Definition 2.23]{Zhang-quan-CBC}\label{Def-quasi-stable-fil}
  A family of filtered $C^{\ast}$-algebras $(A_i)_{i\in \mathcal{I}}$ equipped with filtrations $(A_{i,r})_{i\in \mathcal{I}, r>0}$ is called \textit{uniformly quasi-stable}, if for any positive integer $n$, there exists a family of isometries $(w_i)_{i\in \mathcal{I}}\in \mathcal{M}(M_n(A_i))_{i\in \mathcal{I}}$ satisfying that
	\begin{enumerate}
		\item $w_iw^{\ast}_i=e_{11}=I\oplus 0_{n-1}$ for any $i\in \mathcal{I}$;
		\item $(w_ia_i)_{i\in \mathcal{I}}$ and $(a_iw_i)_{i\in \mathcal{I}}$ are in $M_n(A_{i,r})_{i\in \mathcal{I}}$ for any element $(a_i)_{i\in \mathcal{I}}\in M_n(A_{i,r})_{i\in \mathcal{I}}$ and $r>0$.
	\end{enumerate}
\end{definition}

A family of equivariant Roe algebras $(C^{\ast}(X_i, A_i)^{\Gamma})_{i\in \mathcal{I}}$ equipped with filtrations $(\C[X_i, A_i]^{\Gamma}_{r})_{i\in \mathcal{I}, r>0}$ is uniformly quasi-stable (please see \cite[Lemma 3.6]{Zhang-quan-CBC}).

\begin{lemma}\cite[Corollary 2.27]{Zhang-quan-CBC}\label{Lem-quanK-prod}
	If $(A_i)_{i\in \mathcal{I}}$ is a family of uniformly quasi-stable filtered $C^{\ast}$-algebras, then the natural homomorphisms 
	$$\pi_j:\prod^{uf}_{i\in \mathcal{I}} A_i\rightarrow A_j, \:\: (a_i)_{i\in \mathcal{I}}\mapsto a_j,$$
	induce the following isomorphism
	$$\Pi \pi_{i, \ast}: K^{\ep, r}_{\ast}(\prod^{uf}_{i\in \mathcal{I}} A_i)\rightarrow \prod_{i\in \mathcal{I}}K^{\ep, r}_{\ast}(A_i),$$
	for any $0<\ep<1/4$, $r>0$.
\end{lemma}

Let $A$ be a filtered $C^{\ast}$-algebra. Define $C_{ubf}\left([0,\infty), A\right)$ to be the norm closure of the algebra consisting of all bounded and uniformly continuous functions $u: [0,\infty) \rightarrow A$ satisfying that there exists $r>0$ such that $u(t)\in A_r$ for all $t\in [0,\infty)$. Then we have the following lemma inspired by Lemma \ref{Lem-elementary}.

\begin{lemma}\label{Lem-fil-eva}
	If $(A)_{i\in \mathcal{I}}$ is a family of uniformly quasi-stable filtered $C^{\ast}$-algebras. Then the evaluation at zero map 
	$$e_0: C_{ubf}\left( [0,\infty), A \right) \rightarrow A, u\mapsto u(0),$$
induces an isomorphism on the $K$-theory.
\end{lemma} 
\begin{proof}
It is sufficient to prove that $K_{\ast}\left( \text{ker}(e_0) \right)=0$. Let
$$X=\bigcup_{n\geq 1, \text{odd}} [n,n+1],\:\: Y=\bigcup_{n\geq 2,\text{even}} [n,n+1].$$
Then we have the following pullback diagram
 $$\xymatrix{
   \text{ker}(e_0) \ar[r] \ar[d] &  C_{ubf}\left(X, A\right) \ar[d] \\
   \left(C_{ubf}\left([0,1], A\right)\cap \text{ker}(e_0)\right)\oplus C_{ubf}\left(Y, A\right) \ar[r] &  \prod_{n\geq 1}^{uf} A. 
 }$$
Thus, we have the following Mayer--Vietoris six-term exact sequence (please see \cite[Proposition 2.7.15]{WillettYu-Book})

$$\cdots \rightarrow K_{\ast}\left(\text{ker}(e_0)\right) \rightarrow K_{\ast}\left(C_{ubf}\left( X, A \right)\right) \oplus K_{\ast}\left(C_{ubf}\left( Y, A \right)\right) \rightarrow K_{\ast}\left( \prod_{n\geq 1}^{uf} A \right) \rightarrow \cdots.$$
Moreover, the natural embedding maps 
$$\prod_{n\geq 1, \text{odd}}^{uf} A \hookrightarrow C_{ubf}\left( X, A \right)\:\text{and} \:\prod_{n\geq 2, \text{even}}^{uf} A \hookrightarrow C_{ubf}\left( Y, A \right),$$
induce two isomorphisms on the $K$-theory. Thus, we have the following exact sequence
$$\cdots \rightarrow K_{\ast}\left(\text{ker}(e_0)\right) \rightarrow K_{\ast}\left(\prod_{n\geq 1, \text{odd}}^{uf} A\right) \oplus K_{\ast}\left(\prod_{n\geq 2, \text{even}}^{uf} A\right) \stackrel{\rho_{\ast}}{\rightarrow} K_{\ast}\left( \prod_{n\geq 1}^{uf} A \right) \rightarrow \cdots.$$
Now, we prove that $\rho_{\ast}$ is an isomorphism which implies $K_{\ast}\left(\text{ker}(e_0)\right)=0$. For $0<\ep<1/4$ and $r>0$, consider 
$$\rho^{\ep, r}_{\ast}: K^{\ep, r}_{\ast}\left( \prod^{uf}_{n\geq 1, \text{odd}} A \right) \oplus K^{\ep, r}_{\ast}\left( \prod^{uf}_{n\geq 2, \text{even}} A \right) \rightarrow K^{\ep, r}_{\ast}\left( \prod^{uf}_{n\geq 1 } A \right),$$
defined by 
$$\rho^{\ep, r}_{\ast}\left( [(a_n)_{n\geq 1, \text{odd}}]_{\ep, r},  [(a'_n)_{n\geq 2, \text{even}}]_{\ep, r}\right)=[(a_1, a_1,a_3,a_3,\cdots)]_{\ep, r}-[(0, a'_2, a'_2,a'_4,a'_4,\cdots)]_{\ep, r}.$$
Obviously, the homomorphism
$$\rho^{\ep, r}_{\prod, \ast}: \left(\prod_{n\geq 1, \text{odd}} K^{\ep, r}_{\ast}(A)\right) \oplus \left(\prod_{n\geq 2, \text{even}} K^{\ep, r}_{\ast}(A)\right) \rightarrow \left(\prod_{n\geq 1} K^{\ep, r}_{\ast}(A)\right),$$
defined by
$$\rho^{\ep, r}_{\prod, \ast}\left( ([a_n]_{\ep, r})_{n\geq 1, \text{odd}}, ([a'_n]_{\ep, r})_{n\geq 2, \text{even}} \right)= ([a_1]_{\ep, r}, [a_1]_{\ep, r}-[a'_2]_{\ep,r}, [a_3]_{\ep, r}-[a'_2]_{\ep, r}, \cdots)$$
is an isomorphism. Thus, the homomorphism $\rho^{\ep, r}_{\ast}$ is an isomorphism for any $0<\ep<1/4$ and $r>0$ by Lemma \ref{Lem-quanK-prod}. Therefore, $\rho_{\ast}$ is an isomorphism by Lemma \ref{quan-K-and-K}.
\end{proof}

Now, we have the following relationship for the $K$-theory of two equivariant localization algebras along one direction.

\begin{proposition}\label{Prop-localong-uniforalong}
	The embedding map
	$$\tau: C^{\ast}_{L, P_{k}(N), u}\left(P_{k}(N)\times G, A\right)^{N\rtimes G} \rightarrow C^{\ast}_{L, P_{k}(N)}\left(P_{k}(N)\times G, A\right)^{N\rtimes G},$$
	induces an isomorphism $\tau_{\ast}$ on the $K$-theory for any $k\geq 0$.
\end{proposition}
\begin{proof}
	Let $P_{k}(N)^{m}$ be the $m$-dimensional skeleton of $P_{k}(N)$. For $m=0$, since 
	$$K_{\ast}\left(C^{\ast}_{L, N, u}\left(N\times G, A\right)^{N\rtimes G}\right) \cong K_{\ast}\left( C_{ubf}([0,\infty), C^{\ast}(G, A\otimes \K(\ell^2(N)))^{G}) \right),$$
	and 
	$$K_{\ast}\left(C^{\ast}_{L, N}\left(N\times G, A\right)^{N\rtimes G}\right) \cong K_{\ast}\left( C_{ub}([0,\infty), C^{\ast}(G, A\otimes \K(\ell^2(N)))^{G}) \right).$$
	Then $\tau_{\ast}$ is an isomorphism for $m=0$ by Lemma \ref{Lem-elementary} and Lemma \ref{Lem-fil-eva}.
    Then by a similar argument of the proof of Proposition \ref{Prop-oneloc-twoloc}, we can complete the proof by induction on $m$.
\end{proof}

\subsection{The proof of the third main theorem}

Define
$$W: \ell^2(G)\otimes \ell^2(G) \rightarrow \ell^2(G)\otimes \ell^2(G), \: \delta_{g_1}\otimes \delta_{g_2} \mapsto \delta_{g_1} \otimes \delta_{g_2g_1},$$
where the actions of $G$ on the first $\ell^2(G)$ of the left-hand side and the right-hand side are the left regular representation, the action of $G$ on the second $\ell^2(G)$ of the left-hand side is the right regular representation and on the second $\ell^2(G)$ of the right-hand side is trivial. Then $W$ is a $G$-equivariant unitary.
Thus $W$ induces an isomorphism:
\begin{equation}\label{equ-uloc-loc}
ad_{W}: C^{\ast}_{L, P_{k}(N), u}(P_{k}(N)\times G, A)^{N\rtimes G} \otimes \K(H\otimes \ell^2(G)) \rightarrow C^{\ast}(G, C^{\ast}_{L}(P_{k}(N), A)^{N})^{G},
\end{equation}
defined by 
$$ad_{W}(u\otimes K)=(W^{\ast}\otimes I)\left( (W\otimes I) u (W^{\ast}\otimes I) \otimes K\right) (W\otimes I),$$
for any $u\in C^{\ast}_{L, P_{k}(N), u}(P_{k}(N)\times G, A)^{N\rtimes G}$ and $K\in \K(H\otimes \ell^2(G))$, where the action of $G$ on $H\otimes \ell^2(G)$ is trivial. Therefore, we have the following lemma.

\begin{lemma}\label{Lem-trans-loc}
	The homomorphism $e_{\ast}$ in (\ref{equ-red-2'}) is injective, surjective and isomorphic, respectively, if and only if the homomorphism
	$$e_{\ast}: \lim_{k\rightarrow \infty}K_{\ast}\left( C^{\ast}_{L, P_{k}(N), u}(P_{k}(N)\times G, A)^{N\rtimes G} \right) \rightarrow \lim_{k\rightarrow \infty}K_{\ast}\left( C^{\ast}(P_{k}(N)\times G, A)^{N\rtimes G} \right)$$
	is injective, surjective and isomorphic, respectively.
\end{lemma}

\begin{remark}
	Notice that the algebra $C^{\ast}_{L, P_{k}(N), u}(P_{k}(N)\times G, A)^{N\rtimes G}$ is not isomorphic to $C^{\ast}(G, C^{\ast}_{L}(P_{k}(N), A)^{N})^{G}$, since for any element $(T_{g_1, g_2})_{g_1, g_2\in G}$ in the algebra $C^{\ast}(G, C^{\ast}_{L}(P_{k}(N), A)^{N})^{G}$, there exits $u_{g_1, g_2}\in C^{\ast}_{L}(P_{k}(N), A)^{N}$ and $K_{g_1, g_2}\in \K(H\otimes \ell^2(G))$ which is independent of $t$ such that $T_{g_1, g_2}(t)=u_{g_1, g_2}(t)\otimes K_{g_1, g_2}$ for any $t\in [0,\infty)$, where the action of $G$ on $\ell^2(G)$ is the right regular representation.
\end{remark}

Combining Proposition \ref{Prop-localong-uniforalong} with Lemma \ref{Lem-trans-loc}, we obtain the following proposition. 

\begin{proposition}\label{Prop-Condition2}
    The homomorphism $e_{\ast}$ in (\ref{equ-red-2'}) is injective, surjective and isomorphic, respectively, if and only if the homomorphism
    \begin{equation}\label{eq-along}
    e_{\ast}: \lim_{k\rightarrow \infty}K_{\ast}\left( C^{\ast}_{L, P_{k}(N)}(P_{k}(N)\times G, A)^{N\rtimes G} \right) \rightarrow \lim_{k\rightarrow \infty}K_{\ast}\left( C^{\ast}(P_{k}(N)\times G, A)^{N\rtimes G} \right)
    \end{equation}
    is injective, surjective and isomorphic, respectively.
\end{proposition}

We introduce the following conjectures for $N\rtimes G$.

\begin{conjecture}\label{partial-conj}
	Let $N\rtimes G$ be an isometric semi-direct product and $A$ be an $(N\rtimes G)$-$C^{\ast}$-algebra.
	\begin{enumerate}
		\item The \textit{strong Novikov conjecture} along $N$ with coefficients in $A$ for $N\rtimes G$ asserts that $e_{\ast}$ in (\ref{eq-along}) is injective.
		\item The \textit{surjective assembly conjecture} along $N$ with coefficients in $A$ for $N\rtimes G$ asserts that $e_{\ast}$ in (\ref{eq-along}) is surjective.
		\item The \textit{Baum--Connes conjecture} along $N$ with coefficient in $A$ for $N\rtimes G$ asserts that $e_{\ast}$ in (\ref{eq-along}) is isomorphic.
	\end{enumerate}
\end{conjecture}

Combining Proposition \ref{Prop-reduction-Conj}, homomorphisms (\ref{equ-red-1}) (\ref{equ-red-2'}) with Proposition \ref{Prop-Condition2}, we obtain the third main theorem of this paper.

\begin{theorem}\label{main-THM}
	Let $N\rtimes G$ be an isometric semi-direct product and $A$ be an $(N\rtimes G)$-$C^{\ast}$-algebra. If
	\begin{enumerate}
		\item \label{Condition1} there exists a non-negative sequence $\{k_{i}\}_{i\in \N}$ with $\lim_{i\rightarrow \infty}k_{i}=\infty$ such that $G$ satisfies SNC, SAC and BCC with coefficients in $C^{\ast}_{L}(P_{k_{i}}(N), A)^{N}$ for all $i\in \N$, respectively;
		\item \label{Condition2} $N\rtimes G$ satisfies SNC along $N$, SAC along $N$ and BCC along $N$ with coefficients in $A$, respectively.
	\end{enumerate}
	Then $N\rtimes G$ satisfies SNC, SAC and BCC with coefficients in $A$, respectively.
\end{theorem}

When $A=\C$, we have the following result.

\begin{corollary}
	Let $N\rtimes G$ be an isometric semi-direct product. If
	\begin{enumerate}
		\item \label{Condition1'} there exists a non-negative sequence $\{k_{i}\}_{i\in \N}$ with $\lim_{i\rightarrow \infty}k_{i}=\infty$ such that $G$ satisfies SNC, SAC and BCC with coefficients in $C^{\ast}_{L}(P_{k_{i}}(N))^{N}$ for all $i\in \N$, respectively;
		\item \label{Condition2'} $N\rtimes G$ satisfies the strong Novikov conjecture along $N$, the surjective assembly conjecture along $N$ and the Baum--Connes conjecture along $N$, respectively.
	\end{enumerate}
	Then $N\rtimes G$ satisfies SNC, SAC and BCC, respectively.
\end{corollary}

We remark that Theorem \ref{main-THM} also holds for rational version of the above conjectures.

When replacing $P_{k}(N)$ by $\widetilde{P_{k,m}}(N)$ in (\ref{eq-along}), we can similarly introduce the \textit{rational analytic Novikov conjecture along $N$ with coefficients}. Thus, we have the following theorem by a similar argument of Theorem \ref{main-THM}.

\begin{theorem}\label{main-THM-rational}
	Let $N\rtimes G$ be an isometric semi-direct product and $A$ be an $(N\rtimes G)$-$C^{\ast}$-algebra. If
	\begin{enumerate}
		\item there exists a non-negative sequence $\{k_{i}\}_{i\in \N}$ with $\lim_{i\rightarrow \infty}k_{i}=\infty$ such that $G$ satisfies the rational analytic Novikov conjecture with coefficients in $C^{\ast}_{L}(P_{k_{i}}(N), A)^{N}$ for all $i\in \N$;
		\item $N\rtimes G$ satisfies the rational analytic Novikov conjecture along $N$ with coefficients in $A$.
	\end{enumerate}
	Then $N\rtimes G$ satisfies the rational analytic Novikov conjecture with coefficients in $A$.
\end{theorem}

In \cite{Kasparov-Yu-NovikovConj}, G. Kasparov and G. Yu introduced a notion of (rational) Property (H) for Banach spaces which contains Hilbert spaces, and they proved that the (rational) strong Novikov conjecture with coefficients holds for any countable discrete group which admits a coarse embedding into a Banach space with (rational) Property (H).

Kasparov and Yu's method in \cite{Kasparov-Yu-NovikovConj} is also practicable for the (rational) strong Novikov conjecture along $N$ with coefficients. Thus, we have the following proposition.

\begin{proposition}\label{Prop-Novikov-along}
	Let $N\rtimes G$ be an isometric semi-direct product and $A$ be an $(N\rtimes G)$-$C^{\ast}$-algebra. If $N$ admits a coarse embedding into a Banach space with (rational) Property (H), then the (rational) strong Novikov conjecture along $N$ with coefficients in $A$ holds for $N\rtimes G$.
\end{proposition}

Combing Proposition \ref{Prop-Novikov-along} with Theorem \ref{main-THM}, we obtain the following corollary which generalizes J. Deng's result in \cite{Deng-Novikov-extensions} for the case of isometric semi-direct products.

\begin{corollary}\label{Cor-Novikov-PropertyH}
	Let $N\rtimes G$ and $A$ be as above. If $N$ admits a coarse embedding into a Banach space with (rational) Property (H) and there exists a non-negative sequence $\{k_{i}\}_{i\in \N}$ with $\lim_{i\rightarrow \infty}k_{i}=\infty$ such that $G$ satisfies the (rational) strong Novikov conjecture with coefficients in $C^{\ast}_{L}(P_{k_{i}}(N), A)^{N}$ for all $i\in \N$, then $N\rtimes G$ satisfies the (rational) strong Novikov conjecture with coefficients in $A$.
\end{corollary}

\section{Applications and examples} \label{Sec-6}

\subsection{Applications to central extensions and finite extensions}

Firstly, we show two applications of Theorem \ref{first-main-thm} to central extensions of groups and group extensions by finite groups.

All a-T-menable groups and hyperbolic groups satisfy the Baum--Connes conjecture with coefficients (cf. \cite{Higson-Kasparov} and \cite{Lafforgue-2012}). Moreover, both a-T-menability and hyperbolicity are preserved under finite index extensions. Thus, we have the following theorem by Theorem \ref{first-main-thm}.

\begin{theorem}\label{Thm-N-special}
	Let $1\rightarrow N \rightarrow \Gamma \rightarrow \Gamma/N \rightarrow 1$ be a group extension and $A$ be a $\Gamma$-$C^{\ast}$-algebra. Assume that $N$ is an a-T-menable group or a hyperbolic group. If $\Gamma/N$ satisfies SNC, SAC and BCC with coefficients in $C_0(\Gamma/N, A)\rtimes_{r}\Gamma$, respectively. Then $\Gamma$ satisfies SNC, SAC and BCC with coefficients in $A$, respectively.
\end{theorem}

Any finite groups and abelian groups are a-T-menable. Thus, we have the following two corollaries which imply that SNC, SAC and BCC with coefficients are closed under extensions by finite groups and central extensions.

\begin{corollary}\label{Cor-finite-extension}
	Let $1\rightarrow N \rightarrow \Gamma \rightarrow \Gamma/N \rightarrow 1$ be a group extension by a finite group $N$ and $A$ be a $\Gamma$-$C^{\ast}$-algebra. If $\Gamma/N$ satisfies SNC, SAC and BCC with coefficients in $C_0(\Gamma/N, A)\rtimes_{r}\Gamma$, respectively. Then $\Gamma$ satisfies SNC, SAC and BCC with coefficients in $A$, respectively.
\end{corollary}

\begin{corollary}\label{Cor-central-extension}
	Let $1\rightarrow N \rightarrow \Gamma \rightarrow \Gamma/N \rightarrow 1$ be a central extension of groups and $A$ be a $\Gamma$-$C^{\ast}$-algebra. If $\Gamma/N$ satisfies SNC, SAC and BCC with coefficients in $C_0(\Gamma/N, A)\rtimes_{r}\Gamma$, respectively. Then $\Gamma$ satisfies SNC, SAC and BCC with coefficients in $A$, respectively.
\end{corollary}

\begin{remark}
	The strong Novikov conjecture with coefficients holds for any groups which admit a coarse embedding into a Hilbert space (cf. \cite{STY-2002}\cite{Yu2000}). Corollary \ref{Cor-central-extension} implies that the strong Novikov conjecture is closed under central extensions, but we do not know if coarse embeddability is preserved by central extensions (cf. \cite[Question 5.2]{AT-NotCoarseEmbed}, \cite{DG-2003}).
\end{remark}

Secondly, we show an application of Theorem \ref{second-main-thm} to group extensions of finite groups. Since each finite group satisfies the rational analytic Novikov conjecture with any coefficients, we have the following theorem by Theorem \ref{second-main-thm}.

\begin{theorem}\label{Thm-rational-Novikov-finite extension}
	Let $1\rightarrow N \rightarrow \Gamma \rightarrow G \rightarrow 1$ be a group extension of a finite group $G$ and $A$ be a $\Gamma$-$C^{\ast}$-algebra satisfying the K\"unneth formula. If $N$ satisfies the rational Baum--Connes conjecture with coefficients in $A$, then $\Gamma$ satisfies the rational analytic Novikov conjecture with coefficients in $A$.
\end{theorem}

When $A=\mathbb{C}$, we have the following corollary.
\begin{corollary}\label{Cor-rational-Novikov-finite extension}
	Let $\Gamma$ be an extension of a finite group $G$ by a group $N$. If $N$ satisfies the rational Baum--Connes conjecture, then $\Gamma$ satisfies the rational analytic Novikov conjecture.
\end{corollary}

\begin{remark}
	It is unknown if the (rational) Baum--Connes conjecture is closed for extensions of finite groups. However, Corollary \ref{Cor-rational-Novikov-finite extension} implies that such group extension satisfies the rational analytic Novikov conjecture. Moreover, in \cite{Meyer-BC-counterex}, R. Meyer constructed an extension $\Gamma$ of a finite cyclic group by a group $N$ satisfying the Baum--Connes conjecture with coefficients in a non-trivial $C^{\ast}$-algebra $A$ such that $\Gamma$ does not satisfy the Baum--Connes conjecture with coefficients in $A$. However, Theorem \ref{Thm-rational-Novikov-finite extension} tell us that such group $\Gamma$ satisfies the rational analytic Novikov conjecture with coefficients in $A$ (please see Example \ref{EX-Meyer}).
\end{remark}

\subsection{An application to direct products}
Next, we study an application of Theorem \ref{main-THM} and Theorem \ref{main-THM-rational} to direct products of groups. The direct product $N\times G$ of $N$ and $G$ is an isometric semi-direct product of $G$ on $N$ with the trivial action.

Let $A$ be an $(N\times G)$-$C^{\ast}$-algebra equipped with the action $\alpha$. Then there exists an action $\alpha^{G}$ of $G$ on $A$ defined by $\alpha^{G}_{g}(a)=\alpha_{(e, g)}(a)$. Thus, we obtain an equivariant Roe algebra $C^{\ast}(G, A)^{G}$. Besides, we have an action $\alpha^{N}_{G}$ of $N$ on $C^{\ast}(G, A)^{G}$ defined by
$$\alpha^{N}_{G, n}\left( (T_{g_1, g_2})_{g_1, g_2\in G} \right)=\left( (\alpha_{(n, e)}\otimes I)(T_{g_1, g_2}) \right)_{g_1, g_2\in G},$$
where $\alpha_{(n, e)}\otimes I$ is an operator on $A\otimes \K(H) \otimes \K(\ell^2(G))$ defined by $(\alpha_{(n, e)}\otimes I)(a\otimes K\otimes Q)=\alpha_{(n, e)}(a)\otimes K\otimes Q$. Therefore, we finally obtain an equivariant Roe algebra $C^{\ast}(P_{k}(N), C^{\ast}(G, A)^{G})^{N}$.

Let $U: H \rightarrow H\otimes H$ be a unitary. Define
$$\phi: C^{\ast}(P_{k}(N)\times G, A)^{N\times G} \rightarrow C^{\ast}(P_{k}(N), C^{\ast}(G, A)^{G})^{N},$$
by
$$\phi(T)_{x_1, x_2}=((U\otimes I) T_{(x_1, g_1),(x_2, g_2)} (U^{\ast}\otimes I))_{g_1, g_2\in G},$$
for any $T\in C^{\ast}(P_{k}(N)\times G, A)^{N\times G}$. Moreover, define
$$\tilde{\phi}: C^{\ast}_{L, P_{k}(N)}(P_{k}(N)\times G, A)^{N\times G} \rightarrow C^{\ast}_{L}(P_{k}(N), C^{\ast}(G, A)^{G})^{N},$$
by 
$$\tilde{\phi}(u)(t)=\phi(u(t)),$$
for any $u\in  C^{\ast}_{L, P_{k}(N)}(P_{k}(N)\times G, A)^{N\times G}$.
Then $\phi$ and $\tilde{\phi}$ are two isomorphisms which are commutative with the evaluation at zero maps. Moreover, by Lemma \ref{Lem-eqRoe-crossprod}, $C^{\ast}(G, A)^{G}$ is $N$-equivariantly isomorphic to $(A\rtimes_{r} G) \otimes \K(H)$ equipped with an action of $N$ by $n\cdot(\sum (a_{g}g)\otimes K)=\sum (\alpha_{(n, e)}(a_g)g)\otimes K$. Thus, we obtain the following lemma.

\begin{lemma}\label{Lem-for-products}
	Let $N$ and $G$ be two countable discrete groups and $A$ be an $(N\times G)$-$C^{\ast}$-algebra. Then $N\times G$ satisfies the strong Novikov conjecture along $N$, the surjective assembly conjecture along $N$ and the Baum--Connes conjecture along $N$ with coefficients in $A$, respectively, if and only if $N$ satisfies SNC, SAC and BCC with coefficients in $A\rtimes_{r} G$, respectively.
\end{lemma}

Therefore, we can simplify Theorem \ref{main-THM} to the following theorem for direct products of groups.

\begin{theorem}\label{Thm-direct-prod}
	Let $N$ and $G$ be two countable discrete groups and $A$ be an $(N\times G)$-$C^{\ast}$-algebra. If
	\begin{enumerate}
		\item \label{Condition1-direct} there exists a non-negative sequence $\{k_{i}\}_{i\in \N}$ with $\lim_{i\rightarrow \infty}k_{i}=\infty$ such that $G$ satisfies SNC, SAC and BCC with coefficients in $C^{\ast}_{L}(P_{k_{i}}(N), A)^{N}$ for all $i\in \N$, respectively;
		\item \label{Condition2-direct} $N$ satisfies SNC, SAC and BCC with coefficients in $A\rtimes_{r} G$, respectively.
	\end{enumerate}
	Then $N\times G$ satisfies SNC, SAC and BCC with coefficients in $A$, respectively.
\end{theorem}

we have the following two corollaries.

\begin{corollary}\label{Cor-direct-prod}
	Let $N$ and $G$ be two countable discrete groups. If $G$ and $N$ satisfies SNC, SAC and BCC with coefficients, respectively. Then $N\times G$ satisfies SNC, SAC and BCC with coefficients, respectively.
\end{corollary}

\begin{corollary}\label{Cor-direct-prod-C}
	Let $N$ and $G$ be two countable discrete groups, $A$ be a $C^{\ast}$-algebra with the trivial $(N\times G)$-action (in particular, $A=\C$). If
	$G$ and $N$ satisfy SNC, SAC and BCC with coefficients in any $C^{\ast}$-algebra with the trivial action, respectively. Then $N\times G$ satisfies SNC, SAC and BCC with coefficients in $A$, respectively.
\end{corollary}

\begin{remark}\label{remark-new}
	Corollary \ref{Cor-direct-prod} and Corollary \ref{Cor-direct-prod-C} imply that the strong Novikov conjecture, the surjective assembly conjecture and the Baum--Connes conjecture with coefficients are closed under direct products. Moreover, Condition (\ref{Condition1-direct}) in Theorem \ref{Thm-direct-prod} for the Baum--Connes conjecture is weaker than the Condition (\ref{main-thm1-1}) in Theorem \ref{main-thm1} by combining Proposition \ref{Prop-loc-twoloc}, Proposition \ref{Prop-oneloc-twoloc}, Isomorphism (\ref{equ-uloc-loc}) and Proposition \ref{Prop-localong-uniforalong} with Proposition \ref{Lem-Key}.
\end{remark}

Similar results also hold for the rational analytic Novikov conjecture by Theorem \ref{main-THM-rational}.

\begin{theorem}\label{Thm-direct-prod-rational}
	Let $N$ and $G$ be two countable discrete groups and $A$ be an $(N\times G)$-$C^{\ast}$-algebra. If
	\begin{enumerate}
		\item there exists a non-negative sequence $\{k_{i}\}_{i\in \N}$ with $\lim_{i\rightarrow \infty}k_{i}=\infty$ such that $G$ satisfies the rational analytic Novikov conjecture with coefficients in $C^{\ast}_{L}(P_{k_{i}}(N), A)^{N}$ for all $i\in \N$;
		\item $N$ satisfies the rational analytic Novikov conjecture with coefficients in $A\rtimes_{r} G$.
	\end{enumerate}
	Then $N\times G$ satisfies the rational analytic Novikov conjecture with coefficients in $A$.
\end{theorem}

we also have the following two corollaries.

\begin{corollary}\label{Cor-direct-prod-rational}
	Let $N$ and $G$ be two countable discrete groups. If $G$ and $N$ satisfies the rational analytic Novikov conjecture with coefficients. Then $N\times G$ satisfies the rational analytic Novikov conjecture with coefficients.
\end{corollary}

\begin{corollary}\label{Cor-direct-prod-rational-C}
	Let $N$ and $G$ be two countable discrete groups, $A$ be a $C^{\ast}$-algebra with the trivial $(N\times G)$-action (in particular, $A=\C$). If $N$ and $G$ satisfy the rational analytic Novikov conjecture with coefficients in any $C^{\ast}$-algebra with the trivial action. Then $N\times G$ satisfies the rational analytic Novikov conjecture with coefficients in $A$.
\end{corollary}

\subsection{Examples for the strong Novikov conjecture}

The strong Novikov conjecture with coefficients holds for groups which admit a coarse embedding into a Hilbert space or more general Banach space with property (H) (cf. \cite{Yu2000}\cite{STY-2002}\cite{Kasparov-Yu-NovikovConj}). Such examples contain hyperbolic groups (\cite{Sela-hyer-embed}), linear groups (\cite{GHW-2005}), subgroups of mapping class groups (\cite{Hamenstadt-2009}\cite{Kida-2008}) and Out($F_n$) as well as Aut($F_n$) (\cite{BGH-Out(Fn)}). On the other hand, there exist some groups that can not be coarsely embedded into a Hilbert space and satisfy the Baum--Connes conjecture with coefficients, such as the wreath product $N_{AT}=\Z/2\Z \wr_{G'} H'$ constructed by G. Arzhantseva and R. Tessera in \cite{AT-NotCoarseEmbed}.

\begin{example}\label{Ex-central-extention}
	Let $1\rightarrow N \rightarrow \Gamma \rightarrow G \rightarrow 1$ be a central extension of groups. If $G$ admits a coarse embedding into a Hilbert space, then $\Gamma$ satisfies the strong Novikov conjecture with coefficients by Corollary \ref{Cor-central-extension}, although it is generally unknown if $\Gamma$ can be coarsely embedded into a Hilbert space.
\end{example}

\begin{example}\label{Ex-not-coarse-embed}
	Let $N_{AT}$ be as above and $G_1, \cdots, G_n$ be groups which admit a coarse embedding into a Hilbert space. Then the direct product $N_{AT}\times \cdots \times N_{AT}\times G_1 \times \cdots \times G_n$ satisfies the strong Novikov conjecture with coefficients by Corollary \ref{Cor-direct-prod}, but it can not be coarsely embedded into a Hilbert space.
\end{example}

\begin{example}\label{Ex-wreath-prod}
	Let $N$, $G$ be two countable discrete groups and $S$ be a finite set equipped with a $G$-action. The \textit{wreath product} $N\wr_{S} G$ is defined to be the semi-direct product $N^S\rtimes G$, where the action of $G$ on $N^S$ is defined by $g\cdot(n_{s})_{s\in S}=(n_{g^{-1}s})_{s\in S}$ for any $g\in G$ and $(n_s)_{s\in S}\in N^S$. If $d$ is a proper $N$-invariant metric on $N$, define
	$$d_{S}\left( (n_s)_{s\in S}, (n'_s)_{s\in S} \right)=\max_{s\in S} d(n_s, n'_s),$$
	for any $(n_s)_{s\in S}, (n'_s)_{s\in S}\in N^S$. Then $d_S$ is a $G$-isometric metric on $N^S$ which implies that $N\wr_{S} G$ is an isometric semi-direct product. Thus if $N$ admits a coarse embedding into a Hilbert space and $G$ satisfies the strong Novikov conjecture with coefficients, then $N\wr_{S} G$ also satisfies this conjecture by Corollary \ref{Cor-Novikov-PropertyH}. For an example, $\text{Out}(F_n)\wr_{S} N_{AT}$ satisfies the strong Novikov conjecture with coefficients, although it can not be coarsely embedded into a Hilbert space.
\end{example}

\subsection{Examples for the rational analytic Novikov conjecture}

\begin{example}\label{Ex-Novikov-notembed}
	In \cite{GWY-Novikov}, S. Gong, J. Wu and G. Yu proved that the rational analytic Novikov conjecture holds for groups acting properly and isometrically on an admissible Hilbert-Hadamard space (see also \cite{GWXY-Novikov}\cite{GWWY-HHspaces-Novikov}). As an example, They showed that the rational analytic Novikov conjecture holds for geometrically discrete subgroups of the group of volume preserving diffeomorphisms of a closed smooth manifold. Actually, Gong, Wu and Yu's result is also true for the rational analytic Novikov conjecture with coefficients. The key reason is that \cite[Proposition 8.8]{GWY-Novikov} holds for any coefficients by a Mayer--Vietoris argument and the K\"unneth formula, just like the proof of Lemma \ref{Lem-torsion-to-free}. Thus, if $G$ is a group acting properly and isometrically on an admissible Hilbert-Hadamard space and $N$ is a group satisfying the rational Baum--Connes conjecture (for an example, $N=N_{AT}=\Z/2\Z \wr_{G'} H'$ constructed by Arzhantseva and Tessera in \cite{AT-NotCoarseEmbed}), then any extension of $G$ by $N$ satisfies the rational analytic Novikov conjecture by Corollary \ref{main-cor2}.
\end{example}

\begin{example}\label{EX-Meyer}
	In \cite{Meyer-BC-counterex}, R. Meyer constructed a counterexample of the Baum--Connes conjecture with certain coefficients for group direct products. Let us recall this counterexample. Let $N$ be a countable discrete group and $A$ be an $N$-$C^{\ast}$-algebra such that the Baum--Connes conjecture with coefficients in $A$ does not hold for $N$ (such examples exist by \cite{HLS-2002}). In addition, let $\Z_{p}$ be the cyclic group of order $p$ for a prime number $p$ and $B=\mathcal{O}_2 \otimes \mathcal{K}(\ell^2(\N))$ be the stabilized Cuntz algebra equipped with a $\Z_{p}$-action. Then Meyer in \cite{Meyer-BC-counterex} proved that $N\times \Z_{p}$ does not satisfy the Baum--Connes conjecture with coefficients in $A\otimes B$, although $N$ satisfies the Baum--Connes conjecture with coefficients in $A\otimes B$. However, by Theorem \ref{Thm-rational-Novikov-finite extension}, $N\times \Z_{p}$ satisfies the rational analytic Novikov conjecture with coefficients in $A\otimes B$ when $A$ satisfies the K\"unneth formula (for examples, $A$ is a commutative $C^{\ast}$-algebra).
\end{example}

\section*{Acknowledgments}
The author would like to thank Ralf Meyer and Shintaro Nishikawa for vital discussions on the Baum-Connes conjecture for extensions. The author would also like to thank Kang Li, Jianchao Wu and Jiawen Zhang for valuable suggestions and comments. This work was supported by NSFC 12271165, 12301154.

\bibliographystyle{plain}
\bibliography{BCMKproducts-ref}

\begin{thebibliography}{10}

\bibitem{AAS2016}
Paolo Antonini, Sara Azzali, and Georges Skandalis.
\newblock Bivariant {$K$}-theory with {$\Bbb{R}/\Bbb{Z}$}-coefficients and rho
  classes of unitary representations.
\newblock {\em J. Funct. Anal.}, 270(1):447--481, 2016.

\bibitem{AAS2020}
Paolo Antonini, Sara Azzali, and Georges Skandalis.
\newblock The {B}aum--{C}onnes conjecture localised at the unit element of a
  discrete group.
\newblock {\em Compos. Math.}, 156(12):2536--2559, 2020.

\bibitem{Arano-Kubota-BCextensions}
Yuki Arano and Yosuke Kubota.
\newblock A categorical perspective on the {A}tiyah--{S}egal completion theorem
  in {KK}-theory.
\newblock {\em J. Noncommut. Geom.}, 12(2):779--821, 2018.

\bibitem{AT-NotCoarseEmbed}
Goulnara Arzhantseva and Romain Tessera.
\newblock Admitting a coarse embedding is not preserved under group extensions.
\newblock {\em Int. Math. Res. Not. IMRN}, (20):6480--6498, 2019.

\bibitem{BCH-1994}
Paul Baum, Alain Connes, and Nigel Higson.
\newblock Classifying space for proper actions and {$K$}-theory of group
  {$C^\ast$}-algebras.
\newblock In {\em {$C^\ast$}-algebras: 1943--1993 ({S}an {A}ntonio, {TX},
  1993)}, volume 167 of {\em Contemp. Math.}, pages 240--291. Amer. Math. Soc.,
  Providence, RI, 1994.

\bibitem{BGH-Out(Fn)}
Mladen Bestvina, Vincent Guirardel, and Camille Horbez.
\newblock Boundary amenability of {${\rm Out}(F_N)$}.
\newblock {\em Ann. Sci. \'Ec. Norm. Sup\'er. (4)}, 55(5):1379--1431, 2022.

\bibitem{CE-Permanence-BC}
J\'er\^ome Chabert and Siegfried Echterhoff.
\newblock Permanence properties of the {B}aum--{C}onnes conjecture.
\newblock {\em Doc. Math.}, 6:127--183, 2001.

\bibitem{CEO-2004}
J\'er\^ome Chabert, Siegfried Echterhoff, and Herv\'{e} Oyono-Oyono.
\newblock Going-down functors, the {K}\"unneth formula, and the
  {B}aum--{C}onnes conjecture.
\newblock {\em Geom. Funct. Anal.}, 14(3):491--528, 2004.

\bibitem{DG-2003}
Marius Dadarlat and Erik Guentner.
\newblock Constructions preserving {H}ilbert space uniform embeddability of
  discrete groups.
\newblock {\em Trans. Amer. Math. Soc.}, 355(8):3253--3275, 2003.

\bibitem{Deng-Novikov-extensions}
Jintao Deng.
\newblock The {N}ovikov conjecture and extensions of coarsely embeddable
  groups.
\newblock {\em J. Noncommut. Geom.}, 16(1):265--310, 2022.

\bibitem{Survey-BCC}
Maria~Paula Gomez~Aparicio, Pierre Julg, and Alain Valette.
\newblock The {B}aum--{C}onnes conjecture: an extended survey.
\newblock In {\em Advances in noncommutative geometry---on the occasion of
  {A}lain {C}onnes' 70th birthday}, pages 127--244. Springer, Cham, [2019]
  \copyright 2019.

\bibitem{GWXY-Novikov}
Sherry Gong, Jianchao Wu, Zhizhang Xie, and Guoliang Yu.
\newblock {The Novikov conjecture, the group of diffeomorphisms and continuous
  fields of Hilbert--Hadamard spaces}.
\newblock 2023. arXiv:2310.01219.

\bibitem{GWY-Novikov}
Sherry Gong, Jianchao Wu, and Guoliang Yu.
\newblock The {N}ovikov conjecture, the group of volume preserving
  diffeomorphisms and {H}ilbert--{H}adamard spaces.
\newblock {\em Geom. Funct. Anal.}, 31(2):206--267, 2021.

\bibitem{Green-78}
Philip Green.
\newblock The local structure of twisted covariance algebras.
\newblock {\em Acta Math.}, 140(3-4):191--250, 1978.

\bibitem{GHW-2005}
Erik Guentner, Nigel Higson, and Shmuel Weinberger.
\newblock The {N}ovikov conjecture for linear groups.
\newblock {\em Publ. Math. Inst. Hautes \'Etudes Sci.}, (101):243--268, 2005.

\bibitem{GWY-2024}
Erik Guentner, Rufus Willett, and Guoliang Yu.
\newblock Dynamical complexity and controlled operator {$K$}-theory.
\newblock {\em Ast\'erisque}, (451):89, 2024.

\bibitem{GWWY-HHspaces-Novikov}
Liang Guo, Qin Wang, Jianchao Wu, and Guoliang Yu.
\newblock {{H}ilbert--{H}adamard spaces and the equivariant coarse {N}ovikov
  conjecture}.
\newblock 2024. arXiv:2411.18538.

\bibitem{Hamenstadt-2009}
Ursula Hamenst\"adt.
\newblock Geometry of the mapping class groups. {I}. {B}oundary amenability.
\newblock {\em Invent. Math.}, 175(3):545--609, 2009.

\bibitem{Higson-Kasparov}
Nigel Higson and Gennadi Kasparov.
\newblock {$E$}-theory and {$KK$}-theory for groups which act properly and
  isometrically on {H}ilbert space.
\newblock {\em Invent. Math.}, 144(1):23--74, 2001.

\bibitem{HLS-2002}
Nigel Higson, Vincent Lafforgue, and Georges Skandalis.
\newblock Counterexamples to the {B}aum--{C}onnes conjecture.
\newblock {\em Geom. Funct. Anal.}, 12(2):330--354, 2002.

\bibitem{Kasparov-equi-KK}
Gennadi Kasparov.
\newblock Equivariant {$KK$}-theory and the {N}ovikov conjecture.
\newblock {\em Invent. Math.}, 91(1):147--201, 1988.

\bibitem{Kasparov-signatures}
Gennadi Kasparov.
\newblock {$K$}-theory, group {$C^*$}-algebras, and higher signatures
  (conspectus).
\newblock In {\em Novikov conjectures, index theorems and rigidity, {V}ol.\ 1
  ({O}berwolfach, 1993)}, volume 226 of {\em London Math. Soc. Lecture Note
  Ser.}, pages 101--146. Cambridge Univ. Press, Cambridge, 1995.

\bibitem{KS-bolic}
Gennadi Kasparov and Georges Skandalis.
\newblock Groups acting properly on ``bolic'' spaces and the {N}ovikov
  conjecture.
\newblock {\em Ann. of Math. (2)}, 158(1):165--206, 2003.

\bibitem{Kasparov-Yu-NovikovConj}
Gennadi Kasparov and Guoliang Yu.
\newblock The {N}ovikov conjecture and geometry of {B}anach spaces.
\newblock {\em Geom. Topol.}, 16(3):1859--1880, 2012.

\bibitem{Kida-2008}
Yoshikata Kida.
\newblock The mapping class group from the viewpoint of measure equivalence
  theory.
\newblock {\em Mem. Amer. Math. Soc.}, 196(916):viii+190, 2008.

\bibitem{Lafforgue-2002}
Vincent Lafforgue.
\newblock {$K$}-th\'eorie bivariante pour les alg\`ebres de {B}anach et
  conjecture de {B}aum--{C}onnes.
\newblock {\em Invent. Math.}, 149(1):1--95, 2002.

\bibitem{Lafforgue-2012}
Vincent Lafforgue.
\newblock La conjecture de {B}aum--{C}onnes \`a{} coefficients pour les groupes
  hyperboliques.
\newblock {\em J. Noncommut. Geom.}, 6(1):1--197, 2012.

\bibitem{Meyer-BC-counterex}
Ralf Meyer.
\newblock The {B}aum--{C}onnes conjecture for extensions.
\newblock {\em M\"unster J. Math.}, 18(1):245--247, 2025.

\bibitem{MN-2006}
Ralf Meyer and Ryszard Nest.
\newblock The {B}aum--{C}onnes conjecture via localisation of categories.
\newblock {\em Topology}, 45(2):209--259, 2006.

\bibitem{MY-BC-hyperbolic}
Igor Mineyev and Guoliang Yu.
\newblock The {B}aum--{C}onnes conjecture for hyperbolic groups.
\newblock {\em Invent. Math.}, 149(1):97--122, 2002.

\bibitem{Nowak-Yu-Book}
Piotr~W. Nowak and Guoliang Yu.
\newblock {\em Large scale geometry}.
\newblock EMS Textbooks in Mathematics. EMS Press, Berlin, second edition,
  [2023] \copyright 2023.

\bibitem{Oyono-BC-extensions}
Herv\'e Oyono-Oyono.
\newblock Baum--{C}onnes conjecture and extensions.
\newblock {\em J. Reine Angew. Math.}, 532:133--149, 2001.

\bibitem{OyonoYu2015}
Herv\'{e} Oyono-Oyono and Guoliang Yu.
\newblock On quantitative operator {$K$}-theory.
\newblock {\em Ann. Inst. Fourier (Grenoble)}, 65(2):605--674, 2015.

\bibitem{OyonoYu2019}
Herv\'{e} Oyono-Oyono and Guoliang Yu.
\newblock Quantitative {$K$}-theory and the {K}\"{u}nneth formula for operator
  algebras.
\newblock {\em J. Funct. Anal.}, 277(7):2003--2091, 2019.

\bibitem{QiaoRoe}
Yu~Qiao and John Roe.
\newblock On the localization algebra of {G}uoliang {Y}u.
\newblock {\em Forum Math.}, 22(4):657--665, 2010.

\bibitem{Roe1988}
John Roe.
\newblock An index theorem on open manifolds. {I}, {II}.
\newblock {\em J. Differential Geom.}, 27(1):87--113, 115--136, 1988.

\bibitem{Rosenberg-IHES}
Jonathan Rosenberg.
\newblock {$C\sp{\ast} $}-algebras, positive scalar curvature, and the
  {N}ovikov conjecture.
\newblock {\em Inst. Hautes \'Etudes Sci. Publ. Math.}, (58):197--212, 1983.

\bibitem{Schick-BCextensions}
Thomas Schick.
\newblock Finite group extensions and the {B}aum--{C}onnes conjecture.
\newblock {\em Geom. Topol.}, 11:1767--1775, 2007.

\bibitem{Schochet-Kunneth}
Claude Schochet.
\newblock Topological methods for {$C\sp{\ast} $}-algebras. {II}. {G}eometric
  resolutions and the {K}\"unneth formula.
\newblock {\em Pacific J. Math.}, 98(2):443--458, 1982.

\bibitem{Sela-hyer-embed}
Zlil Sela.
\newblock Uniform embeddings of hyperbolic groups in {H}ilbert spaces.
\newblock {\em Israel J. Math.}, 80(1-2):171--181, 1992.

\bibitem{STY-2002}
Georges Skandalis, Jean-Louis Tu, and Guoliang. Yu.
\newblock The coarse {B}aum--{C}onnes conjecture and groupoids.
\newblock {\em Topology}, 41(4):807--834, 2002.

\bibitem{WillettYu-Book}
Rufus Willett and Guoliang Yu.
\newblock {\em Higher index theory}, volume 189 of {\em Cambridge Studies in
  Advanced Mathematics}.
\newblock Cambridge University Press, Cambridge, 2020.

\bibitem{Yu-BC-coarse}
Guoliang Yu.
\newblock Baum--{C}onnes conjecture and coarse geometry.
\newblock {\em $K$-Theory}, 9(3):223--231, 1995.

\bibitem{Yu-Localizationalg}
Guoliang Yu.
\newblock Localization algebras and the coarse {B}aum--{C}onnes conjecture.
\newblock {\em $K$-Theory}, 11(4):307--318, 1997.

\bibitem{Yu1998}
Guoliang Yu.
\newblock The {N}ovikov conjecture for groups with finite asymptotic dimension.
\newblock {\em Ann. of Math. (2)}, 147(2):325--355, 1998.

\bibitem{Yu2000}
Guoliang Yu.
\newblock The coarse {B}aum--{C}onnes conjecture for spaces which admit a
  uniform embedding into {H}ilbert space.
\newblock {\em Invent. Math.}, 139(1):201--240, 2000.

\bibitem{Zhang-CBCFC}
Jianguo Zhang.
\newblock The coarse {B}aum--{C}onnes conjecture with filtered coefficients and
  product metric spaces.
\newblock {\em Adv. Math.}, 475:Paper No. 110327, 2025.

\bibitem{Zhang-quan-CBC}
Jianguo Zhang.
\newblock On the quantitative coarse {B}aum-{C}onnes conjecture with
  coefficients.
\newblock {\em Sci. China Math.}, 69(6):1537--1564, 2026.

\end{thebibliography}

\end{document}